\documentclass[12pt]{article}
\usepackage{amsmath,amsthm,amsfonts,amssymb,graphicx,epsfig,latexsym,mathrsfs,mathdots,subcaption}
\usepackage[dvipsnames]{xcolor}
\usepackage{epsf,enumerate} 

\usepackage{tikz}
\usetikzlibrary{positioning,shapes.geometric,decorations.markings}

\newtheorem{thm}{Theorem}
\newtheorem{lemma}[thm]{Lemma}
\newtheorem{cor}[thm]{Corollary}

{}

\newtheorem{remark}[thm]{Remark}

\newtheorem*{definition*}{Definition}
\newtheorem*{remark*}{Remark}

\newtheorem{conjecture}[thm]{Conjecture}
\newtheorem{claim}[thm]{Claim}

\newtheorem{lem}[thm]{Lemma}

\newtheorem{definition}[thm]{Definition}

\usepackage[margin=1in]{geometry}
\usepackage[linktocpage=true]{hyperref}
\usepackage{setspace}

\def\ite#1{\hfill\break${}$\hbox to 50pt {\quad(#1)\hfill}}

\renewcommand{\phi}{\varphi}

\renewcommand{\int}{\text{int}}
\newcommand{\ext}{\text{ext}}
\newcommand{\dist}{\text{dist}}

\tikzset{ invisnode/.style={circle, draw=white, fill=white, inner sep=0.07cm},
triangle/.style={regular polygon, regular polygon sides=3, draw, fill=white, inner sep=0.07cm},
square/.style ={regular polygon, regular polygon sides=4, draw, fill=white, inner sep=0.09cm},
bdot/.style ={inner sep=1pt, minimum size=7pt, circle,draw,fill},
wdot/.style ={inner sep=1pt, minimum size=7pt, circle, fill =white, draw}}

\title{List coloring $C_3$-free planar graphs with a sparse matching of restricted lists}
\author{{Stephen G. Hartke}\thanks{Department of Mathematical and Statistical Sciences, University of Colorado Denver, USA. \texttt{stephen.hartke@ucdenver.edu}}
\and{Yupei Li}\thanks{Department of Mathematics, University of South Carolina, USA. \texttt{yupei@email.sc.edu}}\and{Joseph Pappe}\thanks{Department of Mathematics, University of Virginia, USA. \texttt{yyz8xu@virginia.edu}}\and{Fares Soufan\thanks{Department of Mathematics, University of Nebraska-Lincoln, USA. \texttt{fsoufan2@huskers.unl.edu} }}\and{Lin Tian\thanks{Department of Mathmatical Sciences, Middle Tennessee State University, USA. \texttt{lt5b@mtmail.mtsu.edu}}}\and {Zimu Xiang}\thanks{Department of Mathematics, University of Illinois Urbana-Champaign, USA. \texttt{zimux2@illinois.edu}}}
\date{}

\begin{document}

\maketitle
\begin{abstract}
A graph $G$ is $k$-choosable if it has a proper coloring for every $k$-list assignment. While every $C_3$-free planar graph is $4$-choosable, some of them are not $3$-choosable, as constructed by Voigt. Hu and Zhu conjectured that if $G$ is a $C_3$-free planar graph and $X \subseteq V(G)$ induces a bipartite subgraph, then $G$ has a proper $L$-coloring whenever $|L(x)| = 3$ for $x \in X$ and $|L(v)| = 4$ for $v \in V(G) \setminus X$. As evidence, they proved the conjecture when $X$ is an independent set. We provide further evidence by proving the conjecture when the induced subgraph $G[X]$ is an induced sparse matching.
This is the first result supporting the conjecture in which the set $X$ receiving smaller lists may induce a subgraph with edges.
\end{abstract}

\section{Introduction}
Given a graph $G$, a \textit{proper $k$-coloring} of $G$ is a function $f\colon V(G)\rightarrow \{1,2,\dots,k\}$ such that for every $uv\in E(G)$, we have that $f(u)\neq f(v)$.
By a greedy argument, every graph $G$ has a proper $\Delta(G)+1$ coloring, where $\Delta(G)$ is the maximum degree of $G$.
A general theme of graph coloring problems is to color sparse graphs with as few colors as possible.
The most famous result along this line is the Four Color Theorem by Appel and Haken~\cite{appel_haken}, who showed using computer assistance that every planar graph has a proper $4$-coloring.
For planar graphs without $C_3$, Gr\"otzsch~\cite{grotzsch} showed that every such graph has a proper $3$-coloring.

List coloring is a generalization of proper coloring proposed independently by Vizing~\cite{vizing} and Erd\H{o}s, Rubin and Taylor~\cite{erdos_rubin_taylor}.
Given a graph $G$, a \textit{list assignment} $L:V(G)\rightarrow2^{\mathbb{N}}$ of $G$ is a function that assigns a list $L(v)$ for every $v\in V(G)$.
Given a list assignment $L$ of $G$, a \textit{proper $L$-coloring} of $G$ is a function $f:V(G)\rightarrow \mathbb{N}$ such that for every vertex $v\in V(G)$, $f(v)\in L(v)$ and for every edge $uv\in E(G)$, $f(u)\neq f(v)$.
A \textit{$k$-list assignment} of $G$ is a list assignment $L$ with $|L(v)|=k$ for every $v\in V(G)$.
Note that if $L(v)=\{1,2,\dots,k\}$ for every vertex $v$, then a proper $L$-coloring is a proper $k$-coloring of $G$.
If a graph $G$ has a proper coloring for every $k$-list assignment, then $G$ is \textit{$k$-choosable}.

Although every planar graph has a proper $4$-coloring, there are planar graphs that are not $4$-choosable, witnessed by certain $4$-list assignments, see for example~\cite{mirzakhani,voigt_wirth}.
Thus, it is natural to ask whether all planar graphs are $5$-choosable.
Thomassen~\cite{thomassen_every_1994} answered this question in the affirmative.
Furthermore, it is worth mentioning that the proof yielded a stronger statement: instead of assigning every vertex a list of size $5$, we may assign smaller lists to boundary vertices of a plane graph, which facilitates the inductive proof.

For $C_3$-free planar graphs, Voigt~\cite{voigt_not_1995} constructed a graph that is not $3$-choosable, and it was observed by Kratochv\'il and Tuza~\cite{kratochvil_tuza} that every such graph is $4$-choosable.
To understand how far a $C_3$-free planar graph can be from being $3$-choosable, Hu and Zhu~\cite{hu_list_2020} proposed the following conjecture.

\begin{conjecture}\label{conj:huzhu}
       Suppose that for a $C_3$-free planar graph  $G$, $X$ is a vertex subset such that $G[X]$ is bipartite. Let $L$ be a list assignment such that $|L(x)|=3$ for every $x\in X$ while $|L(v)|=4$ for every $v\in V(G)\setminus X$.
        Then $G$ has a proper $L$-coloring.
\end{conjecture}

They proved the following theorem as evidence for Conjecture~\ref{conj:huzhu}.

\begin{thm}[Hu and Zhu~\cite{hu_list_2020}\footnote{A gap in the proof was found in the published version.  Hu and Zhu posted a revised version at \url{https://arxiv.org/abs/1910.12480}}]
\label{thm:huzhu}
    Suppose that for a $C_3$-free planar graph  $G$, $X$ is a vertex subset such that $G[X]$ is independent. Let $L$ be a list assignment such that $|L(x)|=3$ for every $x\in X$ while $|L(v)|=4$ for every $v\in V(G)\setminus X$.
        Then $G$ has a proper $L$-coloring.
\end{thm}

For vertices $x,y\in V(G)$, denote by $\dist_G(x,y)$ the distance between $x$ and $y$ in $G$.
The main goal of this paper is to provide further evidence for Conjecture~\ref{conj:huzhu}.
In particular, the result is the first evidence in which the vertex subset $X$ receiving smaller lists induces a subgraph with edges.
To formally state our result, we need the following definition.

\begin{definition}
    For a graph $G$, a set $A\subseteq V(G)$ is \textbf{$k$-sparse} if each component of $G[A]$ is isomorphic to $K_1$ or $K_2$, and for disjoint components $A_1,A_2$ of $G[A]$ that are not both isomorphic to $K_1$ and vertices $v_1\in A_1,v_2\in A_2$, we have that $\dist_G(v_1,v_2)\ge k$.
\end{definition}

\begin{thm}\label{thm:main}
    Suppose that for a $C_3$-free planar graph  $G$, $X$ is a vertex subset such that $G[X]$ is $15$-sparse. Let $L$ be a list assignment such that $|L(x)|=3$ for every $x\in X$ while $|L(v)|=4$ for every $v\in V(G)\setminus X$.
        Then $G$ has a proper $L$-coloring.
\end{thm}

    As the definition of $k$-sparsity does not impose a distance condition between two $K_1$ components beyond being independent in $G$,  Theorem~\ref{thm:main} implies Hu and Zhu's Theorem~\ref{thm:huzhu}.
    
    To prove Theorem~\ref{thm:main}, we prove a stronger statement involving assigning smaller lists to the boundary vertices of $G$.  To state this theorem, we need the following definitions.

\begin{definition}
    A \textbf{target} $(G,P,L)$ is a triple such that $G$ is a $C_3$-free planar graph with boundary $B$, $P$ a path consisting of at most five vertices of $B$, and $L$ a list assignment such that the following holds:
    \begin{itemize}
        \item $3\le |L(v)|\le 4$ for $v\in V(G)\setminus V(B)$,
        \item  The set $X_{G,L}=\{x\in V(G)\setminus V(B):|L(x)|=3\}$ is $15$-sparse, 
        \item $2\le |L(v)|\le 4$ for $v\in V(B)\setminus V(P)$,
        \item  $|L(v)|=1$ for $v\in V(P)$, and if $uv\in E(G)\setminus E(P)$ for $u,v \in V(P)$, then $L(u)\neq L(v)$.
    \end{itemize}
\end{definition}

Given two boundary vertices $x,y$, we say that $x$ and $y$ are \textit{consecutive} boundary vertices if $xy\in E(B)$; otherwise $xy\in E(G)\setminus E(B)$ and we may refer to the edge $xy$ as a \textit{chord} of $G$.

\begin{definition}
    Assume that $(G,P,L)$ is a target.
    \begin{itemize}
    \item A vertex $u$ is a \textbf{bad vertex} if either $|P|=5$ and $|N_G(u)\cap P|\ge|L(u)|-1$ or $3\le|P|\le 4$, $|L(u)|=2$ and $u$ is adjacent to an end  vertex of $P$. 
    \item An edge $xy$ is a \textbf{bad edge} if $|L(x)|=|L(y)|=2$.
    \item A $4$-cycle $C=xyzw$ is a \textbf{bad cycle} if $w$ is an interior vertex, $x,y,z$ are boundary vertices,
    $|L(w)|=|L(x)|=|L(z)|=3$ and $|L(y)|=2$. 
    \item A $4$-cycle $C=xyzw$ is a \textbf{worse cycle} if $z,w$ are interior vertices, $x,y$ are boundary vertices, and $|L(w)|=|L(x)|=|L(y)|=|L(z)|=3$.
    \item A $4$-cycle $C=xyzw$ is a \textbf{worst cycle} if $z,w$ are interior vertices, $x,y$ are boundary vertices, $|L(x)|=2$, and $|L(y)|=|L(w)|=|L(z)|=3$.
    \end{itemize}
    
    A target $(G,P,L)$ is \textbf{valid} if there is no bad edge, no bad vertex, and no bad, worse, or worst $4$-cycle.
\end{definition}

\begin{remark} \label{rmk:p2}
    If $P$ consists of at most two precolored vertices and $G$ has a vertex $v\in V(B)$ consecutive to $P$ with $|L(v)|=2$, then we assume that $v$ is precolored and contained in $P$ as well.
    If $P$ consists of three precolored vertices, then we may include at most one such $v\in V(B)$ in $P$, to keep $|P|\le 4$.
\end{remark}

\begin{figure}[!htbp]
\centering
\begin{tikzpicture}[scale = .8]
\node[invisnode] at (-.5,0) (a) {};
\node[bdot, label={below:$p_1$}] at (0,0) (p1) {};
\node[bdot, label={below:$p_2$}] at (1,0) (p2) {};
\node[bdot, label={below:$p_3$}] at (2,0) (p3) {};
\node[bdot, label={below:$p_4$}] at (3,0) (p4) {};
\node[bdot, label={below:$p_5$}] at (4,0) (p5) {};
\node[wdot, label={below:$v$}] at (5,0) (v) {};
\node[triangle, label={right:$u$}] at (1,0.75) (u) {};
\node[square, label={above:$w$}] at (1.5, 2) (w) {};
\node[invisnode] at (5.5,0) (b) {};
\draw[black] (a)--(p1)--(p2)--(p3)--(p4)--(p5)--(v)--(b);
\draw[black] (p1)--(u)--(p3);
\draw[black] (p1)--(w)--(p3);
\draw[black] (w)--(p5);
\node[label={below:(a)}] at (2.5,-.4) {};
\end{tikzpicture}
\hspace{.35cm}
\begin{tikzpicture}[scale = .8]
\node[invisnode] at (-.5,0) (a) {};
\node[wdot, label={below:$x$}] at (0,0) (x) {};
\node[wdot, label={below:$y$}] at (1,0) (y) {};
\node[invisnode] at (1.5,0) (b) {};

\draw[black] (a)--(x)--(y)--(b);
\node[label={below:(b)}] at (.5,-.4) {};
\end{tikzpicture}
\hspace{.35cm}
\begin{tikzpicture}[scale = .8]
\node[invisnode] at (-.5,0) (a) {};
\node[triangle, label={below:$x$}] at (0,0) (x) {};
\node[wdot, label={below:$y$}] at (1,0) (y) {};
\node[triangle, label={below:$z$}] at (2,0) (z) {};
\node[triangle, label={above:$w$}] at (1,1) (w) {};
\node[invisnode] at (2.5,0) (b) {};

\draw[black] (a)--(x)--(y)--(z)--(b);
\draw[black] (x)--(w)--(z);
\node[label={below:(c)}] at (1,-.4) {};
\end{tikzpicture}\hspace{.35cm}
\begin{tikzpicture}[scale = .8]
\node[invisnode] at (-.5,0) (a) {};
\node[triangle, label={below:$x$}] at (0,0) (x) {};
\node[triangle, label={below:$y$}] at (1,0) (y) {};
\node[triangle, label={above:$z$}] at (1,1) (z) {};
\node[triangle, label={above:$w$}] at (0,1) (w) {};
\node[invisnode] at (1.5,0) (b) {};

\draw[black] (a)--(x)--(y)--(b);
\draw[black] (x)--(w)--(z)--(y);
\node[label={below:(d)}] at (.5,-.4) {};
\end{tikzpicture}
\hspace{.35cm}
\begin{tikzpicture}[scale = .8]
\node[invisnode] at (-.5,0) (a) {};
\node[wdot, label={below:$x$}] at (0,0) (x) {};
\node[triangle, label={below:$y$}] at (1,0) (y) {};
\node[triangle, label={above:$z$}] at (1,1) (z) {};
\node[triangle, label={above:$w$}] at (0,1) (w) {};
\node[invisnode] at (1.5,0) (b) {};

\draw[black] (a)--(x)--(y)--(b);
\draw[black] (x)--(w)--(z)--(y);
\node[label={below:(e)}] at (.5,-.4) {};
\end{tikzpicture}
\caption{Illustrations of bad subgraphs. Note that other than vertices in $P$, the boundary vertices $x,y,z$ need not be consecutive on the boundary.}
\end{figure}
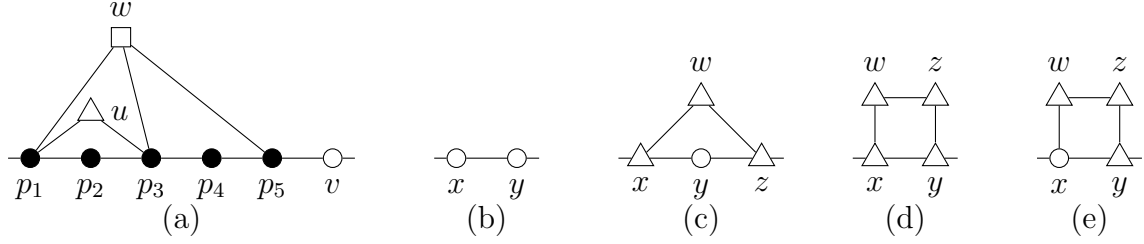

Instead of proving Theorem~\ref{thm:main}, we will prove the following, which implies our main result.

\begin{thm}\label{thm:valid_target}
    If $(G,P,L)$ is a valid target, then $G-E(P)$ is $L$-colorable.
\end{thm}

The rest of the paper is dedicated to proving the theorem above by contradiction, starting with a minimum counterexample. Section~\ref{sec2} focuses on proving some basic properties of the structure of our minimum counterexample and what chords are allowed to exist. In Section~\ref{sec3} we will go over different cases for the structure of our minimum counterexample. In each case, we will ``peel off" (color) certain vertices close to our precolored path and describe the rest of the uncolored graph. When peeling off a vertex, we decrease the size of the lists of its neighboring vertices, since we are fixing its color. The goal of this process is to end up with bad subgraphs, which we call end configurations. Finally, in Section~\ref{sec4}, we will show that different combinations of these end configurations cannot simultaneously occur, implying that it is always possible to peel off vertices from at least one side of the precolored path.

\section{Minimum counterexample} \label{sec2}
Assume that Theorem~\ref{thm:valid_target} fails.
We pick a counterexample $(G,P,L)$ that minimizes $|V(G)|$ and, subject to this, minimizes $\sum_{v\in V(G)}|L(v)|$.

\subsection{Basic properties}

\begin{lem}\label{lem:degree}
    For any vertex $v\in V(G)$, $d(v)\ge |L(v)|$.
\end{lem}

\begin{proof}
    Let $v \in V(G)$ be arbitrary. By the minimality of $G$, $G-v-E(P)$ has a proper $L$-coloring $\phi$. If $d(v) < |L(v)|$, then there exists at least one color in $L(v)$ that we can use for $v$ which was not used for any vertices in its neighborhood. Thus, we can extend $\phi$ to a proper $L$-coloring of $G-E(P)$, a contradiction. This completes the proof.
\end{proof}

\begin{lem}\label{lem:2conn}
    $G$ is $2$-connected.
\end{lem}

\begin{proof}
    Suppose to the contrary that $G$ contains a cut vertex $v$. Let $G_1$ and $G_2$ be the two induced subgraphs of $G$ with $V(G_1) \cap V(G_2) = \{v\}$ and $V(G_1) \cup V(G_2) = V(G)$. If $v \in V(P)$, by the minimality of $G$, there exists an $L$-coloring of $G_1-E(P)$ and $G_2-E(P)$ whose union is an $L$-coloring of $G-E(P)$, a contradiction. Hence $v \notin V(P)$.
    
    Now, without loss of generality, assume that $P \subseteq G_1$. By the minimality of $G$, there exists an $L$-coloring $\phi$ of $G_1 - E(P)$. Let $L_2$ be the restriction of $L$ to the vertices of $G_2$, except with $v$ being precolored $\varphi(v)$, and let $P' = v$. Then $(G_2, P', L_2)$ has no bad vertex since $|P'| = 1$ and no bad edge, bad cycle, worse, or worst cycle since $L_2(u) = L(u)$ for all $u \in V(G_2)-
    \{v\} \subseteq V(G)$. Since $X_{G_2,L_2} \subseteq X_{G,L}$, $(G_2, P', L_2)$ is a valid target and $G_2$ has an $L_2$-coloring $\psi$ that agrees with $\phi$ on $v$. Thus, the union of $\phi$ and $\psi$ is an $L$-coloring of $G-E(P)$, a contradiction. Therefore, $G$ is 2-connected and the proof is complete.
\end{proof}

For a cycle $C$, let $\int[C]$ (resp. $\ext[C]$) be the subgraph induced  by the vertices on $C$ and in the interior (resp. exterior) of $C$.
Let $\int(C)$ be the subgraph induced by the vertices in the interior of $C$.

\begin{lem}\label{lem:sep45}
    $G$ contains no separating $4$-cycle or $5$-cycle.
\end{lem}

\begin{proof}
    Suppose to the contrary that $G$ contains a separating $4$-cycle or $5$-cycle $C$. Set $G_1 = \ext[C]$ and $G_2 = \int[C]$. We choose the cycle $C$ so that $G_2$ has the minimum number of vertices. By the minimality of $G$, $G_1-E(P)$ has an $L$-coloring, call it $\phi$. Let $L_2$ be the restriction of $L$ to $G_2$ with the exception that the vertices on $C$ are precolored by $\phi$. Note that each vertex $v$ in $G_2$ has $|L(v)| \neq 2$, and hence $G_2$ has no bad edge, bad cycle, worse cycle, or worst cycle. If $G_2$ has no bad vertices, then $(G_2,C,L_2)$ is a valid target, and by the minimality of $G$, $G_2 - E(C)$ has an $L_2$-coloring $\psi$. Since $C$ is properly colored by $\phi$, the union of $\phi$ and $\psi$ is an $L$-coloring of $G-E(P)$, a contradiction.

    Thus, $G_2$ has a bad vertex $u$. Moreover, it cannot be adjacent to more than two vertices of $C$ since otherwise it would create a $C_3$. In fact, it is adjacent to two vertices in $C$. By the minimality of $G_2$, $u$ is the only vertex in $\int(C)$. Therefore, $d_G(u) = 2 < |L(u)|$, contradicting Lemma~\ref{lem:degree}. This completes the proof.
\end{proof}

\begin{lem}\label{lem:sep6}
    If $G$ contains a separating $6$-cycle $C=v_1v_2v_3v_4v_5v_6$, then up to relabelling vertices of $C$, there are vertices $x,y\in X_{G,L}$ such that one of the following must hold:
    \begin{itemize}
        \item $\int(C)=\{x\}$ and $N(x)=\{v_1,v_3,v_5\}$ where $v_1$ may or may not equal $y$.
        \item $\int(C)=\{x,y\}$, $xy \in E(X_{G,L})$, $N(x)=\{y,v_1,v_3\}$, and $N(y)=\{x,v_4,v_6\}$.
    \end{itemize}
\end{lem}

\begin{proof} 
    Let $C = v_1v_2v_3v_4v_5v_6$ be a separating $6$-cycle in $G$. By the minimality of $G$, $\ext[C]-E(P)$ has an $L$-coloring $f$.

    We first assume that every vertex in $\int(C)$ is adjacent to at most one vertex in $C$.
    Let $G'=\int[C]-v_6$ and $P'=v_1v_2v_3v_4v_5$, and let $L'$ be the restriction of $L$ to $G'$ with $P'$ precolored by $f$ and $L'(v)=L(v)\setminus\{f(v_6)\}$ for $v\in N(v_6) \cap \int(C)$. As $\int(C)$ has no vertex of list size $2$, $G'$ has no bad edge. By the sparsity condition of $X_{G,L}$, $G'$ has no bad, worse, or worst cycle. Moreover, $G'$ has no bad vertices by construction. Thus, $(G',P',L')$ is a valid target and $G'-E(P')$ has an $L'$-coloring $g$.
    The union of $f$ and $g$ forms an $L$-coloring of $G-E(P)$, a contradiction.

    Therefore, there is some vertex $x \in \int(C)$ that is adjacent to at least two vertices of $C$. Assume that $|\int(C)| \geq 3$ and take $C$ so that $|\int(C)|$ is minimal. Without loss of generality, assume that the only neighbors of $x$ on $C$ are $v_1$ and $v_3$; otherwise, there is a separating $5$-cycle. By the minimality of $C$, the $6$-cycle $C' = v_1xv_3v_4v_5v_6$ contains exactly two internal vertices $y$ and $z$. If $|L(y)|=4$, then either $y$ is adjacent to two consecutive vertices on the boundary of $C'$, creating a $C_3$, or $y$ is adjacent to three vertices of $C'$ creating three $4$-cycles, one of which is a separating $4$-cylce that contains $z$. Thus, by symmetry, it must be the case that $|L(y)| = |L(z)| = 3$. Since $x$ is adjacent to at least one of $y$ or $z$, we have $|L(x)| =4$ by the sparsity of $X_{G,L}$. Therefore, $x$ must be adjacent to both $y$ and $z$, and $yz$ is not an edge to avoid a $C_3$ on $x,y,z$. So that $d(y) \geq |L(y)|$, $y$ must be adjacent to at least two vertices of $C$; however, this creates a separating $4$-cycle containing $z$.
    Hence, $|\int(C)| \leq 2$.     
    
    If $|\int(C)| = 1$, then by Lemma~\ref{lem:degree}, $x$ must have degree $3$ and be adjacent to three nonadjacent vertices on $C$. Moreover, as $x \in X_{G,L}$, either exactly one or none of these three vertices is also in $X_{G,L}$. If $|\int(C)| = 2$, then as before, we may assume that $N(x) = \{v_1, v_3, y\}$ where $y$ is the other interior vertex. As $y$ is contained in the $6$-cycle $v_1xv_3v_4v_5v_6$, $d(y) = 3$ and $y$ is adjacent to $x$, $v_4$, and $v_6$. This completes the proof.
\end{proof}

\subsection{Chords}

\begin{definition}
For $1\le i\le 4$, an $i$-chord is a path $W$ of length $i$ with two ends contained in $B$ and internally disjoint with $B$.
For an $i$-chord $W$ with end vertices $u,v$, if $|L(u)|=a$ and $|L(v)|=b$ (resp. $|L(v)|\ge b$), then we may refer to $W$ as an $(a,b)$-$i$-chord (resp. $(a,b^+)$-$i$-chord).
    
\end{definition}

For every $i$-chord $W$, there are two induced subgraphs $G_{W,1},G_{W,2}$ such that $V(G_{W,1})\cup V(G_{W,2})=V(G)$ and $V(G_{W,1})\cap V(G_{W,2})=V(W)$.
We always assume that, up to relabelling, $P\cap G_{W,1}$ is a path with $|P\cap V(G_{W,1})|$ maximized, so $P\cap B(G_{W,1})$ is not empty.
Assume that $W$ is an $i$-chord with end vertices $u,v$, then for $x\in\{u,v\}$ we may denote by $x'$ the unique consecutive neighbor of $x$ in $B(G)$ that is in $B(G_{W,2})$.
Let $P_W$ be the subpath of $G_{W,2}$ consisting of $(P\cup W)\cap B(G_{W,2})$ and $x'\in B(G_{W,2})$ if $|L(x')|=2$ for $x\in\{u,v\}$.
Let $G'_{W,1}$ be the subgraph of $G$ induced by $V(G_{W,1})\cup P_W$. Note that $G'_{W,1}$ is a valid target as $G$ has no bad, worse, or worst cycles. Therefore, if $G'_{W,1}$ is a proper subgraph of $G$, then $G'_{W,1}-E(P\cap G'_{W,1})$ has an $L$-coloring $\phi$ by the choice of $G$. In this case, we refer to $W$ as a \emph{feasible} chord. Let $L'=L'(W)$ be the restriction of $L$ to $G_{W,2}$ with $L'(x)=\{\phi(x)\}$ if $x\in P_W$ and $L'(x)=L(x)$ otherwise.
Now if $G_{W,1}'-E(P\cap G'_{W,1})$ has an $L$-coloring $\phi$, i.e. $W$ is feasible, then $(G_{W,2},P_W,L')$ should not be a valid target, otherwise $G_{W,2}-E(P_W)$ has an $L'$-coloring $\psi$, and the union of $\phi$ and $\psi$ is an $L$-coloring of $G-E(P)$.

\begin{lem}\label{lem:feasible}
    If $W$ is a feasible chord of $G$, then $|P_W|\ge 5$.
\end{lem}

\begin{proof}
    Suppose to the contrary that there is a feasible chord $W$ with $P_W=q_1q_2\dots q_t$ and $2 \le t\le 4$. Let $W$ be chosen such that the number of vertices in $G_{W,2}$ is minimum. Since $W$ is feasible, $(G_{W,2},P_W,L')$ cannot be a valid target.
    For $x\in G_{W,2}-P_{W}$, if $|L'(x)|\ge 2$, then $|L'(x)|= |L(x)|$, so we may conclude that $(G_{W,2},P_W,L')$ has no bad edge and no bad, worse, or worst cycle.
    If $(G_{W,2},P_W,L')$ has a bad vertex, then there exists a vertex $x$ in $B(G_{W,2})-P_W$ that is adjacent to an end vertex of $P_W$ and has list size $|L'(x)| = |L(x)| =2$. Without loss of generality, we may assume that $x$ is adjacent to $q_1$. Since $x$ is not in $P_W$, it follows that $q_1x$ is a feasible chord which contradicts the minimality of the choice of $W$.
    Thus, $(G_{W,2},P_W,L')$ is a valid target, a contradiction.
\end{proof}

\begin{lem}\label{lem:f-chords}
    There is no 
    \begin{itemize}
        \item $(2,1^+)$-$1$-chord,
         \item $(3,3^+)$-$1$-chord,
        \item $(2,2^+)$-$2$-chord,
        \item $(1,1^+)$-$1$-chord $W$ with $|P\cap B(G_{W,2})| \leq 2$,
        \item $(1,1^+)$-$2$-chord $W=uxv$ with $|P\cap B(G_{W,2})|=1$,
        \item $(1,2)$-$2$-chord $W$ with $|P\cap B(G_{W,2})| \leq 2$.
    \end{itemize}
    
\end{lem}

\begin{proof}
    If such a chord exists, then $W$ is feasible and $|P_W|\le 4$, a contradiction to Lemma~\ref{lem:feasible}.
\end{proof}

\begin{lem}\label{lem:badvertex}
    Assume that $W$ is a feasible chord.
    If $|P_W|=5$, then $(G_{W,2},P_W,L')$ has a bad vertex $x$ with $|L'(x)|=|L(x)|=3$, and $|N_G(x)\cap P_W|\ge 2$. Moreover, if $|P_W\cap P|=3$, then $x$ has exactly one neighbor in $P_W\cap P$ and one neighbor in $P_W-P$.
\end{lem}

\begin{proof}
    Denote $P_W=q_1\dots q_5$. Recall that by definition, if $v\in G_{W,2}-P_W$, then $L'(v)=L(v)$, so $(G_{W,2},P_W,L')$ contains no bad, worse, or worst cycle and no bad edge.
    Since $(G_{W,2},P_W,L')$ is not a valid target, it follows that $(G_{W,2},P_W,L')$ has a bad vertex, say $x$.
    
    Without loss of generality, we assume that there is an edge $q_ix$ for some $1 \le i\le 3$.
    If $|L(x)|=2$, then $q_1\dots q_i x$ contains a feasible chord $W'$ with $|P_{W'}|\le 4$, contradicting Lemma~\ref{lem:feasible}.
    
    If $|L(x)|=4$, then $x$ is adjacent to $q_1,q_3,q_5$. If $x$ is a boundary vertex, we cannot have $P_{W} \cup \{x\}$ be the entire boundary of $G_{W,2}$ while satisfying $d(x) \ge 4$ as this would create a separating $4$-cycle, contradicting Lemma~\ref{lem:sep45}. Thus, one of $q_1x$ or $q_5x$ (if $x$ is on the boundary) or $q_1xq_5$ (if $x$ is in the interior) is a feasible chord. This produces a feasible chord $W'$ containing $x$ with $|P_{W'}|\le 3$, contradicting Lemma~\ref{lem:feasible} again.
    Therefore, $|L'(x)|=|L(x)|=3$, and $|N_G(x)\cap P_W|\ge 2$.

    Note that $x$ is not a bad vertex in $(G,P,L)$, so it has at most one neighbor in $P$. If $|P_W\cap P|=3$ and $|P_{W}| = 5$, then the vertices in $P_{W}-P$ are adjacent. Thus, in this case $x$ must have exactly one neighbor in $P_W\cap P$ and exactly one neighbor in $P_W-P$. This completes the proof of the lemma.
\end{proof}

\begin{lem}\label{lem:233chord}
    There is no 
    \begin{itemize}
        \item feasible $(2,3^+)$-$3$-chord $uv\tilde{v}w$ with $|L(v)| = |L(\tilde{v})|=3$
        \item feasible $(1,3)$-$3$-chord $uv\tilde{v}w$ with $|L(v)| = |L(\tilde{v})|=3$ and $|P\cap B(G_{W,2})| = 1$
    \end{itemize}
\end{lem}

\begin{proof}
    Suppose to the contrary that $W=uv\tilde{v}w$ is a feasible $(2,3^+)$-$3$-chord or $(1,3)$-$3$-chord with $|L(v)| = |L(\tilde{v})|=3$ and $|P\cap B(G_{W,2})| \leq 1$. We choose $W$ such that $|V(G_{W,2})|$ is minimum. Note that $v\tilde{v}$ is an induced matching in $X_{G,L}$.
    By Lemma~\ref{lem:feasible}, we have $|P_W|=5$ and $P_W=uv\tilde{v}ww'$ such that $|L(w')|=2$.
    Since $W$ is feasible, $(G_{W,2},P_W,L)$ cannot be a valid target. By Lemma~\ref{lem:badvertex}, there is some bad vertex $y$ in $(G_{W,2},P_W,L)$ with $|L(y)|=3$ and $|N_G(y)\cap P_W|\geq 2$. Since $v\tilde{v}\in E(X_{G,L})$, $y$ must be a boundary vertex.
    
    If $y$ is adjacent to $v$, then $W' = uvy$ is either a $(2,3^+)$-$2$-chord or a $(1,3)$-$2$-chord with $|P\cap B(G_{W',2})| = 1$ contradicting Lemma~\ref{lem:f-chords}.
    If $y$ is adjacent to $w$, then $G$ has a $(3,3^+)$-$1$-chord, a contradiction. Thus, $y$ is adjacent to at least one of $u$ and $w'$, and the edges $uy$ and $uw'$, if they exist, must lie on the boundary by Lemma~\ref{lem:f-chords}.

    If $|L(w)| = 3$ and $y$ is adjacent to $\tilde{v}$ and $w'$, then $ww'y\tilde{v}$ forms a bad cycle. If $|L(w)| = 4$ and $y$ is adjacent to $\tilde{v}$ and $w'$, then $|L(u)| = 2$ by assumption and the chord $uv\tilde{v}y$ is a feasible $(2,3)$-$3$-chord to avoid having a worst cycle. However, this contradicts our choice of $W$. Assume that $y$ is adjacent to $\tilde{v}$ and $u$ with $uy$ being an edge of $B(G)$. We consider the chord $W' = w\tilde{v}y$. Since $y$ is not adjacent to $w'$, $W'$ is feasible, and $P_{W'} = w'w\tilde{v}yy'$ with $|L(y')| = 2$. By Lemma~\ref{lem:badvertex} and the distance condition on $X_{G,L}$, there exists a boundary vertex $z$ with $|L(z)|=3$ and $|N_G(z)\cap P_{W'}|\ge 2$. In order to respect the minimality of $W$ while avoiding the creation of a chord in Lemma~\ref{lem:f-chords} or a bad cycle, $w'zy'$ must form a path in $B(G)$ and $z$ is not adjacent to $\tilde{v}$. However, this implies $w'w\tilde{v}yy'z$ is a separating $6$-cycle in order for $d(z)\ge 3$. As before, Lemma~\ref{lem:sep6} then implies there exists some vertex in the interior with list size $3$ which contradicts the distance condition on $X_{G,L}$.

    Hence, $y$ must be adjacent to $u$ and $w'$ and not adjacent to $\tilde{v}$. Then $P_W$ along with $y$ forms a separating $6$-cycle as $d(y)\ge 3$. By Lemma~\ref{lem:sep6}, there exists some vertex inside the cycle with list size $3$, contradicting the distance condition on $X_{G,L}$.
    
    Therefore, there is no bad vertex $y$ contradicting Lemma~\ref{lem:badvertex} and the proof is complete.
\end{proof}

\begin{lem}\label{lem:332chord}
    There is no feasible $(3,3^+)$-$2$-chord $uvw$ with $|L(v)|=3$.
\end{lem}

\begin{proof}
    Suppose to the contrary that $W=uvw$ is a feasible $(3,3^+)$-$2$-chord with $|L(v)|=3$ such that $|V(G_{W,2})|$ is minimal.
    By Lemma~\ref{lem:feasible}, $|P_W|=5$ and we may denote $P_W=u'uvww'$ where $|L(u')|=|L(w')|=2$.
    Since $W$ is feasible, $(G_{W,2},P_W,L')$ cannot be a valid target.
    By Lemma~\ref{lem:badvertex}, there is some bad vertex $x$ in $(G_{W,2},P_W,L')$ with $|L(x)|=3$ and $|N_G(x)\cap P_W|=2$. Assume that $x$ is an interior vertex. If $x$ is adjacent to $v$, then $x$ is adjacent to either $u'$ or $w'$. However, in either case, this creates a feasible $(2,3^+)$-$3$-chord contradicting Lemma~\ref{lem:233chord}. Thus, $xv \notin E(G)$. Since $G$ has no $(2,2^{+})$-$2$-chord, $x$ must be adjacent to $u$ and $w$, which contradicts the minimality of our choice of $W$. Therefore, $x$ is a boundary vertex of $G_{W,2}$.
    
    If $x$ is adjacent to $v$, then $x$ is adjacent to one of $u',w'$, which creates a $(2,3)$-$1$-chord, a bad cycle, or a feasible $(3,3^+)$-$2$-chord with smaller $|V(G_{W,2})|$. If $x$ is adjacent to $u$ (resp.~$w$), then $ux$ (resp.~$wx$) is a $(3,3^+)$-$1$-chord, contradicting Lemma~\ref{lem:f-chords}.
    
    Hence, $x$ is adjacent to $u'$ and $w'$ on the boundary of $B(G_{W,2})$, but is not adjacent to the vertex $v$. Thus, $x$ along with $P_W$ forms a separating $6$-cycle in order to meet the degree requirement of $x$. However, this either creates a bad cycle or contradicts the definition of $X_{G,L}$. Therefore, there is no feasible $(3,3^+)$-$2$-chord $uvw$ with $|L(v)|=3$. This completes the proof of the lemma.
\end{proof}

\begin{lem}\label{lem:132chord}
    There is no feasible $(1,3)$-$2$-chord $uvw$ with $|P \cap B(G_{W,2})| \leq 2$ and $|L(v)|=3$.
\end{lem}

\begin{proof}
Suppose to the contrary that there is a feasible $(1,3)$-$2$-chord $W = uvw$ with $|P \cap B(G_{W,2})| \leq 2$ and $|L(v)|=3$. Among all such chords, we choose the one such that $|V(G_{W,2})|$ is minimum. By Lemma~\ref{lem:f-chords}, $|P \cap B(G_{W,2})| = 2$. Assume $p_1, p_2 \in G_{W,2}$ so that $u=p_2$. By Lemma \ref{lem:feasible} and Remark \ref{rmk:p2}, we have $P_W=p_1p_2vww'$ where $w' \in N(w)$ with $|L(w')|=2$. By Lemma \ref{lem:badvertex}, $(G_{W,2},P_W, L')$ has a bad vertex $x$ with $|L'(x)|=|L(x)|=3$ and $|N(x) \cap P_W| \geq 2$.

We claim that $x$ is adjacent to $v$. Assume $x$ is not adjacent to $v$ and is an interior vertex. Then $x$ is not adjacent to $w'$ as $W' = w'xp_i$ with $i = 1,2$ is a feasible chord with $|P_{W'}| \leq 4$. Moreover, $x$ is not adjacent to $w$ as the chord $wxp_i$ with $i = 1,2$ would contradict our minimal choice of $W$. Thus, if $x$ is not adjacent to $v$, then $x$ must be a boundary vertex. Observe that $x$ cannot be adjacent to $w$ as $xw$ would be a $(3,3)$-$1$-chord, contradicting Lemma~\ref{lem:f-chords}. Similarly, if $x$ is adjacent to $p_2$ then $W' = p_2x$ is a $(1,3)$-$1$-chord with $|P \cap B(G_{W',2})| \leq 2$, contradicting Lemma~\ref{lem:f-chords}. If $x$ is adjacent to $p_1$ and $w'$, then $w'xp_1$ is a subpath of $B(G_{W,2})$ in order to avoid creating the chords $w'x$ and $p_1x$. Since we assume $x$ is not adjacent to $v$, $P_W \cup \{x\}$ forms a separating $6$-cycle to meet the degree requirements of $x$. To avoid creating a bad cycle while still satisfying Lemma~\ref{lem:sep6}, this separating $6$-cycle has exactly two interior vertices $y$ and $z$ with $|L(y)| = |L(z)| = 3$ such that $N_{G}(y) = \{x, p_2, z\}$ and  $N_{G}(z) = \{v, w', y\}$. However, the path $vzy$ contradicts the definition of $X_{G,L}$. Thus, $x$ is adjacent to $v$. 

If $x$ is a boundary vertex and adjacent to $v$, then $x$ cannot be adjacent to $w'$ as either $w'x$ would be a $(2,3)$-$1$-chord or $ww'xv$ would be a bad cycle. However, $x$ not being adjacent to $w'$ implies $xvw$ is a feasible $(3,3)$-$2$-chord with $|L(v)|=3$, contradicting Lemma~\ref{lem:332chord}. Hence, $x$ is an interior vertex adjacent to $v$. By Lemma~\ref{lem:233chord}, $x$ cannot be an interior vertex adjacent to both $v$ and $p_1$.

Assume that $x$ is an interior vertex adjacent to $v$ and $w'$ and consider the $(1,2)$-$3$-chord $W'=p_2vxw'$. It is easy to see that $W'$ is feasible and $P_{W'}=p_1p_2vxw'$. Let $L''=L''(W')$ be the restriction of $L'$ to $G_{W',2}$. By Lemma \ref{lem:badvertex}, $(G_{W',2}, P_{W'}, L'')$ contains a bad vertex $y$ with $|L''(y)|=|L(y)|=3$ and $|N(y) \cap P_{W'}| \geq 2$. By the distance condition on $X_{G,L}$, $y$ is in $B(G_{W',2})$. Note that $y$ is not adjacent to $x$ or $v$; otherwise, $w'xy$ is a $(2,3)$-$2$-chord or $wvy$ is a feasible $(3,3)$-$2$-chord with $|L(v)| = 3$, respectively. If $y$ is adjacent to $p_2$, then $W'_1=yp_2$ is a $(1,3)$-$1$-chord with $|P \cap B(G_{W'_1,2})| \leq 2$, contradicting Lemma \ref{lem:f-chords}. Hence, $y$ is adjacent to $p_1$ and $w'$. If $y$ and $w'$ are not consecutive vertices on $B(G)$, then $w'y$ is a $(2,3)$-$1$-chord, contradicting Lemma \ref{lem:f-chords}. If $y$ and $p_1$ are not consecutive vertices on $B(G)$, then $W_2'=p_1 y$ is a $(1,3)$-$1$-chord with $|P \cap B(G_{W'_2,2})| \leq 2$, contradicting Lemma \ref{lem:f-chords}. Thus, $p_1, y, w'$ are three consecutive vertices on $B(G)$. Since $y$ is not adjacent to $v$ or $x$, the cycle $C=p_1 p_2 v x w' y$ must be a separating $6$-cycle in order to meet the degree requirements of $y$. By Lemma \ref{lem:sep6}, the interior vertices of this separating $6$-cycle have list size $3$, which contradicts the distance condition on $X_{G,L}$. Therefore, $(G_{W,2}, P_{W}, L')$ has no bad vertex contradicting Lemma~\ref{lem:badvertex}. This completes the proof.
\end{proof}

\begin{lem}\label{lem:122chord1} 
    There is no $(1,2)$-$2$-chord $uvw$ with $|L(v)| = 3$.
\end{lem}

\begin{proof}
Suppose to the contrary that $W=uvw$ is a $(1,2)$-$2$-chord with $|L(v)|=3$. By Lemma~\ref{lem:f-chords}, $|P \cap B(G_{W,2})|=|P \cap B(G_{W,1})|\geq 2$. Since $|L(w)| = 2$ and $|P| \geq 3$, taking either side to be $G_{W,1}$ will make $W$ a feasible chord. In order for $|P_{W}| = 5$, we must have $|P| =5$ with $|P\cap G_{W,1}| = |P\cap G_{W,2}| = 3$. We always relabel $P$ such that $p_1,p_2,p_3$ are in $B(G_{W,2})$ with $u = p_3$ and $P_{W} = p_1p_2p_3vw$.

First assume that $v$ is adjacent to an interior vertex $\tilde{v} \in X_{G,L}$. Let $G_{W,1}$ be the side of $G$ containing $\tilde{v}$, and we may assume that $G_{W,2}$ has no internal vertex with list size $3$ that is adjacent to $P_W$ by the distance condition on $X_{G,L}$.

By Lemma~\ref{lem:badvertex}, there exists a vertex $x$ such that $|L(x)| = 3$ and $|N_{G}(x) \cap P_W| \geq 2$. Moreover, by our choice of $G_{W,2}$, this vertex $x$ lies on the boundary of $G_{W,2}$. Observe that $x$ cannot be adjacent to $v$ nor $p_2$ as otherwise $wvx$ is a $(2,3)$-$2$-chord or $p_2x$ is a $(1,3)$-$1$-chord with $|P\cap B(G_{W,2})| = 2$, respectively. Since $G$ has no bad vertices and $|N_{G}(x) \cap P_W| \geq 2$, $x$ must be adjacent to $w$. Furthermore, to avoid having a $(2,3)$-$1$-chord, the edge $wx$ must lie in $B(G_{W,2})$. If $x$ is also adjacent to $p_1$, then $p_1x$ must similarly lie in $B(G_{W,2})$; otherwise, $p_1x$ is a $(1,3)$-$1$-chord with $|P\cap B(G_{W,2})| = 1$ contradicting Lemma~\ref{lem:f-chords}. Observe that $P_{W} \cup \{x\}$ would then be a separating $6$-cycle which by Lemma~\ref{lem:sep6} would have an internal vertex with list size $3$ adjacent to $P_W$, which contradicts our choice of $G_{W,2}$. Thus, $x$ is adjacent to $w$ and $p_3$ with $wx$ an edge of $B(G_{W,2})$. 

Since $x$ is not adjacent to $p_1$, the chord $W' = p_3x$ is a feasible chord. In order for $|P_{W'}| = 5$, $x$ must be adjacent to a vertex $x'$ in $B(G_{W',2})$ with $|L(x')| = 2$ and $P_{W'} = p_1p_2p_3xx'$. By Lemma~\ref{lem:badvertex}, $G_{W',2}$ contains a vertex $y$ such that $|L(y)| = 3$ and $|N_{G}(y) \cap P_W| \geq 2$. By choice of $G_{W,2}$, $y$ must be a boundary vertex. A similar argument as above shows that $y$ must be adjacent to $x'$ and $p_3$ with $x'y$ being an edge of $B(G_{W',2})$. The chord $W'' = p_3y$ is once again feasible and $y$ must be adjacent to a vertex $y'$ in $B(G_{W'',2})$ with $|L(y')| = 2$. As this process is infinite, but $G$ is finite, we arrive at a contradiction.

Therefore, we may assume from here on out that $G$ has no $(1,2)$-$2$-chord with interior vertex having list size $3$ that is adjacent to another vertex with list size $3$. We now take $W$ to be a $(1,2)$-$2$-chord $W = uvw$ with $|L(v)| = 3$ such that $|V(G_{W,2})|$ is minimal. Recall that we relabel $P$ such that $p_1,p_2,p_3$ are in $B(G_{W,2})$ with $u = p_3$ and $P_{W} = p_1p_2p_3vw$.

As $W$ is feasible, there exists a vertex $x$ such that $|L(x)| = 3$ and $|N_{G}(x) \cap P_W| \geq 2$ by Lemma~\ref{lem:badvertex}. Note that $x$ cannot be an interior vertex as $\tilde{W} = wxp_i$ where $i = 1,2,3$ either contradicts the minimality of $W$ or is a feasible chord with $|P_{\tilde{W}}| \leq 4$. Thus, $x$ is a boundary vertex. This bad vertex cannot be adjacent to either $v$ or $p_2$ by Lemma~\ref{lem:f-chords}. If $x$ is adjacent to $w$ and $p_1$, then $wxp_1$ forms a path in $B(G)$ and $P_{W} \cup \{ x\}$ forms a separating $6$-cycle. However, the two scenarios in Lemma~\ref{lem:sep6} would either force $v$ to be adjacent to a vertex in $X_{G,L}$ or to be too close to a $K_2$ in $X_{G,L}$. Therefore, $x$ is adjacent to $w$ and $p_3$ with the edge $wx$ lying on the boundary of $G$.

Consider the chord $W' = p_3x$ which is feasible as $x$ is not adjacent to $p_1$. In order for $|P_{W'}| = 5$, $x$ must be adjacent to a vertex $x'$ in $B(G_{W',2})$ with $|L(x')| = 2$ and $P_{W'} = p_1p_2p_3xx'$. By Lemma~\ref{lem:badvertex}, $G_{W',2}$ contains a vertex $y$ such that $|L(y)| = 3$ and $|N_{G}(y) \cap P_W| \geq 2$. If $y$ is an interior vertex, it cannot be adjacent to $x$ by Lemma~\ref{lem:132chord}. Thus, $y$ must be adjacent to $x'$ and $p_3$; however, this contradicts the minimality of $W$. Hence, $y$ must be a boundary vertex. A similar argument as above shows that $y$ must be adjacent to $x'$ and $p_3$ with $x'y$ being an edge of $B(G_{W',2})$. The chord $W'' = p_3y$ is once again feasible and $y$ must be adjacent to a vertex $y'$ in $B(G_{W'',2})$ with $|L(y')| = 2$. As this process is infinite, but $G$ is finite, we arrive at a contradiction. This completes the proof of the lemma.
\end{proof}

We will use the following corollary of the Alon--Tarsi Theorem~\cite{Alon1992} for list colorings of bipartite graphs.

\begin{cor}[Alon--Tarsi~\cite{Alon1992}]
    Let $G$ be a bipartite graph with a list assignment $L$. If $G$ has an orientation $D$ such that $d_D^+(v) < |L(v)|$ for every vertex $v \in V(G)$, then $G$ is $L$-colorable.
\end{cor}

\begin{lem}\label{lem:122chord2} 
    There is no $(1,2)$-$2$-chord $uvw$ with $|L(v)| = 4$.
\end{lem}

\begin{proof}
    Suppose to the contrary that $W=uvw$ is a $(1,2)$-$2$-chord with $|L(v)|=4$. By Lemma~\ref{lem:f-chords}, $3 \leq |P \cap B(G_{W,2})| \leq |P \cap B(G_{W,1})|$, and therefore, $|P| =5$ with $|P\cap G_{W,1}| = |P\cap G_{W,2}| = 3$. As $|L(w)| = 2$ and each side contains three vertices of $P$, taking either side to be $G_{W,1}$ makes $W$ a feasible chord. By the definition of $X_{G,L}$, only one side of $W$ can have adjacent internal vertices $r,s \in X_{G,L}$ such that $r$ is adjacent to $P\cup W$. We may set $G_{W,1}$ to be the side having such internal vertices if they exist, and we may assume that if $t \in X_{G,L}\cap G_{W,2}$ is adjacent to $P_{W}$, then $t$ is not adjacent to any other vertex in $X_{G,L}$. We relabel $P$ such that $p_1,p_2,p_3$ are in $B(G_{W,2})$ with $u = p_3$ and $P_{W} = p_1p_2p_3vw$. By Lemma~\ref{lem:badvertex}, $G_{W,2}$ contains a bad vertex $x$ such that $|L(x)| = 3$ and $x$ has exactly one neighbor in each of the sets $P_W \cap P$ and $P_{W}-P$. We break into cases based on whether or not $G_{W,2}$ contains a bad vertex in its interior.
    
    \textbf{Case 1:} Assume that $G_{W,2}$ does not contain a bad vertex in its interior. This implies that $x$ is a boundary vertex. By Lemma~\ref{lem:f-chords}, $x$ is not adjacent to $v$ nor to $p_2$. If $x$ is adjacent to both $w$ and $p_1$, then $wxp_1$ forms a path on the boundary and $P_{W} \cup \{x\}$ forms a separating $6$-cycle. By our choice of $G_{W,2}$, the interior of this separating $6$-cycle contains a single vertex adjacent to $v$, $x$, and $p_2$. However, this contradicts $G_{W,2}$ not containing a bad vertex in its interior. Thus, $x$ is adjacent to $w$ and $p_3$ with $wx$ being a boundary edge. 
    
    Consider the chord $W' = xp_3$ which is feasible as $x$ is not adjacent to $p_1$. In order for $|P_{W'}| = 5$, $x$ must be adjacent to a vertex $x'\not= w$ in $B(G)$ with $|L(x')| = 2$.  By Lemma~\ref{lem:badvertex}, $G_{W',2}$ contains a bad vertex $y$ such that $|L(y)| = 3$ and $y$ has exactly one neighbor in each of the sets $P_{W'}\cap P$ and $P_{W'}-P$. Lemma~\ref{lem:132chord} and Lemma~\ref{lem:122chord1} imply $y$ is not an interior vertex and must therefore lie on the boundary. By Lemma~\ref{lem:f-chords}, $y$ is not adjacent to $x$ nor to $p_2$. If $y$ is adjacent to $x'$ and $p_1$, then $P_{W'} \cup \{y\}$ is a separating $6$-cycle. However, by Lemma~\ref{lem:sep6}, this either creates a bad cycle or contradicts our choice of $G_{W,2}$. Therefore, $y$ is adjacent to $x'$ and $p_3$ with $x'y$ being a boundary edge. Since $y$ is not adjacent to $p_1$, the chord $W'' = yp_3$ is feasible. As this process is infinite, but $G$ is finite, we arrive at a contradiction.

    \textbf{Case 2:} Assume that $G_{W,2}$ contains a bad vertex in its interior. We may therefore assume $x$ is in the interior. By Lemma~\ref{lem:122chord1}, $x$ is not adjacent to $w$ and must therefore be adjacent to $v$. Note that $x$ cannot be adjacent to both $v$ and $p_1$ as this would create a chord $\tilde{W} = wvxp_1$ with $|P_{\tilde{W}}| = 4$. Thus, $x$ is adjacent to $v$ and $p_2$, and we may consider the feasible chord $W' = wvxp_2$. 

    By Lemma~\ref{lem:badvertex}, $G_{W',2}$ contains a vertex $y$ such that $|L(y)| = 3$ and $|N_{G}(y) \cap P_{W'}| \geq 2$. If $y$ is on the interior, it is not adjacent to $x$ by choice of $G_{W,2}$ and is not adjacent to $w$ by Lemma~\ref{lem:122chord1}. This would imply that $y$ is adjacent to $v$ and one of $p_1$ or $p_2$; however, this creates a separating $4$- or $5$-cycle containing the vertex $x$. Therefore, $y$ is a boundary vertex.
    
    We claim that $y$ is adjacent to $x$. If not, $y$ is adjacent to $w$ and $p_1$ by Lemma~\ref{lem:f-chords} and $P_{W'}\cup \{y\}$ forms a separating $6$-cycle. However, this either contradicts our choice of $G_{W,2}$ or creates a $5$-cycle that contains the vertex $x$. Thus, $y$ is adjacent to $x$. If $y$ is not adjacent to $p_1$ on the boundary, then $p_2xy$ forms a feasible $(1,3)$-$2$-chord that contradicts Lemma~\ref{lem:132chord}. Therefore, $y$ is adjacent to $x$ and $p_1$.

    Assume that $y$ is adjacent to $x$ and $p_1$ but is not adjacent to $w$. We consider the chord $W'' = wvxy$. Since we assume $y$ is not adjacent to $w$, the chord $W''$ is feasible, which implies that the vertex $y$ is adjacent to some $y' \in B(G)$ such that $|L(y')| = 2$. Moreover, there exists some vertex $z$ in $G_{W'',2}$ such that $|L(z)| = 3$ and $|N_{G}(z) \cap P_{W''}| \geq 2$. 
    
    If $z$ is an interior vertex, then by our choice of $G_{W,2}$, $z$ is not adjacent to $x$. In order to avoid creating a $(2,2^+)$-$2$-chord, $z$ is not adjacent to $w$. If $z$ is adjacent to $v$ and $y'$, the chord $wvzy'$ would be a feasible chord that contradicts Lemma~\ref{lem:feasible}. Thus, $z$ is adjacent to $v$ and $y$. Consider now the feasible chord $W''' = wvzy$. As before, there exists some vertex $\alpha$ in $G_{W''',2}$ such that $|L(\alpha)| = 3$ and $|N_{G}(\alpha) \cap P_{W'''}| \geq 2$. By the same logic applied to the vertex $z$, if $\alpha$ is an interior vertex, then it must be adjacent to $v$ and $y$. However, this creates a separating $4$-cycle $vxy\alpha$ that contains the vertex $z$, a contradiction. Therefore, $\alpha$ lies on the boundary. By Lemma~\ref{lem:f-chords}, $\alpha$ is not adjacent to $v$ nor adjacent to $y$. Furthermore, $\alpha$ is not adjacent to $z$, because either $zyy'\alpha$ is a bad cycle or $yz\alpha$ is a feasible $(3,3)$-$2$-chord that contradicts Lemma~\ref{lem:332chord}. This implies $\alpha$ is adjacent to $w$ and $y'$ and $P_{W'''}\cup \{\alpha\}$ forms a separating $6$-cycle. However, this either creates a bad cycle or contradicts our choice of $G_{W,2}$. Hence, $z$ is a boundary, but this cannot happen by the same logic applied when proving $\alpha$ could not be a boundary vertex.

     Therefore, $y$ is adjacent to $x$, $p_1$, and $w$, and $V(G_{W,2}) = V(P_{W})\cup \{x, y\}$. Since $x \in X_{G,L}$ and is adjacent to $P_{W}$, $G_{W,1} \cap X_{G,L}$ has no adjacent vertices that are also adjacent to a vertex in $P\cup W$. We may then apply the same logic to $G_{W,1}$ such that $G_{W,1}$ is isomorphic to $G_{W,2}$ with the same list sizes, see Figure~\ref{fig:122chord}. It remains to show that $G-E(P)$ is list-colorable. Since $G-E(P)$ is bipartite, it suffices by the Alon--Tarsi Theorem to find an orientation of the edges of $G-E(P)$ such that the outdegree of each vertex is strictly less than its list size. Such an orientation is given in Figure~\ref{fig:122chord}. This completes the proof of the lemma.
\end{proof}

\begin{figure}[!htbp]
\centering
\begin{subfigure}[b]{.45\linewidth}
\centering
\begin{tikzpicture}[scale = .9]

\draw circle (2);

\node[bdot, label={below:$p_1$}] at (-30:2) (p1) {};
\node[bdot, label={below:$p_2$}] at (-60:2) (p2) {};
\node[bdot, label={below:$p_3$}] at (270:2) (p3) {};
\node[bdot, label={below:$p_4$}] at (240.5:2) (p4) {};
\node[bdot, label={below:$p_5$}] at (210:2) (p5) {};
\node[wdot, label={above:$w$}] at (90:2) (w) {};

\node[triangle, label={right:$y$}] at (0:2) (y) {};
\node[triangle, label={above: $x$}] at (0:1) (x) {};
\node[triangle, label={left:$\tilde{y}$}] at (180:2) (yt) {};
\node[triangle, label={above: $\tilde{x}$}] at (180:1) (xt) {};
\node[square, label={[label distance = -.1cm]135:$v$}] at (0:0) (v) {};

\draw[-] (yt)--(xt)--(v)--(x)--(y);
\draw[-] (p4)--(xt);
\draw[-] (p2)--(x);
\draw[-] (p3)--(v)--(w);

\end{tikzpicture}
\end{subfigure}
\begin{subfigure}[b]{.45\linewidth}
\centering
\begin{tikzpicture}[scale = .9, decoration={markings, mark=at position 0.5 with {\arrow{>}}}]

\draw[postaction={decorate}, line width=.5pt] (180:2) arc (180:90:2);
\draw[postaction={decorate}, line width=.5pt] (0:2) arc (0:90:2);

\draw[postaction={decorate}, line width=.5pt] (180:2) arc (180:210:2);
\draw[postaction={decorate}, line width=.5pt] (0:2) arc (0:-30:2);

\node[bdot, label={below:$p_1$}] at (-30:2) (p1) {};
\node[bdot, label={below:$p_2$}] at (-60:2) (p2) {};
\node[bdot, label={below:$p_3$}] at (270:2) (p3) {};
\node[bdot, label={below:$p_4$}] at (240.5:2) (p4) {};
\node[bdot, label={below:$p_5$}] at (210:2) (p5) {};

\node[wdot, label={above:$w$}] at (90:2) (w) {};
\node[triangle, label={right:$y$}] at (0:2) (y) {};
\node[triangle, label={above: $x$}] at (0:1) (x) {};
\node[triangle, label={left:$\tilde{y}$}] at (180:2) (yt) {};
\node[triangle, label={above: $\tilde{x}$}] at (180:1) (xt) {};
\node[square, label={[label distance = -.1cm]135:$v$}] at (0:0) (v) {};

\draw[postaction={decorate}, line width=.5pt] (xt)--(yt);
\draw[postaction={decorate}, line width=.5pt] (v)--(xt);
\draw[postaction={decorate}, line width=.5pt] (v)--(x);
\draw[postaction={decorate}, line width=.5pt] (x)--(y);
\draw[postaction={decorate}, line width=.5pt] (w)--(v);

\draw[postaction={decorate}, line width=.5pt] (xt)--(p4);
\draw[postaction={decorate}, line width=.5pt] (v)--(p3);
\draw[postaction={decorate}, line width=.5pt] (x)--(p2);
\end{tikzpicture}
\end{subfigure}
\caption{Graph $G$ whose subgraph $G-E(P)$ is manually proven to be list-colorable in Lemma~\ref{lem:122chord2} by imparting the orientation seen on the right and applying the Alon--Tarsi Theorem.}
\label{fig:122chord}
\end{figure}
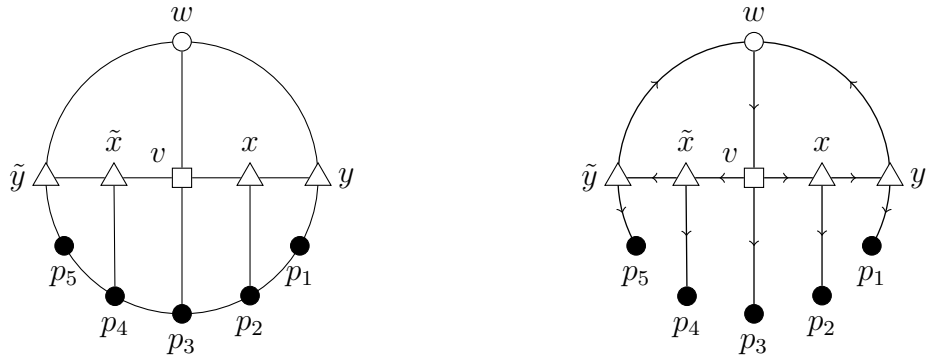

\begin{cor}\label{cor:122chord} 
    There is no $(1,2)$-$2$-chord.
\end{cor}

\begin{lem}\label{lem:112chord}
    There is no $(1,1)$-$2$-chord.
\end{lem}

\begin{proof}
    Suppose to the contrary that there exists a $(1,1)$-$2$-chord $W = p_ixp_j$ with $i < j$. Write $P = p_1p_2\cdots p_k$, where $k = |P|$. Since $G$ is $C_3$-free, we have $j - i \ge 2$. Note that $W$ is feasible because $G$ is $C_3$-free and contains no separating $4$- or $5$-cycle.
    
    We first determine the possible positions of $p_i$ and $p_j$ on $P$. By the definition of $P_W$, apart from the vertices of $P$ and the internal vertex $x$ of $W$, an additional boundary vertex can belong to $P_W$ only if it is consecutive to an end of $P$ and has list size 2. However, no such vertex exists. Indeed, if $k = 5$, then a vertex with list size 2 adjacent to an end of $P$ would be a bad vertex; if $3 \le k \le 4$, then a vertex with list size 2 adjacent to an end of $P$ would again be a bad vertex. Since $(G,P,L)$ is a valid target, neither situation can occur. Consequently, $P_W = p_1\cdots p_i x p_j \cdots p_k$, and hence $|P_W| = k - (j - i) + 2$. By Lemma \ref{lem:feasible}, $|P_W| \ge 5$. Since $j - i \ge 2$ and $k \le 5$, we obtain $5 \le |P_W| = k - (j-i)+2 \le k \le 5$. Therefore, $k = 5$ and $j - i = 2$. Reordering the vertices of $P$, we may assume that $W = p_1xp_3$ or $W = p_2xp_4$.

    Since $x$ is an interior vertex, $|L(x)| \in \{3,4\}$. Moreover, $x$ has two neighbors in $P$. If $|L(x)| = 3$, then $x$ is a bad vertex of $(G,P,L)$, a contradiction. Hence $|L(x)| = 4$. Since $W$ is feasible and $|P_W| = 5$, Lemma \ref{lem:badvertex} implies that $(G_{W,2}, P_W, L')$ has a bad vertex $y$ with $|L(y)| = 3$ and $|N_G(y) \cap P_W| \ge 2$. Since $y$ is not a bad vertex of $(G,P,L)$, we have $|N_G(y) \cap P| \le 1$. Moreover, $P_W - P = \{x\}$. It follows that $y$ is adjacent to $x$ and to exactly one vertex in $P_W \cap P$.

    \textbf{Case 1:} Assume that $W = p_1xp_3$. If $y$ is a boundary vertex, then $W' = p_1xy$ forms a $(1,3)$-$2$-chord $W'$ with $|P\cap B(G_{W',2})| = 1$, which contradicts Lemma~\ref{lem:f-chords}. Thus, $y$ is an interior vertex that is adjacent to $x$ and exactly one of $p_4$ and $p_5$. If $y$ is adjacent to $p_5$, then $W' = p_1xyp_5$ is a feasible chord in order to meet the degree conditions of $y$ without creating a separating $5$-cycle. However, $|P_{W'}| = 4$ contradicting Lemma~\ref{lem:feasible}. Hence, $y$ is adjacent to $x$ and $p_4$, and we set $W^{(1)} = p_1xyp_4$. To avoid $p_3xyp_4$ being a separating $4$-cycle, $W^{(1)}$ is a feasible chord with $P_{W^{(1)}} = p_1xyp_4p_5$. By Lemma~\ref{lem:badvertex}, there exists a vertex $z$ such that $|L(z)| = 3$ and $|N_{G}(z) \cap P_{W^{(1)}}| \geq 2$. 
    
    If $z$ is an interior vertex, then $z$ is not adjacent to $x$ and $p_j$ for $j = 4,5$ as this would create a separating $4$- or $5$-cycle containing $y$. Thus, $z$ is adjacent to $y$. Observe that $z$ cannot be adjacent to $p_1$ as $p_1xyz$ would be a separating $4$-cycle for a neighbor of the vertex $x$. Thus, $W^{(2)} = p_1xyzp_5$ is a chord that must be feasible to meet the degree requirements of $x$ and $z$. There must then exist a bad vertex $w$ in $G_{W^{(2)},2}$ with $|L(w)| = 3$ and $|N_{G}(w) \cap  P_{W^{(2)}}| \geq 2$. By the distance condition on $X_{G,L}$, $w$ lies on the boundary. This bad vertex is not adjacent to $x$ nor to $z$, since this would create a $(1,3)$-$2$-chord satisfying the conditions in Lemma~\ref{lem:f-chords}. Moreover, $w$ is not adjacent to $y$ and $p_i$ where $i = 1,5$ as this would create a separating $4$-cycle containing either a neighbor of $x$ or the vertex $z$. Therefore, $w$ is adjacent to both $p_1$ and $p_5$ with $p_1wp_5$ being a path in $B(G)$. This creates a separating $6$-cycle $p_1wp_5zyx$, which contradicts the distance condition of $X_{G,L}$.

    If $z$ is a boundary vertex, then $z$ cannot be adjacent to $x$ as $p_1xz$ would form a $(1,3)$-$2$-chord $W'$ with $|P\cap B(G_{W',2})| = 1$. In addition, $z$ cannot be adjacent to $y$ and $p_1$ as $p_1xyz$ would be a separating $4$-cycle containing a neighbor of $x$. Thus, $z$ must be adjacent to $y$ and $p_5$ with $z$ consecutive with $p_5$ on the boundary. The chord $W^{(2)} = p_1xyz$ is feasible due to the degree constraints on $x$. Since $|P_{W^{(2)}}|$ must be $5$, we have $|L(z^{-})| = 2$ where $z^{-} \not= p_5$ is the other consecutive neighbor of $z$ on the boundary of $G$. By Lemma~\ref{lem:badvertex}, there exists a vertex $w$ in $G_{W^{(2)},2}$ with $|L(w)| = 3$ and $|N_{G}(w) \cap  P_{W^{(2)}}| \geq 2$.
    
    Assume that $w$ lies on the boundary. Note that $w$ cannot be adjacent to $x$ or $z$ by Lemma~\ref{lem:f-chords}. If $w$ is adjacent to $y$, then it cannot be adjacent to $z^-$ as either $wz^-zy$ is a bad cycle or $wz^-$ is a $(2,3)$-$1$-chord. However, this implies that if $w$ is adjacent to $y$, the chord $wyz$ is a feasible $(3,3)$-$2$-chord with $|L(y)| = 3$ contradicting Lemma~\ref{lem:332chord}. Thus, $w$ is adjacent to $p_1$ and $z^-$. To meet the degree conditions of $w$, $P_{W^{(2)}} \cup \{w\}$ is a separating $6$-cycle; however, this either creates a bad cycle or breaks the distance condition on $X_{G,L}$. Thus, $w$ is an interior vertex.

    Assume that $w$ is not adjacent to the vertex $y$. Note that $w$ cannot be adjacent to $p_1$ as both $p_1wz$ and $p_1wz^-$ would be $(1,2^+)$-$2$-chords that contradict Lemma~\ref{lem:f-chords}. If $w$ is adjacent to $x$ and $z^-$, then $W' = p_1xwz^-$ is a feasible chord with $|P_{W'}| = 4$. Thus, $w$ is adjacent to $x$ and $z$, and we consider the chord $W^{(3)} = p_1xwz$ with $P_{W^{(3)}} = p_1xwzz^{-}$. This chord is feasible as $z^-$ is not adjacent to $p_1$ implying the existence of a bad vertex $v$ in $G_{W^{(3)},2}$ with $|L(v)| = 3$ and $|N_{G}(v) \cap  P_{W^{(3)}}| \geq 2$. By the same logic applied to the vertex $w$, the vertex $v$ cannot lie on the boundary. Moreover, by the distance condition on $X_{G,L}$, $v$ cannot be an interior vertex that is adjacent to $w$. Thus, by a similar argument, $v$ must be an interior vertex adjacent to $x$ and $z$. However, this creates a separating $4$-cycle $xvzy$ containing the vertex $w$.
    
    Assume now that $w$ is adjacent to the vertex $y$. Note that $w$ cannot be adjacent to both $y$ and $p_1$ as $p_1xyw$ would be a separating $4$-cycle containing a neighbor of $x$. Thus, $w$ is adjacent to $y$ and $z^-$, and we consider the chord $W^{(3)} = p_1xy\tilde{y}z^-$. To satisfy the degree conditions of $x$, this chord is feasible, and $G_{W^{(3)},2}$ contains a vertex $v$ satisfying the conditions specified in Lemma~\ref{lem:badvertex}. By the distance condition of $X_{G,L}$, $v$ must lie on the boundary. The vertex $v$ cannot be adjacent to $x$ nor to the vertex $w$ as the chords $p_1xv$ and $z^-wv$ respectively would contradict Lemma~\ref{lem:f-chords}. The vertex $v$ is also not adjacent to $y$ as either $yvz^{-}z$ is a bad cycle or $zyv$ is a feasible $(3,3)$-$2$-chord with $|L(y)| = 3$ contradicting Lemma~\ref{lem:332chord}. Thus, $v$ is adjacent to $p_1$ and $z^-$ with $p_1vz^-$ a subpath of the boundary. In order to meet the degree requirements of $v$, $P_{W^{(3)}} \cup \{v\}$ is a separating $6$-cycle; however, this contradicts the distance condition of $X_{G,L}$.

    \textbf{Case 2:} Assume that $W = p_2xp_4$ and $y$ is an interior vertex. By symmetry, we assume that $y$ is adjacent to $x$ and $p_1$, and we consider the chord $W^{(1)} = p_1yxp_4$ with $P_{W^{(1)}} = p_1yxp_4p_5$. Since $p_1$ is not adjacent to $p_5$, $W^{(1)}$ is a feasible chord. Thus, there exists a vertex $z$ such that $|L(z)| = 3$ and $|N_{G}(z) \cap  P_{W^{(1)}}| \geq 2$. Additionally, $|P_{W^{(1)}} \cap P| = 3$ implies $z$ is adjacent to either $y$ or $x$.

    Assume that $z$ is an interior vertex.  If $z$ is adjacent to $y$, then it must be adjacent to $p_4$ or $p_5$. However, this creates either a separating $4$- or $5$-cycle that contains a neighbor of $x$. If $z$ is adjacent to $x$, then it cannot be adjacent to $p_1$ as $xzp_1p_2$ would be a separating $4$-cycle containing the vertex $y$. Thus, $z$ is adjacent to $x$ and $p_5$, and we consider the chord $W^{(2)} = p_1yxzp_5$. As $p_1$ is not adjacent to $p_5$, $W^{(2)}$ is feasible, and there exists a vertex $w$ in $G_{W^{(2)},2}$ such that $|L(w)| = 3$ and $|N_{G}(w) \cap  P_{W^{(2)}}| \geq 2$. If $w$ is an interior vertex, then it is not adjacent to $y$ or $z$ by the distance condition on $X_{G,L}$. If $w$ is an interior vertex adjacent to $x$, then it is adjacent to $p_1$ or $p_5$ which in either case creates a separating $4$-cycle. Thus, $w$ is a boundary vertex. By the same argument $w$ cannot be adjacent to $x$. Furthermore, $w$ is not adjacent to $y$ as $p_1yw$ is a $(1,3)$-$2$-chord that meets the conditions of Lemma~\ref{lem:f-chords}. Similarly, $w$ cannot be adjacent to $z$.  Since $w$ is not a bad vertex of $G$, it is not adjacent to both $p_1$ and $p_5$. Therefore, $z$ is a boundary vertex.
    
    If $z$ is a boundary vertex, then Lemma~\ref{lem:f-chords} implies that $z$ is not adjacent to $y$ and must therefore be adjacent to $x$. To avoid creating a separating $4$-cycle, $z$ is not adjacent to $p_1$ and must be adjacent to $p_5$ with the edge $p_5z$ lying on the boundary. We now consider the chord $W^{(2)} = p_1yxz$. The chord $W^{(2)}$ is feasible as $z$ is not adjacent to $p_1$ implying the existence of a boundary vertex $z^{-}$ such that $|L(z^{-})|=2$ and $p_5zz^{-}$ forms an induced path in $B(G)$. Like before, there exists a vertex $w$ such that $|L(w)| = 3$ and $|N_{G}(w) \cap  P_{W^{(2)}}| \geq 2$. 
    
    Assume that $w$ is an interior vertex. If $w$ is adjacent to $y$ and $z^-$, then the chord $W' = p_1ywz^-$ is feasible with $|P_{W'}| = 4$, a contradiction. If $w$ is adjacent to $y$ and $z$, then the chord $W' = p_1ywz$ is feasible $(1,3)$-$3$-chord that contradicts Lemma~\ref{lem:233chord}. To avoid creating a chord ruled out in Lemma~\ref{lem:233chord}, $w$ is not adjacent to $p_1$ and $z$ nor to $p_1$ and $z^-$. Thus, $w$ is adjacent to $x$ and must be adjacent to $z^-$ to avoid creating a separating $4$-cycle. Consider the chord $W^{(3)} = p_1yxwz^-$ which is feasible as $z^-$ is not adjacent to $p_1$ on the boundary. Thus, there exists a bad vertex $v$ meeting the requirements in Lemma~\ref{lem:badvertex}. Assume $v$ is an interior vertex. Observe that $v$ is not adjacent to $y$ or $w$ by the definition of $X_{G,L}$. If $v$ is adjacent to $x$, then it is adjacent to either $p_1$ or $z^-$ creating a separating $4$-cycle. If $v$ is adjacent to $p_1$ and $z^-$, then $p_1vz^-$ is a $(1,2)$-$2$-chord that contradicts Lemma~\ref{lem:f-chords}. Thus, $v$ is a boundary vertex. By the same argument as before, $v$ cannot be adjacent to $x$. Moreover, by Lemma~\ref{lem:f-chords}, $v$ is not adjacent to $y$ nor to $w$. Therefore, $v$ is adjacent to $p_1$ and $z^-$ with $p_1vz^-$ being a subpath of the boundary. However, this creates a separating $6$-cycle to meet the degree requirements of $v$ which contradicts the definition of $X_{G,L}$.
    
    Assume now that $w$ is a boundary vertex which, for a similar reason to $z$, must be adjacent to $x$ and $z^{-}$ with $z^{-}w$ an edge in $B(G)$. This creates the feasible chord $W^{(3)} = p_1yxw$ due to $G_{W^{(3)}}$ containing a neighbor of $y$. Therefore, there exists a bad vertex $v$ in $G_{W^{(3)}}$ with $|L(v)| = 3$ and $|N_{G}(v) \cap  P_{W^{(3)}}| \geq 2$. Using the same argument applied to the vertex $w$, the vertex $v$ is adjacent to $x$ and $w^-$ where $w^- \not = z^-$ is $w$'s other consecutive neighbor on the boundary with $|L(w^-)| = 2$. Note that the chord $p_1yxv$ is feasible once again as it contains a neighbor of $y$. As this process is infinite while $V(G)$ is finite, we arrive at a contradiction.

    \textbf{Case 3:} Assume that $W = p_2xp_4$ and $y$ is a boundary vertex. Since $y$ is not a bad vertex of $G$, $y$ must be adjacent to $x$ and exactly one of $p_1$ or $p_5$. By symmetry, we assume $y$ is adjacent to $x$ and $p_1$; moreover, the edge $yp_1$ lies on the boundary by Lemma~\ref{lem:f-chords}. Let $p_1y_0y_0^{+}y_1y_1^+\dots y_k$ be the maximal subpath of $B(G)$ such that $|L(y_i)| = 3$ for all $0 \leq i \leq k$, $|L(y_i^{+})| = 2$ for all $0 \leq i \leq k-1$, and $x$ is adjacent to all $y_i$ for all $0 \leq i \leq k$ where $y_0 = y$. 

    If $y_k$ is adjacent to $p_5$, then $y_kp_5$ is an edge on the boundary of $G$, and $G$ has no other vertices, see Figure~\ref{fig:112chord}. To show the bipartite graph $G-E(P)$ is $L$-colorable, it suffices by Alon--Tarsi to find an orientation on $G-E(P)$ such that the outdegree of every vertex is strictly less than its list size. Such an orientation is achieved by making each precolored vertex a sink, making $y_0y_0^{+}y_1y_1^+\dots y_k$ a directed path, orienting the edge $xy_0$ in the direction of $y_0$, and orienting all other edges $xy_i$ for $1 \leq i \leq k$ in the direction of $x$, see Figure~\ref{fig:112chord}.
    
    Therefore, $y_k$ is not adjacent to $p_5$, and the chord $W^{(1)} = p_4xy_k$ is feasible. By Lemma~\ref{lem:feasible}, there exists a vertex $y_{k}^+ \not= y_{k-1}^+$ adjacent to $y_k$ on $B(G)$ such that $|L(y_{k}^+)| = 2$ and $P_{W^{(1)}} = p_5p_4xy_{k}y_{k}^+$. Moreover, there exists a vertex $z$ such that $|L(z)| = 3$ and $|N_{G}(z) \cap P_{W^{(1)}}| \geq 2$. 
    
    Assume that $z$ is a boundary vertex. It is not adjacent to $p_4$ or $y_k$ by Lemma~\ref{lem:f-chords}. If $z$ is not adjacent to $x$, then $p_5zy_k^+$ must be a subpath of $B(G)$. To meet the degree conditions of $z$, $p_4xy_ky_k^+zp_5$ is a separating $6$-cycle. To avoid creating a bad cycle while still satisfying Lemma~\ref{lem:sep6}, there must exist interior vertices $w$ and $\tilde{w}$ with list size $3$ such that $w$ is adjacent to $y_k^+$ and $x$, $\tilde{w}$ is adjacent to $p_4$ and $z$, and $w\tilde{w} \in X_{G,L}$, see Figure~\ref{fig:112chord}. We claim that $G-E(P)$ is $L$-colorable. Since $G-E(P)$ is bipartite, it suffices, by the Alon--Tarsi Theorem, to find an orientation on $G-E(P)$ such that the outdegree of every vertex is strictly less than its list size. Such an orientation is achieved by making each precolored vertex a sink, making $xw\tilde{w}zy_k^+y_k\dots y_0$ a directed path, orienting the edge $wy_k^+$ in the direction of $y_k^+$, and orienting all edges $xy_i$ for $0 \leq i \leq k$ in the direction of $x$, see Figure~\ref{fig:112chord}.

    Hence, if $z$ is a boundary vertex, it must be adjacent to $x$. Furthermore, by the definition of $y_k$, $z$ is not adjacent to $y_k^+$ and is therefore adjacent to $p_5$ with the edge $zp_5$ lying on the boundary. Let $p_5z_0z_0^{-}z_1z_1^-\dots z_{\ell}$ be the maximal subpath of $B(G)$ such that $|L(z_i)| = 3$ for all $0 \leq i \leq \ell$, $|L(z_i^{-})| = 2$ for all $0 \leq i \leq \ell-1$, and $z_i$ is adjacent to $x$ for all $0 \leq i \leq \ell$ where $z_0 = z$. We consider the chord $W^{(2)} = z_{\ell}xy_{k}$ which is feasible by the maximality of $k$. Thus, there exists a boundary vertex $z_{k}^{-}\not=z_{k-1}^{-}$ adjacent to $z_k$ that satisfies $|L(z_{k}^{-})| = 2$, and $P_{W^{(2)}} = z_k^-z_kxy_ky_k^+$. Additionally, $(G_{W^{(2)},2},P_{W^{(2)}},L'')$ has a bad vertex $w$ such that $|L(w)| = 3$ and $|N_{G}(w) \cap P_{W^{(2)}}| \geq 2$. 
    
    By the choice of $k$ and $\ell$ and the absence of $(3,3)$-$1$-chords, if $w$ is a boundary vertex, then $y_k^+wz_{\ell}^-$ is a subpath of the boundary and $w$ is not adjacent to $x$. However, this implies $P_{W^{(2)}} \cup \{w\}$ is a separating $6$-cycle, which creates a bad cycle. Therefore, $w$ is an interior vertex. Note that $w$ is not adjacent to $y_k$ as either $y_kwz_{\ell}^-$ is a $(2,3)$-$2$-chord or $y_kwz_{\ell}$ is a feasible $(3,3)$-$2$-chord with $|L(w)| = 3$ contradicting Lemmas~\ref{lem:f-chords} and~\ref{lem:332chord}, respectively. By a symmetric argument, $w$ cannot be adjacent to $z_{\ell}$. If $w$ is adjacent to both $y_k^+$ and $z_{\ell}^-$, then $y_k^+wz_{\ell}^-$ is a $(2,2)$-$2$-chord. Therefore, by symmetry, we may assume that $w$ is adjacent to $x$ and $y_{k}^+$, and we consider the chord $W^{(3)} = z_{\ell}xwy_{k}^+$ with $P_{W^{(3)}} = z_{\ell}^-z_{\ell}xwy_{k}^+$. Since $W^{(3)}$ is feasible, $G_{W^{(3)},2}$ has a bad vertex $v$ that is adjacent to at least two vertices in $P_{W^{(3)}}$ and has list size $3$. 
    
    Assume that $v$ is a boundary vertex. Since $G_{W^{(2)},2}$ has no bad vertex on the boundary, $v$ must be adjacent to $w$. However, this creates a $(2,3)$-$2$-chord $y_k^+wv$, which contradicts Lemma~\ref{lem:f-chords}. Thus, $v$ is an interior vertex. Observe that $v$ is not adjacent to $w$ as either $W' = y_k^+wvz_{\ell}^-$ is a feasible chord with $|P_{W'}| = 4$ or $y_k^+wvz_{\ell}$ is a feasible $(2,3)$-$3$-chord that contradicts Lemma~\ref{lem:233chord}. To avoid creating a separating $4$-cycle or a chord in Lemma~\ref{lem:f-chords}, $v$ must be adjacent to $x$ and $z_{\ell}^-$, and we consider the feasible chord $W^{(4)} = z_{\ell}^-vxwy_{k}^+$. There must exist a bad vertex $u$ with list size $3$ adjacent to at least two vertices in $W^{(4)}$. Using an argument similar to the one applied to the vertex $v$ implies that $u$ is an interior vertex. By the definition of $X_{G,L}$, $u$ is not adjacent to $w$ nor to $v$. Additionally, $u$ is not adjacent to both $y_k^+$ and $z_{\ell}^-$ as $y_k^+uz_{\ell}^-$ would be a $(2,2)$-$2$-chord. However, this implies that $u$ is adjacent to $x$ creating a separating $4$-cycle that contains either $v$ or $w$. 
    
     Thus, $z$ is an interior vertex, and we may assume that $G_{W^{(1)},2}$ has no bad vertex on the boundary. Note that $z$ is not adjacent to $y_k$, otherwise either $p_4zy_k$ or $p_5zy_k$ forms a feasible $(1,3)$-$2$-chord satisfying the conditions of Lemma~\ref{lem:132chord}. Moreover, $z$ is not adjacent to $p_4$ as $p_4zy_k^+$ would then form a $(1,2)$-$2$-chord contradicting Corollary~\ref{cor:122chord}. Therefore, by Lemma~\ref{lem:f-chords}, $z$ is adjacent to $x$ and exactly one of $p_5$ and $y_k^+$. We first assume that $z$ is adjacent to $x$ and $y_k^+$ and consider the feasible chord $W^{(2)} = p_4xzy_k^+$ with  $P_{W^{(2)}} = p_5p_4xzy_k^+$. The target $(G_{W^{(2)},2},P_{W^{(2)}},L'')$ must have a bad vertex $w$ with $|L(w)| = 3$ and $|N_{G}(w) \cap P_{W^{(2)}}| \geq 2$. If $w$ is on the boundary, then it is adjacent to $z$ and exactly one vertex in $\{p_5, p_4, x, y_k^+\}$ as $G_{W^{(1)}}$ does not have a bad vertex on the boundary. However, this creates the $(2,3)$-$2$-chord $y_k^+zw$, a contradiction. Hence $w$ is an interior vertex. 
     
     To avoid creating a separating $4$-cycle or a $(1,2)$-$2$-chord, $w$ cannot be adjacent to $y_{k}^+$. If $w$ is adjacent to $z$, then it must be adjacent to $p_4$; otherwise, $W' = p_5wzy_k^+$ is a feasible chord with $|P_{W'}| = 4$. In this case, consider the feasible chord $W^{(3)} = p_4wzy_{k}^+$ with $P_{W^{(3)}} = p_5p_4wzy_{k}^+$. There then exists a bad vertex $v$ of $G_{W^{(3)},2}$ with list size $3$ that is adjacent to at least two vertices in $P_{W^{(3)}}$. Note that by the distance condition on $X_{G,L}$, the vertex $v$ must lie on the boundary. Moreover, as $v$ is not a bad vertex of $G_{W^{(1)},2}$, it is adjacent to exactly one vertex in $\{w,z\}$ and exactly one vertex in $\{p_4, p_5, y_k^{+}\}$. Note that $v$ is not adjacent to $w$ as either $y_{k}^+zwv$ is a feasible $(2,3)$-$3$-chord that contradicts Lemma~\ref{lem:233chord} or $p_4wv$ is a feasible $(1,3)$-$2$-chord that contradicts Lemma~\ref{lem:132chord}. As $v$ is also not adjacent to $z$ due to the absence of $(2,3)$-$2$-chords, we have that $w$ cannot be adjacent to both $z$ and $p_4$. Assume now that $w$ is adjacent to $x$ and $p_5$ and consider the feasible chord $W^{(3)} = p_5wxzy_{k}^+$. As before, there exists a bad vertex $v$ of $G_{W^{(3)},2}$ with list size $3$ that is adjacent to at least two vertices in $P_{W^{(3)}} = W^{(3)}$. Assume that $v$ is a boundary vertex, it implies that $v$ is adjacent to exactly one vertex in $\{w,z\}$ and exactly one vertex in $\{p_5, x, y_k^+\}$. However, if $v$ is adjacent to either $w$ or $z$, then it creates a chord that contradicts Lemma~\ref{lem:f-chords}. Thus, $v$ is an interior vertex. If $v$ is adjacent to either $w$ or $z$, this would contradict the distance condition on $X_{G,L}$. If $v$ is adjacent to $x$, then either $p_4xvp_5$ or $y_kxvy_k^+$ is a separating $4$-cycle that contains the vertex $w$ or $z$ respectively. As $G$ has no $(1,2)$-$2$-chords, $v$ is not adjacent to both $p_5$ and $y_k^{+}$. Therefore, $z$ cannot be adjacent to both $x$ and $y_k^+$.

     We may therefore assume $z$ is adjacent to $x$ and $p_5$, and $G_{W^{(1),2}}$ has no other bad vertex. Consider the feasible chord $W^{(2)} = p_5zxy_k$ with  $P_{W^{(2)}} = p_5zxy_ky_{k}^+$. There must exist a bad vertex $w$ in $G_{W^{(2),2}}$ with list size $3$ adjacent to $z$ and exactly one vertex in $\{p_5,x,y_k,y_k^+\}$. If $w$ is a boundary vertex, then $p_5zw$ is a $(1,3)$-$2$-chord that contradicts Lemma~\ref{lem:f-chords}. Thus, $w$ is an interior vertex. If $w$ is adjacent to $z$ and $y_k$, then $p_5zwy_k$ is a feasible $(1,3)$-$3$-chord that contradicts Lemma~\ref{lem:233chord}. If $w$ is adjacent to $z$ and $y_k^+$, then $W' = p_5zwy_k^+$ is a feasible chord with $|P_{W'}| = 4$. Therefore, $z$ is not adjacent to both $p_5$ and $x$ completing the proof.
    \end{proof}

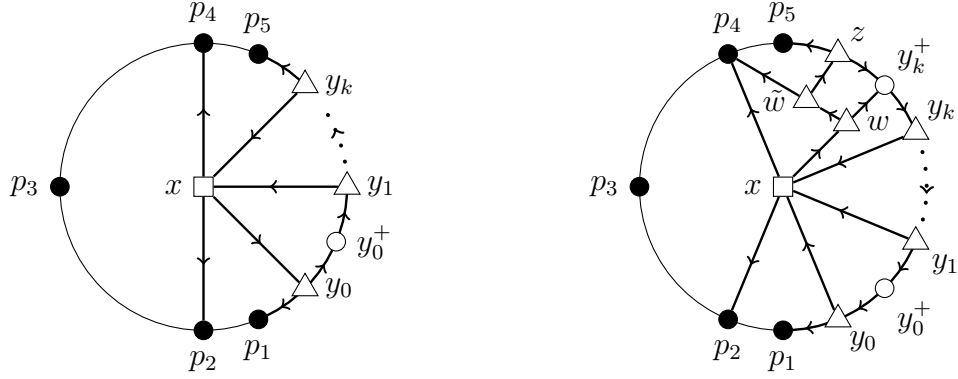
\begin{figure}[!htbp]
\centering
\begin{subfigure}[b]{.45\linewidth}
\centering
\begin{tikzpicture}[scale = .95, decoration={markings, mark=at position 0.55 with {\arrow{>}}}]

\draw (45:2) arc (45:360:2);
\draw[postaction={decorate}, line width=.9pt] (45:2) arc (45:67.5:2);
\draw[postaction={decorate}, line width=.9pt] (-45:2) arc (-45:-67.5:2);
\draw[postaction={decorate}, line width=.9pt] (-45:2) arc (-45:-22.5:2);
\draw[postaction={decorate}, line width=.9pt] (-22.5:2) arc (-22.5:0:2);
\draw[postaction={decorate, line width = 1pt}, line width=1.5pt, line cap=round, dash pattern=on 0pt off 4.1\pgflinewidth] (10:2) arc (10:35:2);

\node[bdot, label={below:$p_1$}] at (-67.5:2) (p1) {};
\node[bdot, label={below:$p_2$}] at (270:2) (p2) {};
\node[bdot, label={left:$p_3$}] at (180:2) (p3) {};
\node[bdot, label={above:$p_4$}] at (90:2) (p4) {};
\node[bdot, label={above:$p_5$}] at (67.5:2) (p5) {};
\node[triangle, label={right:$y_0$}] at (-45:2) (y0) {};
\node[wdot, label={right:$y_0^{+}$}] at (-22.5:2) (y0+) {};
\node[triangle, label={right:$y_1$}] at (0:2) (y1) {};
\node[triangle, label={right: $y_{k}$}] at (45:2) (yk) {};
\node[square, label={left: $x$}] at (0:0) (x) {};

\draw[-, postaction={decorate}, line width=.9pt] (x)--(y0);
\draw[-, postaction={decorate}, line width=.9pt] (x)--(p2);
\draw[-, postaction={decorate}, line width=.9pt] (x)--(p4);
\draw[-, postaction={decorate}, line width=.9pt] (y1)--(x);
\draw[-, postaction={decorate}, line width=.9pt] (yk)--(x);

\end{tikzpicture}
\end{subfigure}
\begin{subfigure}[b]{.45\linewidth}
\centering
\begin{tikzpicture}[scale = .95, decoration={markings, mark=at position 0.6 with {\arrow{>}}}]

\draw (22.5:2) arc (22.5:337.5:2);
\draw[postaction={decorate, line width = 1pt}, line width=1.5pt, line cap=round, dash pattern=on 0pt off 4.2\pgflinewidth] (14:2) arc (14:-14:2);
\draw[postaction={decorate}, line width=.9pt] (-67.5:2) arc (-67.5:-90:2);
\draw[postaction={decorate}, line width=.9pt] (-45:2) arc (-45:-67.5:2);
\draw[postaction={decorate}, line width=.9pt] (-22.5:2) arc (-22.5:-45:2);
\draw[postaction={decorate}, line width=.9pt] (45:2) arc (45:22.5:2);
\draw[postaction={decorate}, line width=.9pt] (67.5:2) arc (67.5:45:2);
\draw[postaction={decorate}, line width=.9pt] (67.5:2) arc (67.5:90:2);

\node[bdot, label={below:$p_1$}] at (270:2) (p1) {};
\node[bdot, label={below:$p_2$}] at (247.5:2) (p2) {};
\node[bdot, label={left:$p_3$}] at (180:2) (p3) {};
\node[bdot, label={above:$p_4$}] at (112.5:2) (p4) {};
\node[bdot, label={above:$p_5$}] at (90:2) (p5) {};
\node[triangle, label={[label distance = -.1cm]67.5:$z$}] at (67.5:2) (z) {};
\node[wdot, label={[label distance = -.1cm]45:$y_{k}^{+}$}] at (45:2) (yk+) {};
\node[triangle, label={[label distance = -.1cm]22.5:$y_{k}$}] at (22.5:2) (yk) {};
\node[triangle, label={[label distance = -.1cm]-22.5:$y_1$}] at (-22.5:2) (y1) {};
\node[wdot, label={[label distance = -.1cm]-45:$y_0^+$}] at (-45:2) (y0+) {};
\node[triangle, label={[label distance = -.1cm]-67.5:$y_0$}] at (-67.5:2) (y0) {};
\node[square, label={left: $x$}] at (0:0) (x) {};
\node[triangle, label={left:$\tilde{w}$}] at (75:1.25) (w') {};
\node[triangle, label={right:$w$}] at (45:1.25) (w) {};

\draw[-, postaction={decorate}, line width=.9pt] (w')--(z);
\draw[-, postaction={decorate}, line width=.9pt] (w)--(yk+);

\draw[-, postaction={decorate}, line width=.9pt] (w)--(w');
\draw[-, postaction={decorate}, line width=.9pt] (w')--(p4);

\draw[-, postaction={decorate}, line width=.9pt] (x)--(w);
\draw[-, postaction={decorate}, line width=.9pt] (x)--(p2);
\draw[-, postaction={decorate}, line width=.9pt] (x)--(p4);
\draw[-, postaction={decorate}, line width=.9pt] (y0)--(x);
\draw[-, postaction={decorate}, line width=.9pt] (y1)--(x);
\draw[-, postaction={decorate}, line width=.9pt] (yk)--(x);

\end{tikzpicture}
\end{subfigure}
\caption{Graphs $G$ whose subgraphs $G-E(P)$ are manually proven to be list-colorable in Lemma~\ref{lem:112chord} by directing the edges in $G-E(P)$ as illustrated and applying the Alon--Tarsi Theorem.}
\label{fig:112chord}
\end{figure}

\begin{lem}\label{lem:chords}
    $G$ has no chords.
\end{lem}

\begin{proof}
    By Lemmas~\ref{lem:feasible} and~\ref{lem:f-chords}, it remains to show that $G$ has no $(1,3^+)$-chords of the form $W = p_3u$ where $|P| = 5$. 
    
    First, assume that $|L(u)| = 3$ with $P$ relabeled such that $p_1$ is a vertex in $G_{W,2}$. Since $G$ has no bad vertex, $u$ is not adjacent to $p_1$ on the boundary and $G_{W,2}$ is feasible. Let $u^+$ be the boundary vertex in $G_{W,2}$ besides $p_3$ that is adjacent to $u$. Note that $|L(u^+)| = 2$ in order for $|P_{W}| = 5$ and $P_W = p_1p_2p_3uu^+$. By Lemma~\ref{lem:badvertex}, there exists a bad vertex $v$ of $G_{W,2}$ with list size $3$ adjacent to at least two vertices in $P_W$. As $v$ is not a bad vertex of $G$, $v$ must be adjacent to either $u$ or $u^+$. If $v$ is an interior vertex, then $v$ must be adjacent to $u$ and $p_2$ by Lemma~\ref{lem:f-chords}; however, this contradicts Lemma~\ref{lem:132chord}. Thus, $v$ must be boundary vertex. To avoid having a $(2,3)$ or $(3,3)$-chord, $u^+vp_5$ must form a subpath of the boundary of $G$ and $P_{W} \cup \{v\}$ forms a separating $6$-cycle. However, the only possible separating $6$-cycles would lead to a bad cycle or a $(1,2)$-$2$-chord in $G$, a contradiction.

    Assume that $|L(u)| = 4$. As $G$ has no bad vertex, $u$ is adjacent to at most one of $p_1$ and $p_5$ on the boundary of $G$. Relabel $P$ such that $u$ is not adjacent $p_1$ and define $G_{W,1}$, $G_{W,2}$ appropriately such that $G_{W,2}$ contains $p_1$. By construction, $G_{W,2}$ is feasible with $u^+$ defined as above and $P_W = p_1p_2p_3uu^+$. By Lemma~\ref{lem:badvertex} there exists a bad vertex $v$ of $G_{W,2}$ with list size $3$ adjacent to at least two vertices in $P_W$. 
    
    Assume $G_{W,2}$ has no bad vertex on its boundary. Thus, $v$ must be an interior vertex and adjacent to $p_2$ and $u$. We consider the $2$-chord $W' = p_2vu$ with $P_{W'} = p_5p_4vuu^+$. As $u^+$ is not adjacent to $p_1$, $G_{W',2}$ is feasible and has a bad vertex $w$ with list size $3$. Assume that $w$ is a boundary vertex. As $G_{W,2}$ has no bad vertex on its boundary by assumption, $w$ must be adjacent to $v$ and exactly one of $p_1$ and $u^+$. Moreover, Lemma~\ref{lem:132chord} implies that $w$ is adjacent to $v$ and $p_1$ with the edge $p_1w$ lying on the boundary. Consider the chord $W'' = uvw$ which is feasible as $w$ is not a bad vertex of $G_{W,2}$. This implies that there exists a boundary vertex $w^- \not = p_1$ that is consecutive to $w$ with list size $2$ and $P_{W''} = u^+uvww^-$. Moreover, $G_{W'',2}$ has a bad vertex $x$ with list size $3$.
    Note that $x$ cannot be an interior vertex by Lemmas~\ref{lem:f-chords},~\ref{lem:233chord}, and~\ref{lem:332chord}. If $x$ is a boundary vertex, then it cannot be adjacent to $v$ as either $xvww^-$ is a bad cycle of $G$ or $p_2vx$ is a feasible $(1,3)$-$2$-chord that contradicts Lemma~\ref{lem:132chord}. This would then imply $P_{W''} \cup \{x\}$ is a separating $6$-cycle contradicting the absence of bad cycles in $G$ or the definition of $X_{G,L}$.
    
    Assume that $w$ is an interior vertex. If $w$ is adjacent to $u$, then $w$ must be adjacent to $p_2$ or $p_1$ creating a separating $4$- or $5$-cycle containing the vertex $v$ in its interior. As $G$ has no $(1,2)$-$2$-chords by Corollary~\ref{cor:122chord}, $w$ must be adjacent to $v$ and exactly one of $u^+$ or $p_1$. Consider the case in which $w$ is adjacent to $v$ and $u^+$ and let $W'' = p_2vwu^+$. Once again, $W''$ is feasible and must therefore contain a bad vertex $x$ with list size $3$. The definition of $X_{G,L}$ implies $x$ is a boundary vertex. By Lemma~\ref{lem:f-chords}, $x$ is not adjacent to $p_2$ nor to $w$. If $x$ is adjacent to $v$, then either $uu^+xv$ is a separating $4$-cycle with $w$ in its interior or $u^+wvx$ is a feasible $(2,3)$-$3$-chord that contradicts Lemma~\ref{lem:233chord}. Thus, $u^+xp_1$ forms a subpath of the boundary contradicting the assumption that $G_{W,2}$ has no bad vertex on its boundary. By a similar argument, $w$ cannot be adjacent to $v$ and $p_1$.

    Therefore, $G_{W,2}$ has a bad vertex $v$ on its boundary. As $G$ has no $(3,1^+)$-chord, $P_{W} \cup \{v\}$ forms a separating $6$-cycle and $G_{W,2}$ is one of two possible subgraphs specified in Lemma~\ref{lem:sep6}. Note that there are only the two possibilities shown in Figure~\ref{fig:chords} for $G_{W,2}$ as $G$ has no $(1,2)$-$2$-chords. Furthermore, $G_{W,1}$ is either the four-cycle $up_3p_4p_5$ or by symmetry also one of the two possible subgraphs specified in Lemma~\ref{lem:sep6}. To complete the proof, it suffices to show that $G-E(P)$ is list-colorable in each of these cases. As $G$ is bipartite in each of these cases, it suffices by the Alon--Tarsi Theorem to find an orientation of the edges of $G-E(P)$ such that the outdegree of each vertex is strictly less than its list size. Such an orientation is given in Figure~\ref{fig:chords} when $G_{W,1}$ is a $4$-cycle. The other cases can be proved similarly.
\end{proof}

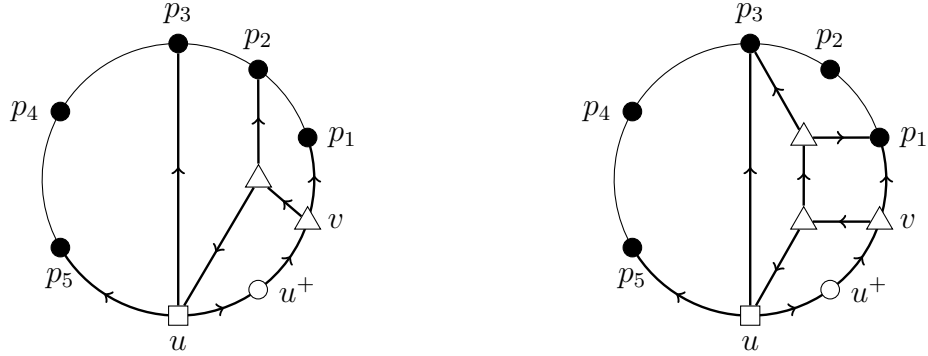
\begin{figure}[!htbp]
\centering
\begin{subfigure}[b]{.45\linewidth}
\centering
\begin{tikzpicture}[scale = .9, decoration={markings, mark=at position 0.55 with {\arrow{>}}}]

\draw (18:2) arc (18:210:2);
\draw[postaction={decorate}, line width=.9pt] (270:2) arc (270:210:2);
\draw[postaction={decorate}, line width=.9pt] (-90:2) arc (-90:-54:2);
\draw[postaction={decorate}, line width=.9pt] (-54:2) arc (-54:-18:2);
\draw[postaction={decorate}, line width=.9pt] (-18:2) arc (-18:18:2);

\node[bdot, label={right:$p_1$}] at (18:2) (p1) {};
\node[bdot, label={above:$p_2$}] at (54:2) (p2) {};
\node[bdot, label={above:$p_3$}] at (90:2) (p3) {};
\node[bdot, label={left:$p_4$}] at (150:2) (p4) {};
\node[bdot, label={below:$p_5$}] at (210:2) (p5) {};
\node[triangle, label={right:$v$}] at (-18:2) (v) {};
\node[wdot, label={right:$u^{+}$}] at (-54:2) (u+) {};
\node[square, label={below: $u$}] at (270:2) (u) {};
\node[triangle] at (0:1.18) (w) {};

\draw[-, postaction={decorate}, line width=.9pt] (u)--(p3);
\draw[-, postaction={decorate}, line width=.9pt] (w)--(u);
\draw[-, postaction={decorate}, line width=.9pt] (w)--(p2);
\draw[-, postaction={decorate}, line width=.9pt] (v)--(w);

\end{tikzpicture}
\end{subfigure}
\begin{subfigure}[b]{.45\linewidth}
\centering
\begin{tikzpicture}[scale = .9, decoration={markings, mark=at position 0.55 with {\arrow{>}}}]

\draw (18:2) arc (18:210:2);
\draw[postaction={decorate}, line width=.9pt] (270:2) arc (270:210:2);
\draw[postaction={decorate}, line width=.9pt] (-90:2) arc (-90:-54:2);
\draw[postaction={decorate}, line width=.9pt] (-54:2) arc (-54:-18:2);
\draw[postaction={decorate}, line width=.9pt] (-18:2) arc (-18:18:2);

\node[bdot, label={right:$p_1$}] at (18:2) (p1) {};
\node[bdot, label={above:$p_2$}] at (54:2) (p2) {};
\node[bdot, label={above:$p_3$}] at (90:2) (p3) {};
\node[bdot, label={left:$p_4$}] at (150:2) (p4) {};
\node[bdot, label={below:$p_5$}] at (210:2) (p5) {};
\node[triangle, label={right:$v$}] at (-18:2) (v) {};
\node[wdot, label={right:$u^{+}$}] at (-54:2) (u+) {};
\node[square, label={below: $u$}] at (270:2) (u) {};
\node[triangle] at (-38.17:1) (w) {};
\node[triangle] at (38:1) (x) {};

\draw[-, postaction={decorate}, line width=.9pt] (u)--(p3);
\draw[-, postaction={decorate}, line width=.9pt] (w)--(u);
\draw[-, postaction={decorate}, line width=.9pt] (v)--(w);
\draw[-, postaction={decorate}, line width=.9pt] (x)--(p3);
\draw[-, postaction={decorate}, line width=.9pt] (x)--(p1);
\draw[-, postaction={decorate}, line width=.9pt] (w)--(x);

\end{tikzpicture}
\end{subfigure}
\caption{Graphs $G$ whose subgraphs $G-E(P)$ are manually proven to be list-colorable in Lemma~\ref{lem:chords} by directing the edges in $G-E(P)$ as illustrated and applying Alon--Tarsi.}
\label{fig:chords}
\end{figure}

\section{Peeling off vertices}\label{sec3}
Denote $B(G)=p_k\dots p_1 v_1\dots v_s$.
\begin{claim} \label{cla:v,p}
The vertex $v_1$ has no common inner neighbor with $p_i$ for $i \geq 3$. The vertex $v_s$ has no common inner neighbor with $p_i$ for $i \leq k-2$.    
\end{claim}

\begin{proof} 
If $u$ is adjacent to $v_1, p_5$, then $W=v_1up_5$ is a $(1,1^+)$-$2$-chord with $|P \cap B(G_{W,2})|=1$, contradicting Lemma~\ref{lem:f-chords}.

If $u$ is adjacent to $v_1,p_3$, then we have a $5$-cycle $C=v_1p_1p_2p_3u$. Since there is no separating $5$-cycle, $\int(C)$ is empty. We consider the graph $G'$ obtained by removing the vertex $p_2$ from $G$ and adding the edge $e=p_1p_3$. Let $L'(v)=L(v)$ for $v\in G'$ and $P'=p_1p_3p_4p_5$. It is easy to see that $(G',P',L')$ is a valid target. By the minimality of $G$, $G'-E(P')$ is $L'$-colorable, and hence, $G-E(P)$ is $L$-colorable, a contradiction.

If $u$ is adjacent to $v_1,p_4$, then we have a $6$-cycle $C=v_1p_1p_2p_3p_4u$. If $\int(C)$ is empty, consider the graph $G'$ obtained by removing the vertex $p_3$ from $G$ and adding the edge $e=p_2p_4$. Let $L'(v)=L(v)$ for $v\in G'$ and $P'=p_1p_2p_4p_5$. Once again, it is easy to see that $(G',P',L')$ is a valid target. By the minimality of $G$, $G'-E(P')$ is $L'$-colorable, and hence, $G-E(P)$ is $L$-colorable, a contradiction.

If $\int(C)$ is nonempty, then by Lemma~\ref{lem:sep6}, $\int(C) = \{x\}$ or $\{x,y\}$ with $x,y \in X_{G,L}$. If $\int(C) =\{x\}$, then either $x$ is adjacent to $p_1,p_3$ or $x$ is adjacent to $p_2,p_4$. In either case, there is a $(1,1)$-$2$-chord, contradicting Lemma \ref{lem:112chord}. If $\int(C) =\{x,y\}$, since there is no $(1,1)$-$2$-chord, we must have $N(x)=\{v_1,p_2,y\}$ and $N(y)=\{x, p_3, u\}$. By the distance condition on $X_{G,L}$, we have $|L(u)|=4$ and the $2$-chord $W=v_1up_4$ is feasible as $4 = |L(u)|\leq d(u)$. By the minimality of $G$, $G'_{W,1}$ has an $L$-coloring say $\phi_1$. Let $L'$ be the restriction of $L$ to $G_{W,2}$ with $L'(v) = \{\phi_1(v)\}$ if $v \in P_W$ and $L'(v)=L(v)$ otherwise. As $W$ is feasible, we have $P_W = p_5p_4uv_1v_2$ with $|L(v_2)| = 2$ by Lemma~\ref{lem:feasible}. It is easy to see that $(G_{W,2},P_W,L')$ has no bad edge, no bad cycle, no worse cycle, and no worst cycle. By Lemma~\ref{lem:badvertex}, $(G_{W,2},P_W,L')$ has a bad vertex $w$ with $|L(w)| = 3$ adjacent to at least two vertices on $P_{W}$. By the distance condition of $X_{G,L}$, $w$ must lie on the boundary of $G$. To avoid creating a chord in Lemma~\ref{lem:f-chords}, $w$ is neither adjacent to $v_1$ nor to $p_4$. If $w$ is adjacent to $p_5$ and $v_2$ on the boundary without being adjacent to $u$, $P_{W} \cup \{w\}$ would form a separating $6$-cycle contradicting the distance condition of $X_{G,L}$. Thus, $w$ is adjacent to $u$ and is consecutive on the boundary to at least one of the vertices $v_2,p_5$. If $w$ is consecutive on the boundary to $v_2$, let $v_1v_2w_0w_1w_2\dots w_{2k}$ be the maximal subpath of $B(G)$ such that $|L(w_{2i})| = 3$ for all $0 \leq i \leq k$, $|L(w_{2i-1})| = 2$ for all $1 \leq i \leq k$, and $u$ is adjacent to all $w_{2i}$ for all $0\leq i \leq k$ where $w_0 = w$. If $w_{2k}$ is adjacent to $p_5$ then $w_{2k}p_5$ is an edge on the boundary, and $G$ has no other vertices, see Figure~\ref{fig:claim}. To show the bipartite graph $G-E(P)$ is $L$-colorable, it suffices by the Alon--Tarsi Theorem~\cite{Alon1992} to find an orientation on $G-E(P)$ such that the outdegree of every vertex is strictly less than its list size. Such an orientation is shown in Figure~\ref{fig:claim}.

Therefore, $w_{2k}$ is not adjacent to $p_5$ and the chord $W' = p_4uw_{2k}$ is feasible. By Lemma~\ref{lem:feasible}, $P_{W'} = p_5p_4uw_{2k}w_{2k}^+$ where $w_{2k}^+ \not = w_{2k-1}$ is a boundary vertex consecutive to $w_{2k}$ and $|L(w_{2k}^+)| = 2$. Moreover, there must exist a bad vertex $z$ such that $|L(z)| = 3$ and $|N_{G}(z) \cap P_{W'}| \geq 2$. By the same logic applied to the vertex $w$, the vertex $z$ must be a boundary vertex adjacent to $u$. Furthermore, by the maximality of $k$, $z$ must be consecutive to $p_5$ on the boundary. Let $p_5z_0z_1\dots z_{2\ell}$ be the maximal subpath of $B(G)$ such that $|L(z_{2i})| = 3$ for all $0 \leq i \leq \ell$, $|L(z_{2i-1})| = 2$ for all $1 \leq i \leq \ell$, and $u$ is adjacent to all $z_{2i}$ for all $0\leq i \leq \ell$ where $z_0 = z$. As $z_{2\ell}$ is not adjacent to $w_{2k}^+$, the chord $W'' = z_{2\ell} u w_{2k}$ is feasible. By Lemma~\ref{lem:feasible}, $P_{W''} = z_{2\ell}^-z_{2\ell}uw_{2k}w_{2k}^+$ where $z_{2\ell}^- \not = z_{2\ell-1}$ is a boundary vertex consecutive to $z_{2\ell}$ and $|L(z_{2\ell}^-)| = 2$. There must then exist a bad vertex $\alpha$ with $|L(\alpha)| = 3$ adjacent to at least two vertices on $P_{W''}$. By the distance condition on $X_{G,L}$, $\alpha$ must lie on the boundary. As there are no $(3,3)$-$1$-chords, $\alpha$ is not adjacent to $w_{2k}$ nor to $z_{2\ell}$. Moreover, by the maximality of $k$ and $\ell$, $\alpha$ is not adjacent to $u$. Therefore, $z_{2\ell}^-uw_{2k}^+$ forms a subpath of the boundary, and $P_{W''} \cup \{\alpha\}$ forms a separating $6$-cycle. However, this creates an interior vertex with list size $3$ contradicting the distance condition on $X_{G,L}$. Similarly, $w$ cannot be adjacent to $u$ while being consecutive on the boundary to $p_5$. Therefore, $(G_{W,2},P_W,L')$ is a valid target and has an $L'$-coloring $\phi_2$. The union of $\phi_1$ and $\phi_2$ gives an $L$-coloring of $G-E(P)$, a contradiction. 
\end{proof}

\begin{figure}[!htbp]
\centering
\begin{tikzpicture}[scale = .9, decoration={markings, mark=at position 0.6 with {\arrow{>}}}]

\draw (90:2) arc (90:270:2);
\draw[postaction={decorate}, line width=.9pt] (-67.5:2) arc (-67.5:-90:2);
\draw[postaction={decorate}, line width=.9pt] (-45:2) arc (-45:-67.5:2);
\draw[postaction={decorate}, line width=.9pt] (-22.5:2) arc (-22.5:-45:2);
\draw[postaction={decorate}, line width=.9pt] (0:2) arc (0:-22.5:2);
\draw[postaction={decorate}, line width=.9pt] (22.5:2) arc (22.5:0:2);
\draw[postaction={decorate}, line width=.9pt] (67.5:2) arc (67.5:90:2);
\draw[ postaction={decorate, line width = 1pt}, line width=1.5pt, line cap=round, dash pattern=on 0pt off 4.2\pgflinewidth] (57:2) arc (57:29:2);

\node[bdot, label={below:$p_1$}] at (270:2) (p1) {};
\node[bdot, label={below:$p_2$}] at (225:2) (p2) {};
\node[bdot, label={left:$p_3$}] at (180:2) (p3) {};
\node[bdot, label={above:$p_4$}] at (135:2) (p4) {};
\node[bdot, label={above:$p_5$}] at (90:2) (p5) {};
\node[triangle, label={[label distance = -.1cm]-20:$v_1$}] at (-67.5:2) (v1) {};
\node[wdot, label={[label distance = -.1cm]-10:$v_2$}] at (-45:2) (v2) {};
\node[triangle, label={right:$w_0$}] at (-22.5:2) (w0) {};
\node[wdot, label={right:$w_1$}] at (0:2) (w1) {};
\node[triangle, label={right: $w_2$}] at (22.5:2) (w2) {};
\node[triangle, label={right: $w_{2k}$}] at (67.5:2) (w2k) {};
\node[square, label={[label distance = -.17cm]230:$u$}] at (0:0) (u) {};
\node[triangle, label={below: $x$}] at (225:1.27) (x) {};
\node[triangle, label={[label distance = -.1cm]90:$y$}] at (180:.9) (y) {};

\draw[-, postaction={decorate}, line width=.9pt] (u)--(y);
\draw[-, postaction={decorate}, line width=.9pt] (u)--(w2k);
\draw[-, postaction={decorate}, line width=.9pt] (u)--(p4);
\draw[-, postaction={decorate}, line width=.9pt] (w2)--(u);
\draw[-, postaction={decorate}, line width=.9pt] (w0)--(u);
\draw[-, postaction={decorate}, line width=.9pt] (y)--(x);
\draw[-, postaction={decorate}, line width=.9pt] (y)--(p3);
\draw[-, postaction={decorate}, line width=.9pt] (x)--(p2);
\draw[-, postaction={decorate}, line width=.9pt] (x)--(v1);
\draw[-, postaction={decorate}, line width=.9pt] (v1)--(u);

\end{tikzpicture}
\caption{Graph $G$ whose subgraph $G-E(P)$ is proven to be list-colorable in Claim~\ref{cla:v,p} by orientating the edges of $G-E(P)$ as illustrated and applying the Alon--Tarsi Theorem.}
\label{fig:claim}
\end{figure}
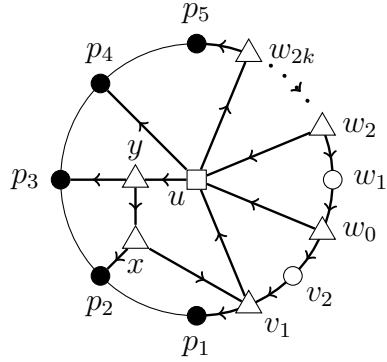

For each following case, we color vertices close to the vertex $p_1$ and describe the remaining uncolored subgraph.
At the end of each case, we stop at a configuration that leads to some bad subgraph due to the fact that some edge $xy\in X_{G,L}$ is close to the boundary.
We will give an upper bound on the distance from that edge to the precolored path $P$.
After finding a configuration for each case, we will argue that any combination of two configurations cannot occur in the next section.

\textbf{Case 1:} $s=0$ or $s\ge 2$ and $|L(v_2)|\ge 3$.
We color $p_1$ by the only color $\alpha$ in its list $L(p_1)$ and denote $G'=G-p_1$, $P'=P-p_1$, and $L'$ the list assignment on $G'$ with $L'(v)=L(v)-\alpha$ if $v\in N(p_1)-P$ and $L'(v)=L(v)$ otherwise. By Lemma~\ref{lem:f-chords} and~\ref{lem:112chord}, $(G',P',L')$ has no bad vertex. Since $(G,P,L)$ has no bad vertex, $|L'(v_1)|\ge |L(v_1)|-1\ge 2$. As $|L'(v_2)|=|L(v_2)|\ge 3$ and $G$ has no $(2,3^+)$-$1$-chords, $v_1$ is not part of a bad edge in $(G',P',L')$. Furthermore, $G$ having no $(1,2)$-$2$-chords by Corollary~\ref{cor:122chord} implies that $(G',P',L')$ has no bad edge. Since $(G,P,L)$ has no worse cycle and does not have any $(1,2^+)$-$2$-chords $W$ with $|P\cap B(G_{W,2})| = 1$, $(G',P',L')$ has no bad cycle. Similarly, $(G',P',L')$ has no worst cycle. If $(G',P',L')$ has a worse cycle $C=xyzw$ where $z,w$ are interior vertices and $x,y$ are boundary vertices of $G'$, then in the original graph $G$, $z$ and $w$ must be interior vertices, $v_1 = x$, and $v_2 = y$.
Thus, we have the configuration as in Figure~\ref{fig:case1}.
In this configuration, the distance between $wz$ and $P$ is at most $2$.

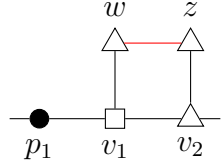
\begin{figure}[!htbp]
\centering
\begin{tikzpicture}
\node[invisnode] at (-.5,0) (a) {};
\node[bdot, label={below:$p_1$}] at (0,0) (p1) {};
\node[square,label={below:$v_1$}] at (1,0) (v1) {};
\node[triangle,label={below:$v_2$}] at (2,0) (v2) {};
\node[invisnode] at (2.5,0) (b) {};
\node[triangle,label={above:$w$}] at (1,1) (w) {};
\node[triangle,label={above:$z$}] at (2,1) (z) {};
\draw[black] (a) --(p1);
\draw[black] (p1)--(v1);
\draw[black] (v1)--(v2);
\draw[black] (v2)--(b);
\draw[black] (v1)--(w);
\draw[black] (v2) --(z);
\draw[red]   (w) --(z);
\end{tikzpicture}
\caption{The end configuration in Case 1 that produces a worse cycle $v_1v_2zw$.}
\label{fig:case1}
\end{figure}

\textbf{Case 2:} $s=1$ or ($s\ge 2$ and (($|L(v_2)|=2$ and $L(v_1)-(L(v_2)\cup L(p_1))\neq\emptyset$).

We color $v_1$ by a color $\alpha\in L(v_1)\setminus (L(p_1)\cup L(p_k))$ if $s=1$ or $\alpha\in L(v_1)\setminus (L(v_2)\cup L(p_1))$ if $s\ge 2$.
Denote $G'=G-v_1$ and $L'$ the list assignment on $G'$ with $L'(v)=L(v)-\alpha$ if $v\in N(v_1)$ and $L'(v)=L(v)$ otherwise.

As $G$ has no $(2,3^+)$-$1/2$-chords by Lemma~\ref{lem:f-chords}, $(G',P,L')$ has no bad edge and no worst cycle. Suppose that $(G',P,L')$ has a bad cycle $C=xyzw$ where $w$ is the interior vertex and $|L'(y)|=2$. If $y$ is an interior vertex of $G$, then $y$ is adjacent to $v_1$ with $|L(y)| = 3$ and $x$ is a boundary vertex. However, as $G$ has no bad cycle, the chord $v_1yx$ is a feasible $(3,3^+)$-$2$-chord with $|L(y)|= 3$ contradicting Lemma~\ref{lem:332chord}. Then it must be the case that $y$ is a boundary vertex of $G$. If $y = v_2$, then $xyzw$ forms a worst cycle in $G$, and if $y \not = v_2$, then there exists a $(2,3^+)$-$2$-chord or a worst cycle in $G$. As both of these are impossible, $(G',P,L')$ has no bad cycle.

Suppose that $(G',P,L')$ has a worse cycle $C=xyzw$ where $z,w$ are interior vertices of $G'$ and $x,y$ are boundary vertices. Then, without loss of generality, we may assume that between $x,y$ only $x$ is an interior vertex of $G$ adjacent to $v_1$ with $|L(x)|=4$ and $|L(y)|=3$. If all of $y,z,w$ are also interior vertices of $G$, then $y,z,w \in X_{G,L}$, which is impossible. If $y$ is on the boundary, this produces the configuration seen in Figure~\ref{fig:case2}. In this configuration, the distance between $xw$ and $P$ is at most $3$.

\begin{figure}[!htbp]
\centering
\begin{tikzpicture}
\node[bdot, label={below:$p_1$}] at (0,0) (p1) {};
\node[triangle,label={below:$v_1$}] at (1.25,0) (v1) {};
\node[wdot,label={below:$v_2$}] at (2.5,0) (v2) {};
\node[triangle,label={below:$y$}] at (3.75,0) (y) {};
\node[square, label={above:$x$}] at (1.25,1) (x) {};
\node[triangle,label={above:$w$}] at (2.5,1) (w) {};
\node[triangle,label={above:$z$}] at (3.75,1) (z) {};
\node[invisnode] at (-.5,0) (a) {};
\node[invisnode] at (2.9,0) (b) {};
\node[invisnode] at (3.35,0) (c) {};
\node[invisnode] at (4.25,0) (d) {};

\draw[black] (a) --(p1)--(v1)--(v2)--(b);
\draw[black] (c)--(y)--(z);
\draw[red] (w)--(z);
\draw[black] (w)--(x)--(v1);
\draw[black] (d)--(y)--(x);
\draw[black] (p1)--(.375,.3);
\draw[black] (x)--(.875,.7);
\draw[black, line width=1.2pt, line cap=round, dash pattern=on 0pt off 4\pgflinewidth] (2.95,0)--(3.3,0);
\draw[black, line width=1.2pt, line cap=round, dash pattern=on 0pt off 4\pgflinewidth] (.4875,.39)--(.8125,.65);
\end{tikzpicture}
\caption{The end configuration of Case 2 that produces a worse cycle $xyzw$. Note that $|L(v_1)| = 3$ or $4$. }
\label{fig:case2}
\end{figure}
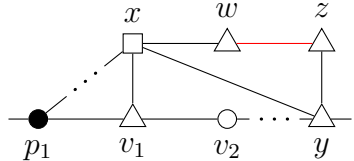

Suppose that $(G',P,L')$ has a bad vertex $x$. By Claim~\ref{cla:v,p}, $x$ must be adjacent to $v_1$ and $p_2$ in $G$ with $|L(x)| = 3$; moreover, $x$ is an interior vertex of $G$ as $G$ has no chords. Since $G$ has no separating $4$-cycles, the interior of the cycle $p_1v_1xp_2$ is empty. Thus, we may consider the subgraph $G'' = G'-\{p_1\}$ with list $L''$ such that $x$ is precolored by a color $\beta$ in $L'(x)-L'(p_2)$ and $L''(v) = L'(v)$ otherwise and $P''=P\cup\{x\}-p_1$. 
 Clearly $(G'',P'',L'')$ has no bad edges and no bad or worst cycle. Moreover, if $(G',P,L')$ has no worse cycle, then $(G'',P'',L'')$ also has no worse cycle. Thus, we may assume $G''$ has a bad vertex $y$. If $y$ is a boundary vertex of $G$, then $yxv_1$ is either a $(2,3^+)$-$2$-chord or a feasible $(3,3^+)$-$2$-chord with $|L(x)| = 3$, a contradiction. Hence, $y$ is an interior vertex of $G$. Note that $y$ cannot be adjacent to $v_1$; otherwise, $y$ must also be adjacent to $p_2$ creating the separating $4$-cycle $yp_2p_1v_1$ containing the vertex $x$. Furthermore, as $G$ has no $(1,1)$-$2$-chord, $|L''(y)| = |L(y)| = 3$ implying $y$ is adjacent to $x$ and exactly one of $p_3$, $p_4$, or $p_5$. If $y$ is adjacent to $x$ and $p_5$, then $p_5yxv_1$ is a feasible $(1,3^+)$-$3$-chord meeting the conditions of Lemma~\ref{lem:233chord} (while in general $G$ could have a feasible $(1,4)$-$3$-chord meeting the conditions of Lemma~\ref{lem:233chord}, the proof of the lemma holds for this particular chord). If $y$ is adjacent to $x$ and $p_4$, then $p_4p_3p_2xy$ is a $5$-cycle. Consider the graph $\tilde{G}$ obtained 
 from $G$ by deleting $p_3$ and adding in the edge $p_2 p_4$ with $\tilde{P} = P-\{p_3\}$.
  Then $(\tilde{G}, \tilde{P}, \tilde{L})$ is a valid target where $\tilde{L}$ is the resulting list assignment of $\tilde{G}$ implying there exists a coloring of $\tilde{G}-E(\tilde{P})$ which extends to a coloring of $G-E(P)$. Therefore, $y$ must be adjacent to $x$ and $p_3$. In this case, consider the graph $G'''$ obtained from $G''$ by deleting $p_2$ and precoloring $y$ with a color in $L''(y)-\{\beta, L(p_3)\}$.

As before $G'''$ has no bad edges and no bad or worst cycle, and only has a worse cycle if $(G',P,L')$ has one. Then it must be the case that $(G''',P''',L''')$ has a bad vertex $z$. By the definition of $X_{G,L}$ and the absence of $(1,1)$-$2$-chords in $G$, $z$ cannot be an interior vertex of $G$. If $z$ is a boundary vertex, then it must be adjacent to $y$ as $G''$ does not have a bad vertex on the boundary. By Lemmas~\ref{lem:233chord} and~\ref{lem:chords}, $|L(z)| = 3$ and $z$ is adjacent to $y$ and $p_5$ with the edge $zp_5$ lying on the boundary of $G$. However, this implies $v_s = z$ contradicting Claim~\ref{cla:v,p} that $v_s$ shares no common interior neighbor with $p_3$ if $|P| = 5$. Therefore, $(G',P,L')$ has no bad vertex if it has no worse cycle.

\textbf{Case 3:} $s\ge 2$ and ($|L(v_2)|=2$ and $L(v_1)-(L(v_2)\cup L(p_1))=\emptyset$) and ($L(v_2) \not\subseteq L(v_3)$ or $|L(v_4)| \geq 3$)

If $L(v_2)\not\subseteq L(v_3)$, then we color $v_2$ by $\alpha_2\in L(v_2) - L(v_3)$.
If $|L(v_4)|\ge 3$, then we color $v_2$ by some $\alpha_2\in L(v_2)$.
In either case, we color $v_1$ by some $\alpha_1\in L(v_1)-(L(p_1)\cup\{\alpha_2\})$, and for $i=1,2$, delete the color $\alpha_i$ from the lists of the neighbors of $v_i$'s. Let $G'=G-\{v_1,v_2\}$ and $L'$ be the resulting list.

Assume that $(G',P,L')$ has a bad edge $xy$ with $|L'(x)| = |L'(y)| = 2$. If $x$ is adjacent to $v_1$ and $y$ is adjacent to $v_2$, then $xyv_2v_1$ is either a bad cycle if $v_3 = y$ or a worst cycle if $x$ and $y$ are interior vertices since $G$ has no chords. Thus, without loss of generality $y$ is not adjacent to $v_1$ or $v_2$; however, this creates a $(2^+,3)$-$2$-chord contradicting Lemma~\ref{lem:f-chords}. 

Assume that $(G',P,L')$ has a bad vertex $x$. As $G$ has no chords and no $(1,1)$-$2$-chords, $x$ must be an interior vertex of $G$ adjacent to exactly one of $v_1$ or $v_2$ and exactly one precolored vertex. By Corollary~\ref{cor:122chord} and Claim~\ref{cla:v,p}, $x$ must be adjacent to $v_1$ and $p_2$. However, by a similar argument to Case 2, if $(G',P,L')$ has no bad, worse, or worst cycle, then $(G',P,L')$ has no such bad vertex $x$.

Assume that $(G',P,L')$ has a bad cycle $xyzw$ such that $y$ is the boundary vertex with $|L'(y)|=2$ and $w$ is the interior vertex with $|L'(w)|=3$. If $y$ is an interior vertex of $G$, then it must be adjacent to exactly one of $v_1$ or $v_2$. In this case, $|L(y)| = |L(w)| = 3$ implying that one of $x$ or $w$ must lie on the boundary of $G$ by the definition of $X_{G,L}$. However, this would imply $G$ contains either a $(2,3)$-$2$-chord or a feasible $(3,3)$-$2$-chord contradicting Lemmas~\ref{lem:f-chords} and~\ref{lem:332chord} respectively. If $y$ is a boundary vertex that is not $v_3$, then $v_3 = x$ with $|L(x)| = 4$, $v_4 = y$, and $v_5 = z$ since $G$ has no worst cycle. However, as $|L(v_4)| = |L(y)| = 2$, we must have $L(v_2) \not\subseteq L(v_3) = L(x)$ and therefore $|L'(x)| = 4$. Thus, $y = v_3$, $z = v_4$, and $x$ is an interior vertex of $G$ with $|L(x)| = 4$ such that $x$ is adjacent to $v_1$ and $v_3$. Since $|L'(v_3)| \not = |L(v_3)|$, $L(v_2) \subseteq L(v_3)$ by construction. Note that as $G$ has no separating $4$-cycles, the only neighbors of $v_3 = y$ are $x$, $v_4 = w$, and $v_2$. Consider the graph $\tilde{G}$ obtained from $G$ by identifying the vertices $v_1$ and $v_3$ into a single vertex $v$. Let $\tilde{L}$ be the resulting list with $\tilde{L}(v) = L(v_1)$. Note that $\tilde{G}$ is still $C_3$-free and $(\tilde{G}, P, \tilde{L})$ is a valid target. Hence, by minimality of $G$, there exists a valid $\tilde{L}$-coloring $\varphi$ of $\tilde{G}-P$. Since $\varphi(v) \in L(v_1)-L(p_1) = L(v_2) \subseteq L(v_3)$, we may extend $\varphi$ to an $L$-coloring of $G-P$ by coloring both $v_1$ and $v_3$ the color $\varphi(v)$.

Assume that $(G',P,L')$ has a worst cycle $xyzw$ where $x,y$ are boundary vertices of $G'$ such that $|L'(x)| = 2$ and $|L'(y)| = 3$. By the definition of $X_{G,L}$, $x$ is a boundary vertex of $G$; moreover, by Lemma~\ref{lem:f-chords}, $|L(x)| = 3$. Thus, $v_3 = x$ and $y$ is an interior vertex of $G$ adjacent to $v_1$ and $v_3$ with $|L(y)| = 4$. Since $|L'(v_3)| \not = |L(v_3)|$, $L(v_2) \subseteq L(v_3)$ implying $|L(v_4)| \geq 3$. Once again we may consider the graph obtained from $G$ by identifying the vertices $v_1$ and $v_3$ whose coloring we may extend as in the discussion of bad cycles. 

Assume $(G',P,L')$ has a worse cycle $C=xyzw$ with $x,y$ boundary vertices of $G'$. At least one $x$ or $y$ must be adjacent to one of $v_1$ or $v_2$; moreover, if this vertex is an interior vertex of $G$, then it must have list size $4$. If $x$ is adjacent to $v_1$ and $y$ is an interior vertex of $G$, then $|L(x)|=|L(y)|=4$ and $y$ is adjacent to $v_2$. This produces one possible end configuration, as shown in Figure~\ref{fig:case31}. If $x$ is adjacent to $v_1$ and $y$ is a boundary vertex with $|L(y)|=4$, then $y = v_3$ and we may once again identify $v_1$ with $v_3$ and apply the same argument as when discussing bad cycles. If $x$ is adjacent to $v_1$ and $y$ is a boundary vertex with $|L(y)|=3$, this produces another possible end configuration, as shown in Figure~\ref{fig:case31}. Thus, we may now assume neither $x$ nor $y$ is adjacent to $v_1$. If $x$ is adjacent to $v_2$, then $v_3 = x$ and $v_4 = y$ with $|L(x)| = |L(y)|+1 = 4$ as $G$ has no $(2,3)$-$2$-chords. This creates another possible end configuration, as shown in Figure~\ref{fig:case31}. In all of these possible configurations, the distance from $wz$ to the precolored path $P$ is at most $4$.

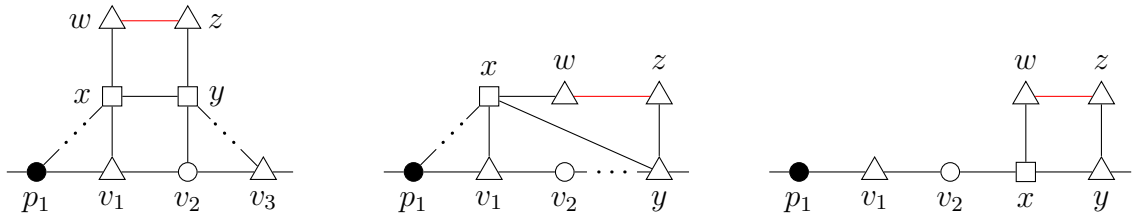
\begin{figure}[!htbp]
\centering
\begin{tikzpicture}
\node[invisnode] at (-.5,0) (a) {};
\node[bdot, label={below:$p_1$}] at (0,0) (p1) {};
\node[triangle,label={below:$v_1$}] at (1,0) (v1) {};
\node[wdot,label={below:$v_2$}] at (2,0) (v2) {};
\node[triangle,label={below:$v_3$}] at (3,0) (v3) {};
\node[invisnode] at (3.5,0) (b) {};
\node[square,label={left:$x$}] at (1,1) (x) {};
\node[square,label={right:$y$}] at (2,1) (y) {};
\node[triangle,label={left:$w$}] at (1,2) (w) {};
\node[triangle,label={right:$z$}] at (2,2) (z) {};

\draw[black] (a)--(p1)--(v1)--(v2)--(v3)--(b);
\draw[black] (v1)--(x)--(w);
\draw[black] (v2)--(y)--(z);
\draw[black] (x)--(y);
\draw[red] (w)--(z);
\draw[black] (p1)--(0.3,0.3);
\draw[black] (x)--(0.7,0.7);
\draw[black] (v3)--(2.7,0.3);
\draw[black] (y)--(2.3,0.7);
\draw[black, line width=1.2pt, line cap=round, dash pattern=on 0pt off 4\pgflinewidth] (.39,.39)--(.65,.65);
\draw[black, line width=1.2pt, line cap=round, dash pattern=on 0pt off 4\pgflinewidth] (2.61,.39)--(2.35,0.65);
\end{tikzpicture}
\hspace{.5cm}
\begin{tikzpicture}
\node[invisnode] at (-.5,0) (a) {};
\node[bdot, label={below:$p_1$}] at (0,0) (p1) {};
\node[triangle,label={below:$v_1$}] at (1,0) (v1) {};
\node[wdot,label={below:$v_2$}] at (2,0) (v2) {};
\node[invisnode] at (2.4,0) (b) {};
\node[invisnode] at (2.85,0) (c) {};
\node[triangle,label={below:$y$}] at (3.25,0) (y) {};
\node[invisnode] at (3.75,0) (d) {};
\node[square,label={above:$x$}] at (1,1) (x) {};
\node[triangle,label={above:$w$}] at (2,1) (w) {};
\node[triangle,label={above:$z$}] at (3.25,1) (z) {};

\draw[black] (a)--(p1)--(v1)--(v2)--(b);
\draw[black] (c)--(y)--(d);
\draw[black] (v1)--(x)--(w);
\draw[red] (w)--(z);
\draw[black] (p1)--(0.3,0.3);
\draw[black] (x)--(0.7,0.7);
\draw[black] (z)--(y)--(x);
\draw[black, line width=1.2pt, line cap=round, dash pattern=on 0pt off 4\pgflinewidth] (2.45,0)--(2.85,0);
\draw[black, line width=1.2pt, line cap=round, dash pattern=on 0pt off 4\pgflinewidth] (.39,.39)--(.65,.65);
\end{tikzpicture}\hspace{.5cm}
\begin{tikzpicture}
\node[invisnode] at (-.5,0) (a) {};
\node[bdot, label={below:$p_1$}] at (0,0) (p1) {};
\node[triangle,label={below:$v_1$}] at (1,0) (v1) {};
\node[wdot,label={below:$v_2$}] at (2,0) (v2) {};
\node[square,label={below:$x$}] at (3,0) (x) {};
\node[triangle,label={below:$y$}] at (4,0) (y) {};
\node[invisnode] at (4.5,0) (b) {};
\node[triangle,label={above:$w$}] at (3,1) (w) {};
\node[triangle,label={above:$z$}] at (4,1) (z) {};

\draw[black] (a)--(p1)--(v1)--(v2)--(x)--(y)--(b);
\draw[black] (x)--(w);
\draw[black] (y)--(z);
\draw[red] (w)--(z);
\end{tikzpicture}
\caption{The end configurations in Case 3 that produce a worse cycle $xyzw$. }
\label{fig:case31}
\end{figure}

\textbf{Case 4:} ($s\ge 2$ and $|L(v_2)|=2$ and $L(v_1)-(L(v_2)\cup L(p_1))=\emptyset$) and ($L(v_2) \subseteq L(v_3)$ and $|L(v_4)| = 2$)

\begin{claim} \label{cla:v3,p}
In this case, $v_3$ and $p_i$ have no common neighbor $y$ with $|L(y)|=3$ for $i=1,2,4,5$,
\end{claim}

\begin{proof}
By Lemma~\ref{lem:132chord}, we may assume that $|L(v_3)| = 4$. If $y$ is adjacent to $v_3$ and $p_1$, then $p_1v_1v_2v_3y$ is a separating $5$-cycle in order to meet the degree conditions of $v_1$. If $y$ is adjacent to $v_3$ and $p_5$, then $p_5yv_3$ is a feasible chord with $|P_{W}| = 4$ contradicting Lemma~\ref{lem:feasible}.

Assume that $y$ is adjacent to $v_3$ and $p_2$. As $d(v_1) \geq 3$, we may assume $y$ is adjacent to $v_1$; otherwise, $p_2p_1v_1v_2v_3y$ is a separating $6$-cycle containing an internal vertex with neighbors $p_2,p_1,v_3$ which we can take to be $y$ instead. We consider the graph $\tilde{G}$ obtained by identifying the vertices $v_1$ and $v_3$ together into a single vertex $v$ and setting $\tilde{L}$ to be the resulting list with $\tilde{L}(v) = L(v_1)$. Since $v_1$ and $v_3$ share a neighbor, $\tilde{G}$ is $C_3$-free. Moreover, as $(\tilde{G}, P, \tilde{L})$ is a valid target, there exists a valid $\tilde{L}$-coloring $\varphi$ of $\tilde{G}-P$ which we can extend to a valid $L$-coloring of $G-P$ by coloring $v_1$ and $v_3$ the color $\varphi(v)$.

Assume that $y$ is adjacent to $v_3$ and $p_4$ and consider the chord $W = p_4yv_3$. We take $y$ such that $|V(G_{W,2})|$ is minimal. By Lemma~\ref{lem:feasible}, $P_W = p_5p_4yv_3v_4$, and by Lemma~\ref{lem:badvertex}, $G_{W,2}$ has a bad vertex $z$ with list size $3$ adjacent to at least two vertices on $P_W$. First, assume that $z$ is an interior vertex. By the minimality of $y$ and Lemma~\ref{lem:f-chords}, $z$ is adjacent to $y$ and exactly one of $p_5$ and $v_4$. If $z$ is adjacent to $y$ and $v_4$, consider the feasible chord $W' = p_4yzv_4$. There must exist a bad vertex $z'$ of $G_{W',2}$ with list size $3$ adjacent to two vertices on $P_{W'}$. Observe that $z'$ lies on the boundary of $G$ by definition of $X_{G,L}$. By our chord lemmas, $p_5z'v_4$ forms a subpath of the boundary of $G$ and $p_5p_4yzv_4z'$ is a separating $6$-cycle contradicting the distance condition of $X_{G,L}$. By a similar argument, $z$ cannot be adjacent to $y$ and $p_5$. Thus, $z$ must be a boundary vertex, and we may assume $G_{W,2}$ has no bad vertex on its interior. By the chord lemmas, $z$ is adjacent to $y$ and $p_5zv_4$ forms a subpath of the boundary of $G$. In this case, we may peel from the other side of $P$ and apply a symmetrical argument. Note that peeling from the other side does not result in an end configuration.
\end{proof}

If $L(v_4)=L(v_2)$, then we color $v_3$ by some $c\in L(v_3)-L(v_2)$ and remove $c$ from the lists of its neighbors to obtain a subgraph $G'$ and a new list assignment $L'$.
By Lemma~\ref{lem:chords}, there is no $(2,3)$-2-chord so $(G',P,L')$ contains no bad edge.
By Lemma~\ref{lem:332chord} and the fact that $G$ contains no worst cycle, $(G',P,L')$ contains no bad or worst cycle.
Assume that $(G',P,L')$ has a worse cycle $xyzw$ with $z$ and $w$ interior vertices of $G'$. As $G$ has no chords and no worse cycles, $x$ is an interior vertex of $G$ adjacent to $v_3$ with $|L(x)| = 4$ and $y$ is a boundary vertex of $G$ with $|L(y)| = 3$. This creates an end configuration as shown in Figure~\ref{fig:case4a}. In this configuration, the distance from $wz$ to the precolored path $P$ is at most $5$.

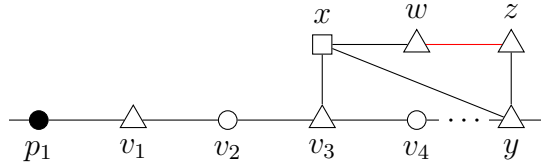
\begin{figure}[!htbp]
\centering
\begin{tikzpicture}
\node[bdot, label={below:$p_1$}] at (0,0) (p1) {};
\node[triangle,label={below:$v_1$}] at (1.25,0) (v1) {};
\node[wdot,label={below:$v_2$}] at (2.5,0) (v2) {};
\node[triangle,label={below:$v_3$}] at (3.75,0) (v3) {};
\node[wdot,label={below:$v_4$}] at (5,0) (v4) {};
\node[triangle,label={below:$y$}] at (6.25,0) (y) {};
\node[square, label={above:$x$}] at (3.75,1) (x) {};
\node[triangle,label={above:$w$}] at (5,1) (w) {};
\node[triangle,label={above:$z$}] at (6.25,1) (z) {};
\node[invisnode] at (-.5,0) (a) {};
\node[invisnode] at (5.4,0) (b) {};
\node[invisnode] at (5.85,0) (c) {};
\node[invisnode] at (6.75,0) (d) {};

\draw[black] (a)--(p1)--(v1)--(v2)--(v3)--(v4)--(b);
\draw[black] (c)--(y)--(z);
\draw[red] (w)--(z);
\draw[black] (w)--(x)--(v3);
\draw[black] (d)--(y)--(x);
\draw[black, line width=1.2pt, line cap=round, dash pattern=on 0pt off 4\pgflinewidth] (5.45,0)--(5.8,0);

\end{tikzpicture}
\caption{The end configuration of Case 4 that produces a worse cycle $xyzw$. Note that $|L(v_3)| = 3$ or $4$. }
\label{fig:case4a}
\end{figure}

If $(G',P,L')$ has a bad vertex, then its list size in $G$ cannot be of size $4$ as $G$ has no $(1,1)$-$2$-chords. Thus, any bad vertex must have list size $3$ in $G$ and be adjacent to $v_3$ and $p_3$ by Claim~\ref{cla:v3,p}. As $G$ has no separating $4$-cycles, note that $v_3$ and $p_3$ have at most two shared neighbors with list size $3$. If $(G',P,L')$ has a single bad vertex $y$, then we may precolor $y$ and partition $G'$ into two graphs $G_1'$ and $G_2'$ with precolored paths $P_1'=p_1p_2p_3y$ and $P_2'=yp_3p_4p_5$, respectively.
Let $L_i'$ be $L'$ restricted to $G_i'$ for $i=1,2$.
Then $(G_i',P_i',L_i')$ are both valid targets, and by induction hypothesis both have a proper coloring, then we may combine these two colorings to get a proper $L'$-coloring of $G'-P$, a contradiction. An analogous argument works if $(G',P,L')$ has two bad vertices $y$ and $z$.

If $L(v_4)\neq L(v_2)$, then we color $v_1,v_3$ by some same color $b\in L(v_2)-L(v_4)$ and color $v_2$ by the only choice $a\in L(v_2)-b$.
Let $G'=G-\{v_1,v_2,v_3\}$ and $L'$ be the restriction of $L$ to $G'$ with colors of $v_1,v_2,v_3$ removed from the lists of their neighbors.
If $v_1,v_3$ have a common neighbor $w$, then by the choice of $b$, $|L'(w)|\ge |L(w)|-1$ and hence $(G',P,L')$ remains a target.

If it has a bad edge, then there must be $w,z\in X_{G,L}$ such that each of $w,z$ is adjacent to one of $v_1,v_2,v_3$. Furthermore, as $G$ has no worst cycle, $z$ must be adjacent to $v_3$, and if $w$ is adjacent to $v_2$, then $|L(v_3)| = 4$, see Figure~\ref{fig:case4b1}. The distance between the edge $wz$ and $P$ is at most $3$.

\begin{figure}[!htbp] 
\centering
\begin{tikzpicture}
\node[invisnode] at (-.5,0) (a) {};
\node[bdot, label={below:$p_1$}] at (0,0) (p1) {};
\node[triangle,label={below:$v_1$}] at (1,0) (v1) {};
\node[wdot,label={below:$v_2$}] at (2,0) (v2) {};
\node[square,label={below:$v_3$}] at (3,0) (v3) {};
\node[wdot,label={below:$v_4$}] at (4,0) (v4) {};
\node[invisnode] at (4.5,0) (b) {};
\node[triangle,label={above:$w$}] at (2,1) (x) {};
\node[triangle,label={above:$z$}] at (3,1) (y) {};

\draw[black] (a)--(p1)--(v1)--(v2)--(v3)--(v4)--(b);
\draw[black] (v2)--(x);
\draw[black] (v3)--(y);
\draw[red] (x)--(y);

\draw[black] (p1)--(0.6,0.3);
\draw[black] (x)-- (1.4,0.7);
\draw[black] (y)--(3.3,0.7);
\draw[black] (v4)--(3.7,0.3);
\draw[black, line width=1.2pt, line cap=round, dash pattern=on 0pt off 4.5\pgflinewidth] (.84,0.42)--(1.28,0.64);
\draw[black, line width=1.2pt, line cap=round, dash pattern=on 0pt off 4\pgflinewidth] (3.61,0.39)--(3.35,0.65);

\end{tikzpicture}
\hspace{1cm}
\begin{tikzpicture}
\node[invisnode] at (-.5,0) (a) {};
\node[bdot, label={below:$p_1$}] at (0,0) (p1) {};
\node[triangle,label={below:$v_1$}] at (1,0) (v1) {};
\node[wdot,label={below:$v_2$}] at (2,0) (v2) {};
\node[triangle,label={below:$v_3$}] at (3,0) (v3) {};
\node[wdot,label={below:$v_4$}] at (4,0) (v4) {};
\node[invisnode] at (4.5,0) (b) {};
\node[triangle,label={above:$w$}] at (2,1) (x) {};
\node[triangle,label={above:$z$}] at (3,1) (y) {};

\draw[black] (a)--(p1)--(v1)--(v2)--(v3)--(v4)--(b);
\draw[black] (v1)--(x);
\draw[black] (v3)--(y);
\draw[red] (x)--(y);

\draw[black] (p1)--(0.6,0.3);
\draw[black] (x)-- (1.4,0.7);
\draw[black] (y)--(3.3,0.7);
\draw[black] (v4)--(3.7,0.3);
\draw[black, line width=1.2pt, line cap=round, dash pattern=on 0pt off 4.5\pgflinewidth] (.84,0.42)--(1.28,0.64);
\draw[black, line width=1.2pt, line cap=round, dash pattern=on 0pt off 4\pgflinewidth] (3.61,0.39)--(3.35,0.65);

\end{tikzpicture}
\caption{The end configurations in Case 4 that produce a bad edge $wz$. Note that in the second configuration $|L(v_3)| = 3$ or $4$.}
\label{fig:case4b1}
\end{figure}
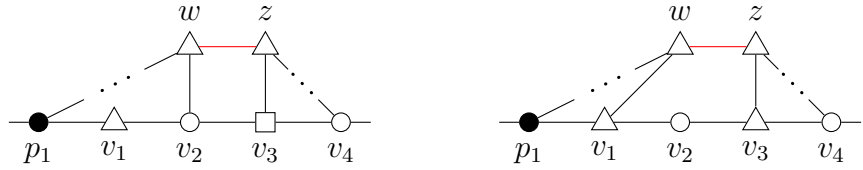

It is easy to see that $(G',P,L')$ has no worst cycle. Assume $(G',P,L')$ has a worse cycle $xyzw$ with $z,w$ being interior vertices and $x,y$ boundary vertices. If $x$ and $y$ are both interior vertices of $G$, then $|L(x)|=|L(y)|=4$ with each of $x,y$ adjacent to exactly one of $\{v_1,v_2,v_3\}$. Otherwise, $y$ is a boundary vertex of $G$ with $|L(y)| = 3$ and $x$ is an interior vertex with $|L(x)| = 4$ adjacent to exactly one of $\{v_1,v_3\}$.
The distance between the edge $wz$ and $P$ is at most $5$.

\begin{figure}[!htbp]
\centering
\begin{tikzpicture}[scale = .9]
\node[invisnode] at (-.5,0) (a) {};
\node[bdot, label={below:$p_1$}] at (0,0) (p1) {};
\node[triangle,label={below:$v_1$}] at (1,0) (v1) {};
\node[wdot,label={below:$v_2$}] at (2,0) (v2) {};
\node[triangle,label={below:$v_3$}] at (3,0) (v3) {};
\node[wdot,label={below:$v_4$}] at (4,0) (v4) {};
\node[invisnode] at (4.5,0) (b) {};

\node[square,label={left:$x$}] at (2,1) (x) {};
\node[square,label={right:$y$}] at (3,1) (y) {};
\node[triangle,label={left:$w$}] at (2,2) (w) {};
\node[triangle,label={right:$z$}] at (3,2) (z) {};

\draw[black] (a)--(p1)--(v1)--(v2)--(v3)--(v4)--(b);
\draw[black] (v1)--(x)--(y)--(v3);
\draw[black] (x)--(w);
\draw[black] (z)--(y);
\draw[red] (z)--(w);

\draw[black] (p1)--(0.6,0.3);
\draw[black] (x)-- (1.4,0.7);
\draw[black] (y)--(3.25,0.75);
\draw[black] (v4)--(3.75,0.25);
\draw[black, line width=1.2pt, line cap=round, dash pattern=on 0pt off 4.5\pgflinewidth] (.84,0.42)--(1.28,0.64);
\draw[black, line width=1.2pt, line cap=round, dash pattern=on 0pt off 4\pgflinewidth] (3.66,0.34)--(3.3,0.7);

\end{tikzpicture}
\hspace{.6cm}
\begin{tikzpicture}[scale = .9]
\node[invisnode] at (-.5,0) (a) {};
\node[bdot, label={below:$p_1$}] at (0,0) (p1) {};
\node[triangle,label={below:$v_1$}] at (1,0) (v1) {};
\node[wdot,label={below:$v_2$}] at (2,0) (v2) {};
\node[triangle,label={below:$v_3$}] at (3,0) (v3) {};
\node[wdot,label={below:$v_4$}] at (4,0) (v4) {};
\node[invisnode] at (4.5,0) (b) {};
\node[square,label={left:$x$}] at (2,1) (x) {};
\node[square,label={right:$y$}] at (3,1) (y) {};
\node[triangle,label={left:$w$}] at (2,2) (w) {};
\node[triangle,label={right:$z$}] at (3,2) (z) {};

\draw[black] (a)--(p1)--(v1)--(v2)--(v3)--(v4)--(b);
\draw[black] (v2)--(x)--(y)--(v3);
\draw[black] (x)--(w);
\draw[black] (z)--(y);
\draw[red] (z)--(w);

\draw[black] (p1)--(0.6,0.3);
\draw[black] (x)-- (1.4,0.7);
\draw[black] (y)--(3.25,0.75);
\draw[black] (v4)--(3.75,0.25);
\draw[black, line width=1.2pt, line cap=round, dash pattern=on 0pt off 4.5\pgflinewidth] (.84,0.42)--(1.28,0.64);
\draw[black, line width=1.2pt, line cap=round, dash pattern=on 0pt off 4\pgflinewidth] (3.66,0.34)--(3.3,0.7);

\end{tikzpicture}
\hspace{.6cm}
\begin{tikzpicture}[scale = .9]
\node[invisnode] at (-.5,0) (a) {};
\node[bdot, label={below:$p_1$}] at (0,0) (p1) {};
\node[triangle,label={below:$v_1$}] at (1,0) (v1) {};
\node[wdot,label={below:$v_2$}] at (2,0) (v2) {};
\node[triangle,label={below:$v_3$}] at (3,0) (v3) {};
\node[wdot,label={below:$v_4$}] at (4,0) (v4) {};
\node[invisnode] at (4.5,0) (b) {};
\node[square,label={left:$x$}] at (2,1) (x) {};
\node[square,label={right:$y$}] at (3,1) (y) {};
\node[triangle,label={left:$w$}] at (2,2) (w) {};
\node[triangle,label={right:$z$}] at (3,2) (z) {};

\draw[black] (a)--(p1)--(v1)--(v2)--(v3)--(v4)--(b);
\draw[black] (v1)--(x)--(y)--(v2);
\draw[black] (x)--(w);
\draw[black] (z)--(y);
\draw[red] (z)--(w);

\draw[black] (p1)--(0.6,0.3);
\draw[black] (x)-- (1.4,0.7);
\draw[black] (y)--(3.25,0.75);
\draw[black] (v4)--(3.75,0.25);
\draw[black, line width=1.2pt, line cap=round, dash pattern=on 0pt off 4.5\pgflinewidth] (.84,0.42)--(1.28,0.64);
\draw[black, line width=1.2pt, line cap=round, dash pattern=on 0pt off 4\pgflinewidth] (3.66,0.34)--(3.3,0.7);
\end{tikzpicture}

\begin{tikzpicture}[scale = .9]
\node[invisnode] at (-.5,0) (a) {};
\node[bdot, label={below:$p_1$}] at (0,0) (p1) {};
\node[triangle,label={below:$v_1$}] at (1,0) (v1) {};
\node[wdot,label={below:$v_2$}] at (2,0) (v2) {};
\node[triangle,label={below:$v_3$}] at (3,0) (v3) {};
\node[wdot,label={below:$v_4$}] at (4,0) (v4) {};
\node[triangle,label={below:$y$}] at (5.25,0) (y) {};

\node[invisnode] at (4.4,0) (b) {};
\node[invisnode] at (4.8,0) (c) {};
\node[invisnode] at (5.75,0) (d) {};
\node[square,label={above:$x$}] at (2,1) (x) {};
\node[triangle,label={above:$w$}] at (3,1) (w) {};
\node[triangle,label={above:$z$}] at (4,1) (z) {};

\draw[black] (a)--(p1)--(v1)--(v2)--(v3)--(v4)--(b);
\draw[black] (c)--(y)--(d);
\draw[black] (v1)--(x)--(y);
\draw[black] (x)--(w);
\draw[black] (z)--(y);
\draw[red] (z)--(w);

\draw[black] (p1)--(0.6,0.3);
\draw[black] (x)-- (1.4,0.7);

\draw[black, line width=1.2pt, line cap=round, dash pattern=on 0pt off 4.5\pgflinewidth] (.84,0.42)--(1.28,0.64);

\draw[black, line width=1.2pt, line cap=round, dash pattern=on 0pt off 4\pgflinewidth] (4.41,0)--(4.79,0);
\end{tikzpicture}
\hspace{1cm}
\begin{tikzpicture}[scale = .9]
\node[invisnode] at (-.5,0) (a) {};
\node[bdot, label={below:$p_1$}] at (0,0) (p1) {};
\node[triangle,label={below:$v_1$}] at (1,0) (v1) {};
\node[wdot,label={below:$v_2$}] at (2,0) (v2) {};
\node[triangle,label={below:$v_3$}] at (3,0) (v3) {};
\node[wdot,label={below:$v_4$}] at (4,0) (v4) {};
\node[triangle,label={below:$y$}] at (5.25,0) (y) {};

\node[invisnode] at (4.4,0) (b) {};
\node[invisnode] at (4.8,0) (c) {};
\node[invisnode] at (5.75,0) (d) {};
\node[square,label={above:$x$}] at (2,1) (x) {};
\node[triangle,label={above:$w$}] at (3,1) (w) {};
\node[triangle,label={above:$z$}] at (4,1) (z) {};

\draw[black] (a)--(p1)--(v1)--(v2)--(v3)--(v4)--(b);
\draw[black] (c)--(y)--(d);
\draw[black] (v3)--(x)--(y);
\draw[black] (x)--(w);
\draw[black] (z)--(y);
\draw[red] (z)--(w);

\draw[black] (p1)--(0.6,0.3);
\draw[black] (x)-- (1.4,0.7);

\draw[black, line width=1.2pt, line cap=round, dash pattern=on 0pt off 4.5\pgflinewidth] (.84,0.42)--(1.28,0.64);

\draw[black, line width=1.2pt, line cap=round, dash pattern=on 0pt off 4\pgflinewidth] (4.41,0)--(4.79,0);
\end{tikzpicture}
\caption{The end configurations in Case 4 that produce a worse cycle $xy$. In all of the configurations $|L(v_3)| = 3$ or $4$.}
\label{fig:case4b2}
\end{figure}
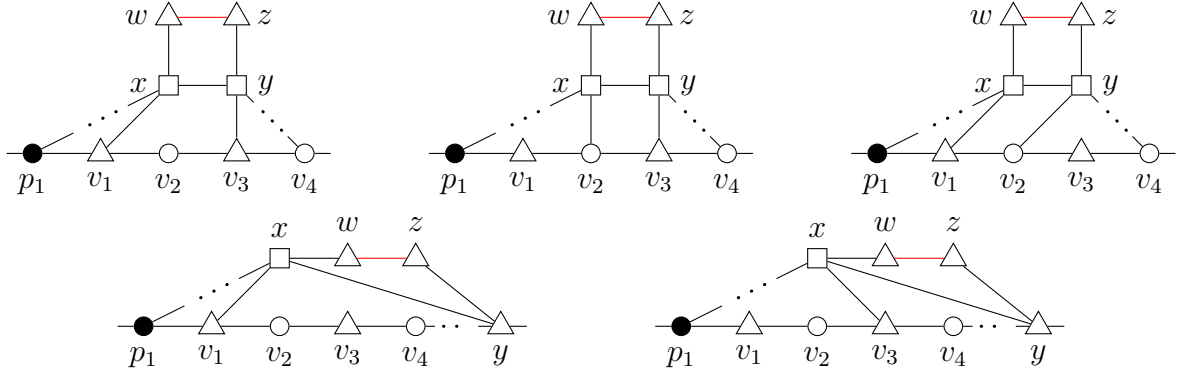

If $G'$ has a bad cycle $xyzw$ with $|L'(y)|=2$, then it must be the case that $|L(x)|=|L(z)|=4$, $|L(y)|=|L(w)|=3$ and $xv_1,yv_2,zv_3 \in E(G)$.

Now we further color $y$ with some color $c\in L(y)-a$, and $c=b$ if possible. Let $H=G'-y$ and $M$ be the restriction of $L'$ to $H$ with $c$ removed from $L'(x),L'(z)$.

Suppose that $(H,P,M)$ contains a bad edge. Then it must be incident to exactly one of $x,z$.
We may assume the edge is $xx'$. Then $x'$ must be adjacent to $v_3$, as it is not a boundary vertex in $G$ and it could not be adjacent to any of $y,v_1,v_2$.
Then $xx'v_3zy$ forms a separating $5$-cycle, a contradiction.

Since $(G',P,L')$ has no bad, worse or worst cycle other than $xyzw$, if $(H,P,M)$ has a forbidden cycle, then it must contain exactly one of $x,z$. So, without loss of generality, we assume it contains $x$.
If $(H,P,M)$ contains a bad or worse cycle, then it must contain $xw$, but then $w$ is adjacent to another interior triangle in $G$, a contradiction.
If $(H,P,M)$ contains a worst cycle, then $x$ is adjacent to a triangle $x'$ in $G$, and $x'$ must be adjacent to $v_3$, but then $xx'v_3zy$ forms a separating $5$-cycle with $w$ in the interior, a contradiction.

We first claim that if $z$ is adjacent to some $p_i$, then $i\in\{3,4\}$. 
If $i=5$, the path $p_5zv_3$ is a $(1,3)$-$2$-chord with $|P \cap B(G_{W,2})|=1$, contradicting Lemma~\ref{lem:f-chords}. If $i=1$ or $i=2$, then $zp_1v_1v_2yz$ or $zp_2p_1v_1v_2yz$ is a separating $5$- or $6$-cycle where $x$ with $|L(x)|=4$ is in the interior of this separating cycle, contradicting Lemma~\ref{lem:sep45} and Lemma~\ref{lem:sep6}.

\begin{claim}\label{cla: bad vertices}
Any two of the following cannot occur in $(H,P,M)$ at the same time:
\begin{enumerate}
    \item[(1)] $v_1,p_2$ have a common neighbor $u$ with $|L(u)|=3$ in $G$,
    \item[(2)] $x$ is adjacent to $p_2$,
    \item[(3)] $z$ is adjacent to $p_3$ or $p_4$,
    \item[(4)] $v_3,p_3$ have a common neighbor $v$ with $|L(v)|=3$ in $G$.
\end{enumerate}
\end{claim}

\begin{proof}
Suppose $(1)$ and $(2)$ occur at the same time. Then $v_1p_1p_2x$ forms a separating $4$-cycle containing $u$, a contradiction. 

Suppose $(1)$ and $(3)$ occur at the same time. Suppose $v_1,p_2$ has a common neighbor $u$ with $|L(u)|=3$ and $z$ is adjacent to $p_4$. We consider the subgraph $H' = H-\{p_1\}$ with list $M'$ such that $u$ is precolored by a color $\alpha_1 \in M(u)-M(p_2)$ and $M'(v) = M(v)$ otherwise. Then we precolor $z$ by a color $\alpha_2 \in M'(z)-M'(p_4)$ and decompose $H'$ into two graphs $H'_1$ and $H'_2$ where $H'_1$ has precolored path $P'_1=up_2p_3p_4z$ and $H'_2$ has precolored path $P'_2=p_5p_4z$. Let $M'_i$ be the restriction of $M'$ on $H'_i$ for $i=1,2$. We observe that $(H'_1, P'_1, M'_1)$ has no bad edge, bad cycle, worse cycle, or worst cycle. If it has a bad vertex $u'$, then $u'$ could be a common neighbor of $v_1,p_2$ or $u'$ could be a common neighbor of $v_1,z$ with $|L(u')|=3$, or $u'$ could be a common neighbor of $p_i,z$ for $i=2,3$ with $|L(u')|=3$, or $u'$ could be the common neighbor of $u,p_3,z$ with $|L(u')|=4$. If $u'$ is a common neighbor of $v_1,p_2$, then $v_1p_1p_2u'v_1$ is a separating $4$-cycle, a contradiction. 

If $u'$ is a common neighbor of $v_1$ and $z$, then $v_1u'zv_3v_2v_1$ is a 5-cycle in $G$. Since $v_1v_2v_3$ is a subpath on the boundary of $G$, and the vertex $y$ lies in the interior of $C$, it follows that $C$ is a separating 5-cycle, which contradicts Lemma~\ref{lem:sep45}.

If $u'$ is a common neighbor of $p_2$ and $z$, then $p_2p_3p_4zu'$ forms a $5$-cycle. As $G$ has no separating $5$-cycle, $N_{G}(p_3) = \{p_2, p_4\}$ and $N_{G}(p_2) \cap N_{G}(p_4) = \emptyset$. Thus, the graph obtained by deleting $p_3$ from $G$ and adding in the edge $p_2p_4$ is a valid target with a coloring that extends to $G-E(P)$. If $u'$ is a common neighbor of $p_3,z$, we first precolor $u'$ by a color $\alpha'_1 \in M'_1(u')-M'_1(p_3)-M'_1(z)$. Then we further decompose $H'_1$ into $2$ subgraphs a $4$-cycle $p_3p_4zu'p_3$ whose interior is empty and $H''_1$ where $H''_1$ has precolored path $P''_1=up_2p_3u'z$. Let $M''_1$ be the restriction of $M'_1$ on $H''_1$. It is easy to see $(H''_1, P''_1, M''_1)$ has no bad objects and thus is a valid target. Since the $4$-cycle $p_3p_4zu'p_3$ has empty interior, $(H'_1, P'_1, M'_1)$ is a valid target.

If $u'$ is the common neighbor of $u,p_3,z$ with $|L(u')|=4$, then $u'uv_1v_2v_3zu'$ is a separating $6$-cycle of $G$. We have $x$ is in the interior of the separating $6$-cycle with $|L(x)|=4$, contradicting Lemma~\ref{lem:sep6}.

Hence, $(H'_1, P'_1, M'_1)$ is a valid target. Since $|P'_2|=3$, it is easy to see that $(H'_2, P'_2, M'_2)$ has no bad objects and thus is a valid target. Together with colored $v_1,v_2,v_3,y$, we have a coloring of $G-E(P)$, a contradiction. The case of $z$ adjacent to $p_3$ is similar. The proof of the remaining cases are similar to the proof of the case of $(1)$ and $(3)$ occur.
\end{proof}

Suppose that $(H,P,M)$ contains a bad vertex. If the bad vertex is $x$, then by Claim~\ref{cla:v,p}, $x$ is adjacent to $p_2$. Let $H'=H-p_1$. We consider the subgraph $H' = H-\{p_1\}$ with list $M'$ such that $x$ is precolored by a color $\alpha \in M(x)-M(p_2)$ and $M'(v) = M(v)$ otherwise and $P'=p_5\ldots p_2x$. We observe that $(H',P',M')$ has no bad edge, bad cycle, worse cycle, or worst cycle. By Claim \ref{cla: bad vertices}, if $(H',P',M')$ has a bad vertex $x'$, then we have $|M'(x)|=|L(x)|=3$ and $x'$ must be adjacent to $x$ and exactly one of $p_3,p_4,p_5$. Suppose $x'$ is adjacent to $p_i$ for some $i=3,4,5$. We consider the subgraph $H'' = H'-\{p_2,\ldots,p_{i-1}\}$ with list $M''$ such that $x'$ is precolored by a color $\beta \in M'(x')-M'(p_2)-M'(x)$ and $M''(v) = M'(v)$ otherwise and $P''=p_5 \ldots p_i x' x$. We observe that $(H'',P'',M'')$ has no bad edge, bad cycle, worse cycle, or worst cycle. If $i=4,5$, then $|P''|\leq 4$ and there is no bad vertex, and thus $(H'',P'',M'')$ is a valid target. Thus, we must have $x'$ is adjacent to $p_3$. If $(H'',P'',M'')$ has a bad vertex $x''$, then $|M''(x'')|=|L(x'')|=3$ and adjacent to exactly one of $p_3,p_4,p_5$, and exactly one of $x',x$. By the definition of $X_{G,L}$, $x''$ cannot be adjacent to $x'$. Thus, $x''$ is adjacent to $x$. If $x''$ is adjacent to $p_3$ or $p_4$, then we get a separating $4$ or $5$-cycle, a contradiction. Thus, $x''$ is adjacent to $p_5$ and $x''p_5p_4p_3p_2xx''$ is a separating $6$-cycle. By Lemma \ref{lem:sep6}, we have $x'$ is adjacent to $p_5$. Then $p_5x'p_3$ is a $(1,1)$-$2$-chord, contradicting Lemma~\ref{lem:112chord}. 

If $z$ is a bad vertex for $(H,P,M)$, then $z$ is adjacent to $p_i$ for some $i \in \{3,4\}$. If $i=4$, we precolor $z$ with some color $\gamma \in M(z)-M(p_4)$. We decompose $H$ into two graphs $H_1$ and $H_2$ where $H_1$ has precolored path $P_1=p_1p_2p_3p_4z$ and $H_2$ has precolored path $P_2=p_5p_4z$. Let $M_i$ be the restriction of $M$ on $H_i$ for $i=1,2$. By the definition of $X_{G,L}$, $(H_1,P_1,M_1)$ has no bad edge, bad cycle, worse cycle, or worst cycle. If $(H_1,P_1,M_1)$ has a bad vertex $u_1$, then by Claim \ref{cla: bad vertices}, we must have $|L(u_1)|=3$ and $u_1$ is adjacent to $z$ and exactly one of $p_1,p_2,p_3$. If $u_1$ is adjacent to $z$ and $p_1$, then $p_1v_1v_2v_3zu_1p_1$ is a separating $6$-cycle of $G$ whose interior contradicts Lemma~\ref{lem:sep6}. If $u_1$ is adjacent to $z$ and $p_2$, we have a $5$-cycle $p_4p_3p_2u_1zp_4$ and thus its interior must be empty. We precolor $u_1$ by a color $\gamma' \in M_1(u_1)-M_1(p_2)-M_1(z)$. We consider the subgraph $H'_1$ of $H_1$ where $H'_1$ has precolored path $P'_1=p_1p_2u_1z$. Let $M'_1$ be the restriction of $M_1$ to $H'_1$. By the definition of $X_{G,L}$, $(H'_1,P'_1,M'_1)$ has no bad edge, bad cycle, worse cycle, or worst cycle. If it has a bad vertex $u'_1$, then $u'_1$ is the common neighbor of $v_1$ and $z$. Then $d(x)=3 <|L(x)|=4$, contradicting Lemma \ref{lem:degree}. Thus, $(H'_1,P'_1,M'_1)$ is a valid target and together with colored $u_1$, we have a coloring of $H_1$. The case of $u_1$ adjacent to $z$ and $p_3$ is similar. Hence, $(H_1,P_1,M_1)$ is a valid target. Now we will show that $(H_2,P_2,M_2)$ is also a valid target. Since there is no $(2,3^+)$-$2$-chord, we have $|P_2|\leq 3$ and it has no bad vertex. Furthermore, by the sparsity of $X_{G,L}$, we observe that $(H_2,P_2,M_2)$ has no bad edge, bad cycle, worse cycle or worst cycle. Hence, $(H_2,P_2,M_2)$ is a valid target. Together with colored $v_1,v_2,v_3,y$, we have a coloring of $G-E(P)$, a contradiction. The case of $i=3$ is similar to the case of $i=4$.

If neither $x$ nor $z$ is a bad vertex for $(H,P,M)$, then let $u$ be a bad vertex. Since there is no $(1,1)$-$2$-chord, we have $|L(u)|=3$ and $u$ is either the common neighbor of $v_1,p_2$ or the common neighbor of $v_3,p_3$. If $u$ is the common neighbor of $v_1,p_2$, then we consider the subgraph $H' = H-\{p_1\}$ with list $M'$ such that $u$ is precolored by a color $\alpha \in M(u)-M(p_2)$ and $M'(v) = M(v)$ otherwise and $P'=p_5\ldots p_2u$. We observe that $(H',P',M')$ has no bad edge, bad cycle, worse cycle, or worst cycle. By Claim \ref{cla: bad vertices} and the sparsity of $X_{G,L}$, if $(H',P',M')$ has a bad vertex $u'$, we must have either $u'$ is the common neighbor of $v_3,u,p_i$ for some $i \in \{3,4,5\}$ with $|L(u)|=4$ or $u'$ is the common neighbor of $v_1,p_2$ with $|L(u)|=3$. In the former case, $v_1v_2v_3u'uv_1$ forms a separating $5$-cycle, and in the latter case, $v_1p_1p_2u'v_1$ forms a separating $4$-cycle, a contradiction. Thus, $(H',P',M')$ is a valid target. Together with colored $v_1,v_2,v_3,y$, we have a $L$-coloring of $G-E(P)$, a contradiction. 

If $u$ is the common neighbor of $v_3,p_3$, We precolor $u$ by a color $\gamma'' \in M(u)-M(p_3)$. We decompose $H$ into two subgraphs $\tilde{H_1}$ and $\tilde{H_2}$ where $\tilde{H_1}$ has precolored path $\tilde{P_1}=p_1p_2p_3u$ and $\tilde{H_2}$ has precolored path $\tilde{P_2}=p_5p_4p_3u$. Let $\tilde{M_i}$ be restriction of $M$ to $\tilde{H_i}$ for $i=1,2$. We observe that $(\tilde{H_1},\tilde{P_1},\tilde{M_1})$ has no bad edge, bad cycle, worse cycle, or worst cycle. By Claim \ref{cla: bad vertices}, $(\tilde{H_1},\tilde{P_1},\tilde{M_1})$ has no bad vertex. Thus, $(\tilde{H_1},\tilde{P_1},\tilde{M_1})$ is a valid target. Now we will show that $(\tilde{H_2},\tilde{P_2},\tilde{M_2})$ is also a valid target. Since $|\tilde{P_2}|\leq 4$ and there is no $(2,3^+)$-$2$-chord, it has no bad vertex. Moreover, we observe that it has no bad edge, bad cycle, worse cycle or worst cycle. Hence, $(\tilde{H_2},\tilde{P_2},\tilde{M_2})$ is a valid target. Together with colored $v_1,v_2,v_3,y$, we have a coloring of $G-E(P)$, a contradiction.

Hence, $(H,P,M)$ is a valid target, and together with colored $y,v_1,v_2,v_3$, we have a coloring of $G-E(P)$, a contradiction.

It remains to show that $(G',P,L')$ has no bad vertex. Since there is no $(1,1)$-$2$-chord by Lemma \ref{lem:112chord}, if it has a bad vertex, then the bad vertex must have list size $3$ in $G$ and be the common neighbor of one of $p_i$ and at least one of $v_1,v_2,v_3$. By Corollary~\ref{cor:122chord}, it cannot be the common neighbor of one of the $p_i$ and $v_2$, since otherwise there is a $(1,2)$-$2$-chord. Thus, by Claims~\ref{cla:v,p} and~\ref{cla:v3,p}, the bad vertex must be the common neighbor of $p_2$ and $v_1$ or of $p_3$ and $v_3$. In particular, there are at most three such common neighbors otherwise there is a separating $4$-cycle. Assume that $G$ has a bad vertex $x$ adjacent to $p_3$ and $v_3$ and let $y$ and $z$ be the bad vertices adjacent to $v_1$ and $p_2$ or $v_3$ and $p_3$ respectively if they exist. Then we may partition $G'$ into induced subgraphs $G_i'$ for $i=1,2,3$ each with a precolored path.
Namely, $P_1'=p_1p_2y$, $P_2'=yp_2p_3x$ and $P_3'=zp_3p_4p_5$ (where $P_1'=p_1p_2p_3x$ if $y$ does not exist and $P_3'=xp_3p_4p_5$ if $z$ does not exist). As each $|P_i'| <5$ and $(G',P,L')$ has no bad edge, each $(G_i', P_i', L_i')$ is a valid target.
By minimality of $G$, we may color each $G_i'- E(P_i')$ and combine them to get a proper coloring of $G'$, a contradiction. 

Hence, the only bad vertex $x$ of $(G',P,L')$ is a common neighbor of $v_1$ and $p_2$ with $|L(x)| = 3$. Let $G'' = G'-p_1$ with precolored path $P'' = xp_2p_3p_4p_5$ and $L''(x) = L'(x)-L(p_1)$. As $(G',P,L')$ has no bad edge, bad cycle, worse cycle, or worst cycle, it suffices to show that $G''$ has no bad vertex $w$. By Lemmas~\ref{lem:f-chords} and~\ref{lem:332chord}, $w$ must be an interior vertex of $G$. As $G'$ has no other bad vertex besides $x$ and does not have a bad edge, $|L''(w)| = 3$ and must be adjacent to $x$. If $|L(w)| = 4$, then $w$ is adjacent to $v_3$ and exactly one of $p_3$ and $p_4$. Note that if $w$ is adjacent to $p_4$, then $N_{G}(p_3) = \{p_2, p_4\}$. The graph obtained from $G$ by deleting $p_3$ and adding in the edge $p_2p_4$ with precolored path $p_1p_2p_4p_5$ is a valid target and has a coloring that can be extended to $G-P$. Hence, if $|L(w)| = 4$ it is adjacent to $x$, $v_3$, and $p_3$. In this case we may partition $G''$ into induced subgraphs $G_1''$ and $G_2''$ with precolored paths $P_1'' = xp_2p_3w$ and $P_2'' = wp_3p_4p_5$ respectively. It is easily checked that both $(G_1'',P_1'',L_1'')$ and $(G_2'',P_2'',L_2'')$ are valid targets that combine to give a valid coloring of $G'' - P''$, a contradiction.

Thus, $|L(w)| = 3$ and $w$ must be adjacent to $x$ and $p_3$ by the same argument as above. We partition $G''$ into induced subgraphs $G_1''$ and $G_2''$ with precolored paths $P_1'' = xp_2p_3w$ and $P_2'' = xwp_3p_4p_5$ respectively. As $(G_1'',P_1'',L_1'')$ is trivially a valid target, it remains to check that $(G_2'',P_2'',L_2'')$ is valid. As with $G''$, it suffices to show $G_2''$ has no bad vertex $v$. If $v$ is on the boundary of $G$, it must be consecutive on the boundary to $p_5$ and be adjacent to $w$ creating the $5$-cycle $p_3p_4p_5vw$. Considering the graph obtained from $G$ by deleting $p_4$ and adding the edge $p_3p_5$, we obtain a valid target whose coloring can be extended to $G-P$. Therefore, $v$ is an interior vertex of $G$. As $v$ is not a bad vertex of $G''$, $v$ must be adjacent to $w$ and $p_4$. Furthermore, by the proximity of $v$ to $xw \in X_{G,L}$, $|L(v)| = 4$ and $v$ must be adjacent to at least one of $v_1$ and $v_3$. In this case, we repartition, $G''$ into the induced subgraphs $G_1''$, $G_2''$, and $G_3''$ with precolored paths $P_1'' = xp_2p_3w$, $P_2'' = wp_3p_4v$, and $P_3'' = vp_4p_5$. Each $(G_i'',P_i'', L_i'')$ is a valid target with valid colorings that give a valid coloring of  $G'' - P''$, a contradiction.

\section{Combinations of end configurations} \label{sec4}
In the previous section, we saw that an end configuration will produce either a bad edge or a worse cycle. Let $wz$ be a bad edge or the interior edge of a worse cycle produced by the end configuration.

We now consider peeling vertices from both sides of the precolored path $P$. Each side must produce one of the end configurations demonstrated in the previous section. We say that the two end configurations can be combined if they share the same edge $wz$; otherwise, we say that they cannot be combined.

Let $xyzw$ be a worse cycle produced by peeling off the given set of vertices stated in the previous section, where $w,z$ are in the interior and $x,y$ are on the boundary. We say that an end configuration produces a \textit{worse cycle of type I} if it produces a worse cycle $xywz$ where at least one of $x,y \in B(G)$. We say that an end configuration produces a \textit{worse cycle of type II} if it produces a worse cycle $xywz$ where $x,y \in \int(G)$.

\begin{lemma}
Two end configurations in which
\begin{enumerate}
    \item both produce bad edges, or
    \item both produce worse cycles of type I, or
    \item one produces a bad edge and the other produces a worse cycle of type I
\end{enumerate}
cannot be combined.    
\end{lemma}

\begin{proof}
Suppose they can be combined. Then we observe that there will be a $(2,2^+)$-$2$-chord or a feasible $(3,3^+)$-$2$ chord $uvw$ with $|L(v)|=3$, contradicting Lemma~\ref{lem:f-chords} and Lemma~\ref{lem:332chord}.
\end{proof}

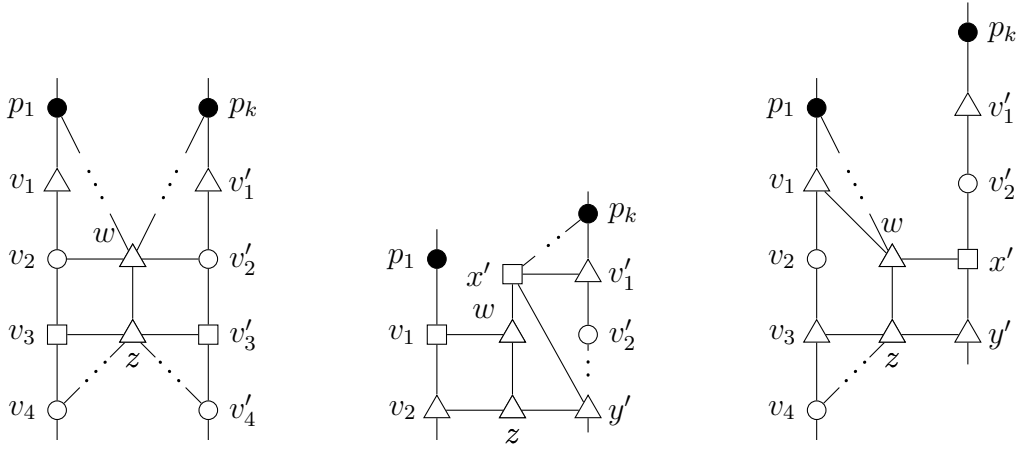
\begin{figure}[!htbp]
\centering
\makebox[\textwidth][c]{
\hfill
\begin{tikzpicture}
\node[invisnode] at (1,1.5) (a) {};
\node[bdot,label={left:$p_1$}] at (1,1) (p1) {};
\node[triangle,label={left:$v_1$}] at (1,0) (v1) {};
\node[wdot,label={left:$v_2$}] at (1,-1) (v2) {};
\node[square,label={left:$v_3$}] at (1,-2) (v3) {};
\node[wdot,label={left:$v_4$}] at (1,-3) (v4) {};
\node[invisnode] at (1,-3.5) (b) {};

\node[triangle,label={135:$w$}] at (2,-1) (w) {};
\node[triangle,label={below:$z$}] at (2,-2) (z) {};

\draw[black] (a)--(p1)--(v1)--(v2)--(v3)--(v4)--(b);
\draw[black] (v2)--(w);
\draw[black] (v3)--(z);
\draw[black] (w)--(z);

\draw[black] (p1)--(1.3,0.4);
\draw[black] (w)--(1.7,-0.4);
\draw[black] (z)--(1.7,-2.3);
\draw[black] (v4)--(1.3,-2.7);

\draw[black, line width=1.2pt, line cap=round,
  dash pattern=on 0pt off 4.5\pgflinewidth]
  (1.42,0.16)--(1.64,-0.28);

\draw[black, line width=1.2pt, line cap=round,
  dash pattern=on 0pt off 4\pgflinewidth]
  (1.39,-2.61)--(1.65,-2.35);

\node[invisnode] at (3,1.5) (a) {};
\node[bdot,label={right:$p_k$}] at (3,1) (p1) {};
\node[triangle,label={right:$v'_1$}] at (3,0) (v1) {};
\node[wdot,label={right:$v'_2$}] at (3,-1) (v2) {};
\node[square,label={right:$v'_3$}] at (3,-2) (v3) {};
\node[wdot,label={right:$v'_4$}] at (3,-3) (v4) {};
\node[invisnode] at (3,-3.5) (b) {};

\node[triangle] at (2,-1) (w) {};
\node[triangle,label={below:$z$}] at (2,-2) (z) {};

\draw[black] (a)--(p1)--(v1)--(v2)--(v3)--(v4)--(b);
\draw[black] (v2)--(w);
\draw[black] (v3)--(z);

\draw[black] (p1)--(2.7,0.4);
\draw[black] (w)--(2.3,-0.4);
\draw[black] (z)--(2.3,-2.3);
\draw[black] (v4)--(2.7,-2.7);

\draw[black, line width=1.2pt, line cap=round,
  dash pattern=on 0pt off 4.5\pgflinewidth]
  (2.58,0.16)--(2.36,-0.28);

\draw[black, line width=1.2pt, line cap=round,
  dash pattern=on 0pt off 4\pgflinewidth]
  (2.61,-2.61)--(2.35,-2.35);
\end{tikzpicture}

\hfill

\begin{tikzpicture}
\node[invisnode] at (1,0.5) (a) {};
\node[bdot,label={left:$p_1$}] at (1,0) (p1) {};
\node[square,label={left:$v_1$}] at (1,-1) (v1) {};
\node[triangle,label={left:$v_2$}] at (1,-2) (v2) {};
\node[invisnode] at (1,-2.5) (b) {};

\node[triangle,label={135:$w$}] at (2,-1) (w) {};
\node[triangle,label={below:$z$}] at (2,-2) (z) {};

\draw[black] (a)--(p1)--(v1)--(v2)--(b);
\draw[black] (v1)--(w);
\draw[black] (v2)--(z);
\draw[black] (w)--(z);

\node[invisnode] at (3,1) (a) {};
\node[bdot,label={right:$p_k$}] at (3,0.6) (p1) {};
\node[triangle,label={right:$v'_1$}] at (3,-0.2) (v1) {};
\node[wdot,label={right:$v'_2$}] at (3,-1) (v2) {};
\node[invisnode] at (3,-1.32) (b) {};
\node[invisnode] at (3,-1.68) (c) {};
\node[triangle,label={right:$y'$}] at (3,-2) (y) {};
\node[invisnode] at (3,-2.4) (d) {};

\node[square,label={left:$x'$}] at (2,-0.2) (x) {};
\node[triangle] at (2,-1) (w) {};
\node[triangle,label={below:$z$}] at (2,-2) (z) {};

\draw[black] (a)--(p1)--(v1)--(v2)--(b);
\draw[black] (c)--(y)--(d);
\draw[black] (v1)--(x)--(w);

\draw[black] (p1)--(2.7,0.36);
\draw[black] (x)--(2.3,0.04);
\draw[black] (z)--(y)--(x);

\draw[black, line width=1.2pt, line cap=round,
  dash pattern=on 0pt off 4\pgflinewidth]
  (3,-1.36)--(3,-1.68);

\draw[black, line width=1.2pt, line cap=round,
  dash pattern=on 0pt off 4\pgflinewidth]
  (2.61,0.288)--(2.35,0.08);
\end{tikzpicture}

\hfill

\begin{tikzpicture}

\node[invisnode] at (1,1.5) (a) {};
\node[bdot,label={left:$p_1$}] at (1,1) (p1) {};
\node[triangle,label={left:$v_1$}] at (1,0) (v1) {};
\node[wdot,label={left:$v_2$}] at (1,-1) (v2) {};
\node[triangle,label={left:$v_3$}] at (1,-2) (v3) {};
\node[wdot,label={left:$v_4$}] at (1,-3) (v4) {};
\node[invisnode] at (1,-3.5) (b) {};

\node[triangle] at (2,-1) (w) {};
\node[triangle,label={below:$z$}] at (2,-2) (z) {};

\draw[black] (a)--(p1)--(v1)--(v2)--(v3)--(v4)--(b);
\draw[black] (v1)--(w);
\draw[black] (v3)--(z);
\draw[black] (w)--(z);

\draw[black] (p1)--(1.3,0.4);
\draw[black] (w)--(1.7,-0.4);
\draw[black] (z)--(1.7,-2.3);
\draw[black] (v4)--(1.3,-2.7);

\draw[black, line width=1.2pt, line cap=round,
  dash pattern=on 0pt off 4.5\pgflinewidth]
  (1.42,0.16)--(1.64,-0.28);

\draw[black, line width=1.2pt, line cap=round,
  dash pattern=on 0pt off 4\pgflinewidth]
  (1.39,-2.61)--(1.65,-2.35);

  \node[invisnode] at (3,2.5) (a) {};
\node[bdot,label={right:$p_k$}] at (3,2) (p1) {};
\node[triangle,label={right:$v'_1$}] at (3,1) (v1) {};
\node[wdot,label={right:$v'_2$}] at (3,0) (v2) {};
\node[square,label={right:$x'$}] at (3,-1) (x) {};
\node[triangle,label={right:$y'$}] at (3,-2) (y) {};
\node[invisnode] at (3,-2.5) (b) {};

\node[triangle,label={above:$w$}] at (2,-1) (w) {};
\node[triangle,label={below:$z$}] at (2,-2) (z) {};

\draw[black] (a)--(p1)--(v1)--(v2)--(x)--(y)--(b);
\draw[black] (x)--(w);
\draw[black] (y)--(z);

\end{tikzpicture}
\hfill
}

\caption{ Examples of combinations of two end configurations in which both produce bad edges (left figure), both produce worse cycles of type I (middle figure), and one produces a bad edge and the other produces a worse cycle of type I (right figure).   }
\end{figure}

Given an end configuration that produces a bad edge $wz$, let $S(wz)$ be the set of peeled vertices in the end configuration. Let $v_w=N(w) \cap S(wz)$ and $v_z=N(z) \cap S(wz)$. For example, in the first end configuration of Figure~\ref{fig:case4b1}, $S(wz) = \{v_1, v_2, v_3\}$, $v_w=v_2$, and $v_z=v_3$.

Similarly, let $S(xywz)$ be the set of peeled vertices in an end configuration that produces the worse cycle $xywz$. If $y \in B(G)$ (resp.~$x \in B(G)$), then we set $v_y=y$ (resp.~$v_x = x$). Otherwise, let $\{v_x\}=N(x) \cap S(xywz)$ and $\{v_y\}=N(y) \cap S(xywz)$ as before. For example, in the first end configuration of Figure~\ref{fig:case31}, $S(xwyz) = \{v_1, v_2\}$, $v_x=v_1$, and $v_y=v_2$. In the second end configuration of Figure~\ref{fig:case31}, $S(xwyz) = \{v_1, v_2\}$, $v_x=v_1$, and $v_y=y$.

In the remaining proofs, if $v$ is a precolored vertex, we use $v^+$ to denote its neighbor that has a list of size two, which may or may not exist.

\begin{lemma}
An end configuration that produces a bad edge cannot be combined with one that produces a worse cycle of type II.
\end{lemma}

\begin{proof}
Suppose that they can be combined. Without loss of generality, assume that the worse cycle $xywz$ is produced on the side of $p_1$ and the bad edge $wz$ is produced on the side of $p_k$. Let $G_1$ be the induced subgraph with boundary $B(G_1)=wxv_x B(G) p_1 \cdots p_k B(G) v_w w$ where $a B(G) b$ denotes the shortest path on the boundary of $G$ between the boundary vertices $a$ and $b$. Let $G_2$ be the subgraph with boundary $B(G_2)=xyzwx$. Let $G_3$ be the subgraph with boundary $B(G_3)=zyv_y B(G-G_1) v_z z$. Now color and peel off $S(xywz)$ and $S(wz)$ as described in the previous section. Let $L'$ be the resulting list given in the previous section. Let $G'_i = G_i-S(xywz)-S(wz)$ and $L'_i$ be $L'$ restricted to $G_i'$ for $i=1,2,3$. By the definition of $X_{G,L}$, $(G'_1,P,L'_1)$ has no bad edge and no bad, worse or worst cycle. If it has a bad vertex $v$ such that $|L'_1(v)|=3$ and $|N(v) \cap P|=2$, then there is a $(1,1)$-$2$-chord, contradicting Lemma \ref{lem:112chord}. Thus $(G'_1,P,L'_1)$ is a valid target and thus has a proper $L'_1$-coloring $\phi_1$. Let $L''_2(u)=\{\phi_1(u)\}$ if $u=x,w$ and $L''_2(u)=L'_2(u)$ otherwise. We observe that $(G'_2, P_2, L''_2)$ is a valid target with $P_2=xw$ and thus has a proper $L''_2$-coloring $\phi_2$. Let $L''_3(u)=\{\phi_2(u)\}$ if $u=z,y$ and $L''_3(u)=L'_3(u)$ otherwise. By the definition of $X_{G,L}$, $(G'_3, P_3, L''_3)$ is a valid target with $P_3=z^+zyy^+$ and thus has a proper $L''_3$-coloring $\phi_3$. Then $\phi_1 \cup \phi_2 \cup \phi_3$ gives a proper $L'$-coloring of $G-E(P)-S(xyzw)-S(zw)$, a contradiction.
\end{proof}

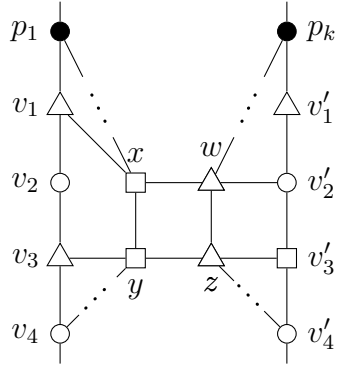
\begin{figure}[!htbp]
\centering

\begin{tikzpicture}
\node[invisnode] at (0,1.5) (a) {};
\node[bdot,label={left:$p_1$}] at (0,1) (p1) {};
\node[triangle,label={left:$v_1$}] at (0,0) (v1) {};
\node[wdot,label={left:$v_2$}] at (0,-1) (v2) {};
\node[triangle,label={left:$v_3$}] at (0,-2) (v3) {};
\node[wdot,label={left:$v_4$}] at (0,-3) (v4) {};
\node[invisnode] at (0,-3.5) (b) {};

\node[square,label={above:$x$}] at (1,-1) (x) {};
\node[square,label={below:$y$}] at (1,-2) (y) {};
\node[triangle,label={above:$w$}] at (2,-1) (w) {};
\node[triangle,label={below:$z$}] at (2,-2) (z) {};

\draw[black] (a)--(p1)--(v1)--(v2)--(v3)--(v4)--(b);
\draw[black] (v1)--(x)--(y)--(v3);
\draw[black] (x)--(w);
\draw[black] (z)--(y);
\draw[black] (z)--(w);

\draw[black] (p1)--(0.3,0.4);
\draw[black] (x)--(0.7,-0.4);
\draw[black] (y)--(0.75,-2.25);
\draw[black] (v4)--(0.25,-2.75);

\draw[black, line width=1.2pt, line cap=round,
  dash pattern=on 0pt off 4.5\pgflinewidth]
  (0.42,0.16)--(0.64,-0.28);

\draw[black, line width=1.2pt, line cap=round,
  dash pattern=on 0pt off 4\pgflinewidth]
  (0.34,-2.66)--(0.7,-2.3);

\node[invisnode] at (3,1.5) (a) {};
\node[bdot,label={right:$p_k$}] at (3,1) (p1) {};
\node[triangle,label={right:$v'_1$}] at (3,0) (v1) {};
\node[wdot,label={right:$v'_2$}] at (3,-1) (v2) {};
\node[square,label={right:$v'_3$}] at (3,-2) (v3) {};
\node[wdot,label={right:$v'_4$}] at (3,-3) (v4) {};
\node[invisnode] at (3,-3.5) (b) {};

\node[triangle] at (2,-1) (w) {};
\node[triangle,label={below:$z$}] at (2,-2) (z) {};

\draw[black] (a)--(p1)--(v1)--(v2)--(v3)--(v4)--(b);
\draw[black] (v2)--(w);
\draw[black] (v3)--(z);

\draw[black] (p1)--(2.7,0.4);
\draw[black] (w)--(2.3,-0.4);
\draw[black] (z)--(2.3,-2.3);
\draw[black] (v4)--(2.7,-2.7);

\draw[black, line width=1.2pt, line cap=round,
  dash pattern=on 0pt off 4.5\pgflinewidth]
  (2.58,0.16)--(2.36,-0.28);

\draw[black, line width=1.2pt, line cap=round,
  dash pattern=on 0pt off 4\pgflinewidth]
  (2.61,-2.61)--(2.35,-2.35);

\end{tikzpicture}

\caption{Example of combinations of two end configurations in which one produces a worse cycle of type II on the left and the other produces a bad edge on the right.  }
\end{figure}    

\begin{lemma}
An end configuration that produces a worse cycle of type I cannot be combined with one that produces a worse cycle of type II.
\end{lemma}

\begin{proof}
Suppose for the sake of contradiction that they can be combined. Without loss of generality, assume that the worse cycle of type II $xywz$ is produced on the side of $p_1$ and the worse cycle of type I $x'y'wz$ is produced on the side of $p_k$. Let $G_1$ be the subgraph with boundary $B(G_1)=wxv_x B(G) p_1 \cdots p_k B(G) v_{x'} x' w$. Let $G_2$ be the subgraph with boundary $B(G_2)=xwx'y'zyx$. Let $G_3$ be the subgraph with boundary $B(G_3)=v_yyzy'B(G-G_1)v_y$. Let $G_4$ be the subgraph with boundary $B(G_4)=v_{x'}x'y' B(G-G_1-G_3) v_{x'}$. (Note that if $x' \in B(G)$, then $G_4$ does not exist). Now color and then peel off $S(xywz)$ and $S(x'y'wz)$ as described in the previous section. Let $L'$ be the resulting list given in the previous section. Let $G'_i = G_i-S(xywz)-S(x'y'wz)$ and $L'_i$ be $L'$ restricted to $G_i'$ for $i=1,2,3,4$. By the definition of $X_{G,L}$, $(G'_1,P,L'_1)$ has no bad edge and no bad, worse, or worst cycle. If it has a bad vertex, then $G$ would have a $(1,1)$-$2$-chord, contradicting Lemma \ref{lem:112chord}. Thus, it has a proper $L'_1$-coloring $\phi_1$. Let $L''_2(u) = \{\phi_1(u)\}$ if $u=x,w,x'$ and $L''_2(u) = L'_2(u)$ otherwise. We observe that $(G'_2, P_2, L''_2)$ is a valid target with $P_2=xwx'$ and thus has a proper $L''_2$-coloring $\phi_2$. Let $L''_3(u) = \{\phi_2(u)\}$ if $u=y,z,y'$ and $L''_3(u)=L'_3(u)$ otherwise. Let $P_3=y^+yzy'y'^+$. By the definition of $X_{G,L}$, $(G'_3, P_3, L''_3)$ has no bad edge, no bad, worse or worst cycle. If it has a bad vertex $v$, then we must have $|L''_3(v)|=4$ and $N(v) \cap P_3=\{y^+,z,y'^+\}$ where $|L(y^+)|=|L(y'^+)|=2$. Then $y^+vy'^+$ is a $(2,2)$-$2$-chord, contradicting Lemma \ref{lem:f-chords}. Thus, $(G'_3, P_3, L''_3)$ is a valid target and has a proper $L''_3$-coloring $\phi_3$. Let $L''_4(u) = \{\phi_2(u)\}$ if $u=x',y'$ and $L''_4(u)=L'_4(u)$ otherwise. By the definition of $X_{G,L}$, $(G'_4, P_4, L''_4)$ is a valid target with  $P_4=x'^+x'y'y'^+$ and thus has a proper $L''_4$-coloring $\phi_4$.Then $\phi_1 \cup \phi_2 \cup \phi_3 \cup \phi_4$ gives a proper $L'$-coloring of $G-P-S(xyzw)-S(x'y'zw)$, a contradiction.   
\end{proof}

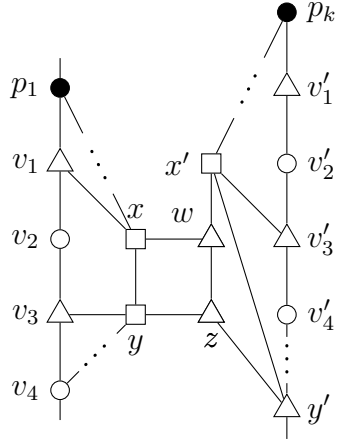
\begin{figure}[!htbp]
\centering

\begin{tikzpicture}
\node[invisnode] at (0,1.5) (a) {};
\node[bdot,label={left:$p_1$}] at (0,1) (p1) {};
\node[triangle,label={left:$v_1$}] at (0,0) (v1) {};
\node[wdot,label={left:$v_2$}] at (0,-1) (v2) {};
\node[triangle,label={left:$v_3$}] at (0,-2) (v3) {};
\node[wdot,label={left:$v_4$}] at (0,-3) (v4) {};
\node[invisnode] at (0,-3.5) (b) {};

\node[square,label={above:$x$}] at (1,-1) (x) {};
\node[square,label={below:$y$}] at (1,-2) (y) {};
\node[triangle,label={135:$w$}] at (2,-1) (w) {};
\node[triangle,label={below:$z$}] at (2,-2) (z) {};

\draw[black] (a)--(p1)--(v1)--(v2)--(v3)--(v4)--(b);
\draw[black] (v1)--(x)--(y)--(v3);
\draw[black] (x)--(w);
\draw[black] (z)--(y);

\draw[black] (p1)--(0.3,0.4);
\draw[black] (x)--(0.7,-0.4);
\draw[black] (y)--(0.75,-2.25);
\draw[black] (v4)--(0.25,-2.75);

\draw[black, line width=1.2pt, line cap=round,
  dash pattern=on 0pt off 4.5\pgflinewidth]
  (0.42,0.16)--(0.64,-0.28);

\draw[black, line width=1.2pt, line cap=round,
  dash pattern=on 0pt off 4\pgflinewidth]
  (0.34,-2.66)--(0.7,-2.3);

\node[bdot,label={right:$p_k$}] at (3,2) (p1) {};
\node[triangle,label={right:$v'_1$}] at (3,1) (v1) {};
\node[wdot,label={right:$v'_2$}] at (3,0) (v2) {};
\node[triangle,label={right:$v'_3$}] at (3,-1) (v3) {};
\node[wdot,label={right:$v'_4$}] at (3,-2) (v4) {};
\node[triangle,label={right:$y'$}] at (3,-3.25) (y) {};

\node[invisnode] at (3,-2.4) (b) {};
\node[invisnode] at (3,-2.8) (c) {};
\node[invisnode] at (3,-3.75) (d) {};

\node[square,label={left:$x'$}] at (2,0) (x) {};
\node[triangle] at (2,-1) (w) {};
\node[triangle,label={below:$z$}] at (2,-2) (z) {};

\draw[black] (p1)--(v1)--(v2)--(v3)--(v4)--(b);
\draw[black] (c)--(y)--(d);
\draw[black] (v3)--(x)--(y);
\draw[black] (x)--(w);
\draw[black] (z)--(y);
\draw[black] (z)--(w);

\draw[black] (p1)--(2.7,1.4);
\draw[black] (x)--(2.3,0.6);

\draw[black, line width=1.2pt, line cap=round,
  dash pattern=on 0pt off 4.5\pgflinewidth]
  (2.58,1.16)--(2.36,0.72);

\draw[black, line width=1.2pt, line cap=round,
  dash pattern=on 0pt off 4\pgflinewidth]
  (3,-2.41)--(3,-2.79);

\end{tikzpicture}

\caption{Example of combinations of two end configurations in which one produces a worse cycle of type II on the left and the other produces a worse cycle of type I on the right.  }
\end{figure}

\begin{lemma}
Two end configurations that both produce a worse cycle of type II cannot be combined.
\end{lemma}

\begin{proof}
Suppose that they can be combined. Assume $xywz$ is produced on the side of $p_1$ and $x'y'wz$ is produced on the side of $p_k$. Let $G_1$ be the subgraph with boundary $B(G_1)=wxv_x B(G) p_1 \cdots p_k B(G) v_{x'} x' w$. Let $G_2$ be the subgraph with boundary $B(G_2)=xwx'y'zyx$. Let $G_3$ be the subgraph with boundary $B(G_3)=v_yyzy'v_{y'}B(G-G_1)v_y$. Now color and then peel off $S(xywz)$ and $S(x'y'wz)$ as described in the previous section. Let $L'$ be the resulting list given in the previous section. Let $G'_i = G_i-S(xywz)-S(x'y'wz)$ and $L'_i$ be $L'$ restricted to $G_i'$ for $i=1,2,3$. By the definition of $X_{G,L}$, $(G'_1,P,L'_1)$ is a valid target and thus has a proper $L'_1$-coloring $\phi_1$. Let $L''_2(u)=\{\phi_1(u)\}$ if $u=x,w,x'$ and $L''_2(u)=L'_2(u)$ otherwise. We observe that $(G'_2, P_2, L''_2)$ is a valid target with $P_2=xwx'$ and thus has a proper $L''_2$-coloring $\phi_2$. Let $L''_3(u)= \{\phi_2(u)\}$ if $u=y,z,y'$ and $L''_3(u)=L'_3(u)$ otherwise. Let $P_3=y^+yzy'y'^+$. By the definition of $X_{G,L}$, $(G'_3, P_3, L''_3)$ has no bad edge, no bad, worse or worst cycle. If it has a bad vertex $v$, then we must have $|L''_3(v)|=4$ and $N(v) \cap P_3=\{y^+,z,y'^+\}$ where $|L(y^+)|=|L(y'^+)|=2$. Then $y^+vy'^+$ is a $(2,2)$-$2$-chord, contradicting Lemma \ref{lem:f-chords}. Thus, $(G'_3, P_3, L''_3)$ is a valid target and has a proper $L''_3$-coloring $\phi_3$. Then $\phi_1 \cup \phi_2 \cup \phi_3 $ gives a proper $L'$-coloring of $G-P-S(xyzw)-S(x'y'zw)$, a contradiction.   
\end{proof}

\begin{figure}[!htbp]
\centering

\begin{tikzpicture}
\node[invisnode] at (0,1.5) (a) {};
\node[bdot,label={left:$p_1$}] at (0,1) (p1) {};
\node[triangle,label={left:$v_1$}] at (0,0) (v1) {};
\node[wdot,label={left:$v_2$}] at (0,-1) (v2) {};
\node[triangle,label={left:$v_3$}] at (0,-2) (v3) {};
\node[wdot,label={left:$v_4$}] at (0,-3) (v4) {};
\node[invisnode] at (0,-3.5) (b) {};

\node[square,label={above:$x$}] at (1,-1) (x) {};
\node[square,label={below:$y$}] at (1,-2) (y) {};
\node[triangle,label={above:$w$}] at (2,-1) (w) {};
\node[triangle,label={below:$z$}] at (2,-2) (z) {};

\draw[black] (a)--(p1)--(v1)--(v2)--(v3)--(v4)--(b);
\draw[black] (v1)--(x)--(y)--(v3);
\draw[black] (x)--(w);
\draw[black] (z)--(y);

\draw[black] (p1)--(0.3,0.4);
\draw[black] (x)--(0.7,-0.4);
\draw[black] (y)--(0.75,-2.25);
\draw[black] (v4)--(0.25,-2.75);

\draw[black, line width=1.2pt, line cap=round,
  dash pattern=on 0pt off 4.5\pgflinewidth]
  (0.42,0.16)--(0.64,-0.28);

\draw[black, line width=1.2pt, line cap=round,
  dash pattern=on 0pt off 4\pgflinewidth]
  (0.34,-2.66)--(0.7,-2.3);

\node[invisnode] at (4,1.5) (a) {};
\node[bdot,label={right:$p_k$}] at (4,1) (p1) {};
\node[triangle,label={right:$v'_1$}] at (4,0) (v1) {};
\node[wdot,label={right:$v'_2$}] at (4,-1) (v2) {};
\node[triangle,label={right:$v'_3$}] at (4,-2) (v3) {};
\node[wdot,label={right:$v'_4$}] at (4,-3) (v4) {};
\node[invisnode] at (4,-3.5) (b) {};

\node[square,label={above:$x'$}] at (3,-1) (x) {};
\node[square,label={below:$y'$}] at (3,-2) (y) {};
\node[triangle] at (2,-1) (w) {};
\node[triangle,label={below:$z$}] at (2,-2) (z) {};

\draw[black] (a)--(p1)--(v1)--(v2)--(v3)--(v4)--(b);
\draw[black] (v2)--(x)--(y)--(v3);
\draw[black] (x)--(w);
\draw[black] (z)--(y);
\draw[black] (z)--(w);

\draw[black] (p1)--(3.7,0.4);
\draw[black] (x)--(3.3,-0.4);
\draw[black] (y)--(3.25,-2.25);
\draw[black] (v4)--(3.75,-2.75);

\draw[black, line width=1.2pt, line cap=round,
  dash pattern=on 0pt off 4.5\pgflinewidth]
  (3.58,0.16)--(3.36,-0.28);

\draw[black, line width=1.2pt, line cap=round,
  dash pattern=on 0pt off 4\pgflinewidth]
  (3.66,-2.66)--(3.3,-2.3);

\end{tikzpicture}

\caption{Example of combinations of two end configurations that both produce worse cycles of type II.  }
\end{figure}
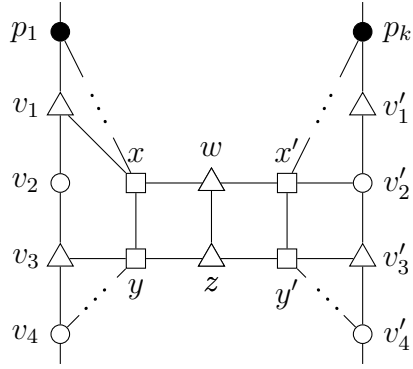

\section{Conclusion} \label{sec5}

As our result provides some evidence to Conjecture~\ref{conj:huzhu}, in particular for some non-empty $G[X]$, it would be interesting to investigate the conjecture when $G[X]$ is some other sparse bipartite graph, such as a path of arbitrary length, an even cycle, a union of an induced matching and an independent set.

It would also be interesting to investigate whether Theorem~\ref{thm:huzhu} or results in this direction can be strengthened through the lens of \emph{weak degeneracy}, a notion recently introduced by Bernshteyn and Lee~\cite{bernshteyn_lee}.
We give a rough definition as follows.

Given a graph $G$ and a function $f: V(G) \to \mathbb{N}$, define two operations:
\begin{itemize}
    \item $\operatorname{Delete}(v)$: remove $v$ from $G$; update $f(w)$ to be $f(w)-1$ if $wv\in E(G)$.
    \item $\operatorname{DelSave}(v,u)$:  remove $v$ from $G$; update $f(w)$ to be $f(w)-1$ if $wv\in E(G)$ and $w\neq u$.
\end{itemize}
Both operations are legal if the updated $f$ remains non-negative. In particular, $\operatorname{DelSave}(v,u)$ is legal if $f(v)>f(u)$ before the operation.
A graph $G$ is \emph{weakly $f$-degenerate} if all vertices of $G$ can be removed by a sequence of legal $\operatorname{Delete}$ and $\operatorname{DelSave}$ operations.
A graph $G$ is weakly $d$-degenerate for an integer $d$ if $G$ is weakly degenerate with respect to the constant $d$ function. 
The \textit{weak degeneracy} of $G$, $\operatorname{wd}(G)$, is the minimum $d$ such that $G$ is weakly $d$-degenerate.
It has been observed that if a graph $G$ is weakly $d$-degenerate, then it follows naturally that $G$ is $(d+1)$-DP colorable, and thus $(d+1)$-choosable and $(d+1)$-colorable.

Bernshteyn, Lee and Smith-Roberge~\cite{bernshteyn_lee_roberge} proved that planar graphs are weakly $4$-degenerate, generalizing the result of Thomassen that planar graphs are $5$-choosable.
Their proof also relies on some strengthening techniques.
Thus, this motivates us to propose the following conjecture.

\begin{conjecture}
     Suppose that for a $C_3$-free planar graph  $G$, $X$ is a vertex subset such that $G[X]$ is independent. Let $f:V(G) \to \mathbb{N}$ be a function such that $f(x)=2$ for every $x\in X$ while $f(v)=3$ for every $v\in V(G)\setminus X$.
    Then $G$ is weakly $f$-degenerate.
\end{conjecture}
    
\medskip{}
{ {\bf Acknowledgment.}
This project was initiated during the Graduate Research Workshop in Combinatorics held at Iowa State University in June 2025. The workshop was supported in part by NSF DMS-2152490, Barbara Jansons Professorship, and Iowa State University. The authors thank Will Hausmann for introducing this problem during the workshop. The authors also thank Hricha Acharya, Gyaneshwar Agrahari, Isaiah Hollars, and Bernard Lidick\'{y} for early helpful discussions.}

\medskip{}
{ {\bf Declaration of AI Use.} During the preparation of this manuscript, the authors used ChatGPT for language editing, proof checking and improving the clarity and readability of the text. The authors reviewed and revised all AI-assisted content and take full responsibility for the manuscript.

\bibliography{refs}
\end{document}